\documentclass[10pt, oneside, dvipsnames]{article}      % use "amsart" instead of "article" for AMSLaTeX format
\usepackage[a4paper, margin=1in]{geometry}                       % See geometry.pdf to learn the layout options. There are lots.
\usepackage{stmaryrd}
\SetSymbolFont{stmry}{bold}{U}{stmry}{m}{n}
\usepackage{titlesec}
\usepackage[parfill]{parskip}           % Activate to begin paragraphs with an empty line rather than an indent
\usepackage{graphicx}
\usepackage{amssymb}
\usepackage{amsmath}
\usepackage{amsthm}
\usepackage{amsfonts}
\usepackage[shortlabels]{enumitem}
\usepackage{xcolor}
\usepackage[noadjust]{cite}
\usepackage{stackengine}
\usepackage{mathrsfs}
\usepackage{stmaryrd}
\usepackage{mathtools}
\usepackage{bm}
\usepackage{bbm}
\usepackage{bold-extra}
\usepackage[normalem]{ulem}
\usepackage{enumitem}
\usepackage{hyperref}
\hypersetup{
    colorlinks=true,
    linkcolor=blue,
    filecolor=magenta,
    urlcolor=cyan,
    citecolor=RubineRed,
    pdftitle={Fluctuations in QUE at the Spectral Edge},
    pdfpagemode=FullScreen,
    }

\allowdisplaybreaks
\newcommand{\defeq}{=} %{\mathrel{\mathop:}=}
\numberwithin{equation}{section}
\usepackage{comment}
\usepackage{cleveref}

\setlist[itemize]{leftmargin=*}
\setlist[enumerate]{leftmargin=*}

\newtheorem{theorem}{Theorem}[section]
\newtheorem{lemma}[theorem]{Lemma}
\newtheorem{proposition}[theorem]{Proposition}

\newtheorem{corollary}[theorem]{Corollary}

\newtheorem{definition}[theorem]{Definition}
\newtheorem{remark}[theorem]{Remark}

\def\I{\mathcal I}

\def\Im{\operatorname{Im}}
\def\Re{\operatorname{Re}}
\newcommand{\Tr}{\operatorname{Tr}}
\newcommand{\matn}{\operatorname{Mat}_N}

\def\e{\bm e}
\def\d{\mathrm d}
\def\m{m_\mathrm{sc}}
\def\q{\bm q}
\def\x{\bm x}
\def\y{\bm y}
\def\E{\mathbb E}
\def\Od{O_\prec(N^{-D})}
\def\R{\mathbb R}
\def\N{\mathbb N}
\def\P{\mathbb P}

\def\leq{\leqslant}
\def\le{\leqslant}
\def\geq{\geqslant}
\def\ge{\geqslant}
\def\err{\mathrm{err}}
\def\even{\mathsf{even}}
\def\odd{\mathsf{odd}}

\def\tilde{\widetilde}
\def\mathbf{\bm}
\def\bar{\overline}
\def\ra{R^{(a)}}
\def\rab{\bar R^{(a)}}

\def\rr{\left(\ra\rab\right)}
\def\rar{\left(\ra A\rab\right)}

\def\phi{\varphi}
\def\epsilon{\varepsilon}
\def\trans{\mathsf{T}}
\def\Tr{\operatorname{Tr}}

\newcommand{\one}{{\boldsymbol{1}}}

\newcommand{\iu}{\mathrm{i}}
\newcommand{\Oparity}[2]{O_{\prec,\mathsf{#1}}(#2)}

\def\eeb{\normalcolor}

\newcommand{\coloneqq}{=}

\title{\scshape{Fluctuations in Quantum Unique Ergodicity at the Spectral Edge}}
\author{L. \textsc{Benigni}\\[-.7ex]\vspace{-0.25cm}\footnotesize{\it{Universit\'e de Montr\'eal}}\\\footnotesize{\it{lucas.benigni@umontreal.ca}}
\and N. \textsc{Chen}\\[-.7ex]\vspace{-0.25cm}\footnotesize{\it{University of Chicago}}\\\footnotesize{\it{nixiachen@uchicago.edu}} 
\and P. \textsc{Lopatto}\\[-.7ex]\vspace{-0.25cm}\footnotesize{\it{Brown University}}\\\footnotesize{\it{\text{patrick\textunderscore lopatto@brown.edu}}}
\and X. \textsc{Xie}\\[-.7ex]\vspace{-0.25cm}\footnotesize{\it{Brown University}}\\\footnotesize{\it{xiaoyu\textunderscore xie@brown.edu}}
}
 \date{}                           % Activate to display a given date or no date
\titleformat{\paragraph}[runin]{\itshape\normalsize}{\theparagraph}{}{}
\titleformat{\subparagraph}[runin]{\itshape\normalsize}{\theparagraph}{0em}{}
\titleformat{\section}[block]{\normalfont\filcenter}{\Large\thesection .}{.7em}{\Large\scshape}
\titleformat{\subsection}[runin]{\normalfont}{\large \bf \thesubsection .}{.5em}{\large\bf}[.]
\titleformat{\subsubsection}[runin]{\normalfont}{\bf \thesubsubsection .}{.5em}{\bf}[.]

\begin{document}
\maketitle

\begin{abstract}
    We study the eigenvector mass distribution of an $N\times N$ Wigner matrix on a set of coordinates $\mathcal{I}$ satisfying $\vert \mathcal{I}\vert \geqslant cN$ for some constant $c>0$. For eigenvectors corresponding to eigenvalues at the spectral edge, we show that the sum of the mass on these coordinates converges to a Gaussian in the $N\to\infty$ limit, after a suitable rescaling and centering. 
    More generally, we establish a central limit theorem for observables of the form $\langle \bm u,A\bm u\rangle$, where $\bm u$ is an edge eigenvector and $A$ is a deterministic matrix with $\Tr(A^2) \ge c N$. 
    The proof proceeds by a two moment matching argument. We directly compare edge eigenvector observables of an arbitrary Wigner matrix to those of a Gaussian matrix, which may be computed explicitly.
\end{abstract}

\section{Introduction}

%In mathematical physics, 
Quantum Unique Ergodicity (QUE) refers to the observation that for the quantization of a %classical
chaotic dynamical system, the eigenstates of the Hamiltonian become uniformly distributed in phase space in the high-energy limit.  This phenomenon has been intensely studied by both physicists and mathematicians, and we refer the reader to \cite{sarnak2012recent} for a survey.
%who have established this principle in special cases, and wide-reaching generalizations.
%, who have probed  its theoretical underpinnings and introduced wide-reaching generalizations. 
Recently, a number of works have investigated QUE, and other closely related principles, in the context of Wigner random matrices \cite{bourgade2017eigenvector,cipolloni2022normal, adhikari2023eigenstate,bao2021equipartition,cipolloni2023gaussian,BenLop22QUE,benigni2022fluctuations, cipolloni2022thermalisation, cipolloni2022rank}. Because such matrices are the simplest class of chaotic quantum Hamiltonians, they form a natural testbed for the study of these ideas. %To contextualize our work, we begin by summarizing some of these previous results. 

We recall that a Wigner matrix is a symmetric matrix $H=\{h_{ij} \}_{1 \le i, j \le N}$ of real random variables with mean zero and variance $N^{-1}$, such that the upper triangular elements $\{h_{ij}\}_{1\le i\leq j \le N}$ are independent. 
%(See \Cref{def:wigner} below for a precise definition.)
The eigenvectors of Wigner matrices are \emph{delocalized}, meaning that their mass is spread approximately uniformly among their entries. The simplest manifestation of delocalization is the high-probability bound 
\begin{equation}\label{intro:deloc}
  \sup_{\alpha \in \llbracket 1, N \rrbracket} N \langle  \bm{q}_\alpha, \bm{u} \rangle^2 \le  N^{ \epsilon},
\end{equation}
which holds for any $\epsilon>0$, eigenvector $\bm{u}$, and orthonormal basis $( \bm{q}_\alpha )_{\alpha =1}^N$, for sufficiently large $N$ \cite{alex2014isotropic}.\footnote{All eigenvectors in this work are normalized so that $\| \bm{u} \|_2 =1$.} %In fact, the same bound holds when $\bm {u}$ is expressed in an arbitrary orthonormal basis. 
 %Delocalization is expected to be a genetic feature of mean-field random matrix models; it to other random matrix ensembles \cite{aggarwal2021goe, bauerschmidt2016local, erdos2013spectral, bourgade2018random}

QUE for Wigner matrices asserts a more refined form of delocalization, concerning the equidistribution of the eigenvector coordinates. % in an arbitrary basis. 
Let $\mathcal I \subset \llbracket 1, N \rrbracket$ be any deterministic subset of indices. Then for any eigenvector $\bm u$, we have the high-probability bound
\begin{equation}\label{intro:que}
\left|\sum_{\alpha\in\I}\left\langle \q_\alpha,\bm u\right\rangle^2-\frac{|\I|}{N}\right| \le %N^{-1+\epsilon} \cdot 
\frac{N^\epsilon\sqrt{|\mathcal{I}|}}{N}.
\end{equation}
%for any $\epsilon >0$ and sufficiently large $N$. %We observe that when $\{ \bm{q}_\alpha \}_{\alpha =1}^N$ is chosen to be the standard basis, \eqref{intro:que} is an improvement over the bound that results from applying \eqref{intro:deloc} to each term in the previous sum.
A weaker version of this claim was first established in \cite{bourgade2017eigenvector}, and the optimal error term stated in \eqref{intro:que} was shown in \cite{BenLop22QUE, cipolloni2022rank}. 

In this article, we consider the fluctuations around the leading-order term identified in \eqref{intro:que}. Based on explicit calculations with Gaussian random matrices \cite[Theorem~2.4]{o2016eigenvectors}, we expect that 
\begin{align}\label{i:CLT}
\begin{split}
    \sqrt{\frac{ N^3}{2|\I|\big(N-|\I|\big)}}\left(\sum_{\alpha \in \I}\left\langle \q_\alpha,\bm u \right\rangle^2 -\frac{|\I|}{N}\right) \rightarrow \mathcal{N}(0,1),
\end{split}
\end{align}
with convergence in distribution, for all Wigner matrices, whenever $| \mathcal I | \gg 1$. We observe that when $|\mathcal \I|\ll N$, the summands act as independent Gaussians, while correlations arising from the condition that $\|\bm{u} \|_2 =1$  are present when $| \mathcal{I} |$ is of order $N$. It has been shown in the recent work 
\cite{cipolloni2022rank} that \eqref{i:CLT} is true for eigenvectors $\bm{u}$ corresponding to eigenvalues in the bulk of the spectrum in the following sense. Label the eigenvalues of $H$ in increasing order, $\lambda_1 \le \lambda_2 \le \dots \le \lambda_N$, and let $i = i_N$ be a sequence of indices such that $\min(i, N-i) > cN$, for some constant $c>0$ and all $N \in \mathbb{N}$. Then \eqref{i:CLT} holds for the eigenvectors $\bm{u}_i$.
%when $(\bm{u}_i)_{i=1}^\infty$ is the sequence of eigenvectors corresponding to $(\lambda_i)_{i=1}^\infty$ \cite{cipolloni2022rank}.%; see also \cite{cipolloni2022normal} for earlier results in the bulk.

%we elaborate on this below, in Section~\ref{s:background}.

At the edge of the spectrum, previous results are less complete. In \cite{BenLop22QUE}, it was shown that \eqref{i:CLT} holds for \emph{any} eigenvector $\bm{u}$, if $N^\tau \le | \mathcal I | \le N^{1-\tau}$ for some $\tau > 0$. However, this leaves open the case with $| \mathcal {I} |$ proportional to $N$, where correlations between eigenvector entries arise. This case is of particular interest since it  parallels the original QUE conjecture, which concerned the mass of eigenstates on subsets containing a constant fraction of phase space. In this article, we address this regime and show that \eqref{i:CLT} holds for any $\mathcal I$ such that $|\mathcal I | \ge  N^{1-c}$ and any eigenvector $\bm{u}_i$ such that $\min(i, N-i) \le N^{1-\tau}$, where $\tau>0$ is an arbitrary constant, and $c(\tau)>0$ is a small constant depending on $\tau$. This completes the characterization of fluctuations in QUE for Wigner matrices at the spectral edge. %\footnote{\red{Need to rewrite the last three paragraphs for the more general traceless $A$.}}

%Assuming this claim, our result, in conjunction with the previous results of \cite{cipolloni2022rank} and \cite{BenLop22QUE} mentioned above, completes the analysis of fluctuations in QUE for Wigner matrices throughout the entire spectrum.

While our primary interest is the quantum unique ergodicity observable in \eqref{i:CLT}, our main result goes further and establishes a central limit theorem for observables of the form $\langle \bm u,A\bm u\rangle$, where $\bm u$ is an edge eigenvector and $A$ satisfies $\Tr(A) = 0$ and $\Tr (A^2) \ge N^{1- c}$. The statement \eqref{i:CLT} follows from this more general claim after taking $A$ to be a projection onto the set $\{\bm q_{\alpha}\}_{\alpha \in \mathcal I}$.

\subsection{Main Results}
We first define Wigner matrices. 
\begin{definition}[Wigner matrix]\label{def:wigner}
A Wigner matrix $H=H_N = \{ h_{ij}\}_{1\le i,j \le N}$ is a real symmetric or complex Hermitian $N \times N$ matrix whose upper triangular elements $\left\{h_{i j}\right\}_{1\le i \leqslant j \le N}$ are independent random variables that satisfy 
\begin{align}\label{hij}
\begin{split}
    \E [h_{ij}]=0,\qquad \E \big[|h_{ij}|^2\big]=\frac{1+\delta_{ij}}{N}\,.
\end{split}
\end{align}
In the complex case, we additionally suppose that $\E[h^2_{ij}] = 0$. 
Further, we suppose that the normalized entries have finite moments, uniformly in $N$, $i$, and $j$, in the sense that for all $p \in \mathbb{N}$ there exists a constant $\mu_p$ such that
\begin{align}\label{finite-moment}
\begin{split}
    \mathbb{E}\left[\left|\sqrt{N} h_{i j}\right|^p \right]\leqslant \mu_p 
\end{split}
\end{align}
for all $N, i$, and $j$.
\end{definition}
\begin{remark}
In \Cref{def:wigner}, we assumed that the diagonal entries have variance $ 2N^{-1}$. This assumption is made for convenience, and our results still hold if the diagonal variances are replaced by any constant multiple of $N^{-1}$. 
More precisely, the second condition in \eqref{hij} could be relaxed to require only that 
\begin{equation} \mathbb{E}\big[|h_{i j}|^2 \big] = \frac{1 + (d-1) \delta_{ij} }{N}
\end{equation}
for some constant $d > 0$. The modifications to the proofs in this case are straightforward, and we omit them for brevity.
\end{remark}

Our main theorem is the following central limit theorem for Wigner matrix eigenvectors. A matrix $A$ is said to be traceless if $\Tr(A)=0$.

\begin{theorem}[Central Limit Theorem]\label{thm:CLT}
Let $H$ be a Wigner matrix and fix $\tau\in(0,1)$. Then there exists $\delta=\delta(\tau)\in(0,1)$ such that the following holds.
Let $A=A_N\in\mathbb R^{N\times N}$ be a deterministic sequence of traceless matrices such that $A =A^*$, $\|A\|\leq1$, and $\Tr(A^2)\geq N^{1-\delta}$.
% Let $\I=\I_N$ be a deterministic sequence of subsets such that $ \vert \mathcal{I}\vert\in \llbracket  N^{1-\delta} , N \rrbracket$, %\textcolor{blue}{(shouldn't it be $\mathcal{I}$ such that $\vert \mathcal{I}\vert\in \llbracket  N^{1-\delta} , N \rrbracket$ ?)}
Let $\ell=\ell_N \in \llbracket 1, N^{1-\tau} \rrbracket\cup\llbracket N-N^{1-\tau},N\rrbracket$ be a deterministic sequence of indices, and let $\bm{u}=\bm{u}_{\ell}^{(N)}=(u(1),\ldots,u(N))$ be the corresponding sequence of $\ell^2$-normalized eigenvectors of $H$. 
%Finally, 
% Let $\left(\bm{q}_\alpha\right)_{\alpha \in I}=\left(\bm{q}_\alpha^{(N)}\right)_{\alpha \in I}$ be a deterministic sequence of sets of orthogonal vectors in $\mathbb{S}^{N-1}$.
Then
\begin{align}\label{CLT}
\begin{split}
    \sqrt{\frac{\beta N^2}{2\Tr(A^2)}}\left\langle \bm u,A\bm u\right\rangle\rightarrow \mathcal N(0,1),
\end{split}
\end{align}
with convergence in distribution. Here $\mathcal{N}(0,1)$ is a standard real Gaussian random variable; we take $\beta=1$ if $H$ is real symmetric, or $\beta=2$ if it is complex Hermitian.
\end{theorem}

As noted below in Remark~\ref{r:momcvg}, the convergence in distribution can be improved to convergence in moments. 
%\begin{remark}
%The previous theorem also holds if $\delta = N^{-c}$, where $c>0$ is a fixed, small constant that may depend on $\tau$ and the $\mu_p$. 
%\end{remark}
% Because the Gaussian distribution is determined by its moments, the above theorem immediately implies that \eqref{CLT} also holds for convergence in distribution \cite[Theorem 30.2]{billingsley2008probability}. 

%For brevity, we focus on the case of real symmetric Wigner matrices in our proof of Theorem~\ref{thm:CLT}. The details for the complex Hermitian case are nearly identical, and hence omitted.
\subsection{Related Works}\label{s:background}

Delocalization estimates have received significant attention from the random matrix community over the past decade and a half. The estimate \eqref{intro:deloc} has a long history, and increasingly strong versions of this statement were proved in \cite{erdos2009semicircle, erdos2009local, erdos2010wegner, tao2011random, tao2010random,erdos2012bulk, erdos2012rigidity,vu2015random, aggarwal2019bulk, gotze2018local, gotze2019local}. The optimal high-probability upper bound of  $\sqrt{(2+\epsilon)\log N}$ was recently established in \cite{BL22LInfinity}. Going beyond Wigner matrices, similar estimates have been shown for band matrices \cite{xu2022bulk,bourgade2018random,erdos2013local}, heavy-tailed random matrices \cite{bordenave2013localization,bordenave2017delocalization,aggarwal2021goe,aggarwal2022mobility}, and adjacency matrices of sparse random graphs \cite{bauerschmidt2016local,erdos2013spectral}. 
Fluctuations of individual eigenvector entries of Wigner matrices were first studied in \cite{bourgade2017eigenvector}, where they were shown to be Gaussian (see also \cite[Corollary B.18]{BL22LInfinity} and \cite{benigni2021fermionic}). Arbitrary finite collections of bulk eigenvector entries were shown to be jointly Gaussian in \cite{marcinek2022high}. Fluctuations for eigenvector entries of non-Hermitian matrices were studied in \cite{dubova2024gaussian}.

As noted above, the first QUE estimate for Wigner matrices was shown in \cite{bourgade2017eigenvector}. Estimates of the form \eqref{intro:que} have also been shown for deformed Wigner matrices \cite{benigni2020eigenvectors}, band matrices \cite{xu2022bulk,bourgade2018random}, sparse random matrices \cite{bauerschmidt2016local,bourgade2017sparse,anantharaman2019quantum,anantharaman2015quantum,anantharaman2008entropy}, and heavy-tailed random matrices \cite{aggarwal2021eigenvector}. Further, a more general version of \eqref{intro:que}, known as \emph{eigenvector thermalization}, has appeared recently (motivated by the phenomena surveyed in \cite{srednicki1994chaos,d2016quantum,deutsch1991quantum}). Let $A$ be a deterministic $N\times N$ matrix such that $\| A \| \le 1$, where $\| A \|$ denotes the spectral norm of $A$. Then for any eigenvector $\bm{u}$ of a Wigner matrix, we have the high-probability bound 
\begin{equation}\label{ETH}
\left| \langle \bm{u}, A \bm{u} \rangle -  \frac{1}{N} \Tr A \right| \le \frac{N^\epsilon}{\sqrt{N}},
\end{equation}
for any $\epsilon >0$ and sufficiently large $N$ \cite{cipolloni2021eigenstate}. Subsequently, fluctuations around the leading order term in \eqref{ETH} were identified in \cite{cipolloni2022normal}, and an optimal-order error term was established in \cite{cipolloni2022rank}. A generalization of \eqref{ETH} to generalized Wigner matrices is provided in \cite{riabov2024eigenstate}.

Our expect that our proof strategy extends straightforwardly to yield the joint fluctuations for any finite set of edge eigenvectors, i.e. 
\begin{equation}
    \sqrt{\frac{N^3}{2|\I|(N-|\I|)}}\left(\sum_{\alpha \in \I}\left\langle \q_\alpha,\bm u_{\ell_1} \right\rangle^2 -\frac{|\I|}{N},\cdots,\sum_{\alpha \in \I}\left\langle \q_\alpha,\bm u_{\ell_k} \right\rangle^2 -\frac{|\I|}{N}\right)\rightarrow (Z_1,\cdots,Z_k),
\end{equation}
with convergence in distribution, where $\ell_1<\cdots<\ell_k\leq N^{1-\tau}$ and $Z_1,\ldots,Z_k$ are independent Gaussian random variables with zero mean and unit variance.
We briefly remark on this extension in \Cref{a:joint}. 

Our work does not address the intermediate spectral regime where $ i/N$ tends to $0$ slower than any negative power of $N$. We expect that this regime can be handled by a straightforward (but tedious) modification of the arguments in \cite{cipolloni2022rank}. %, and in fact that the results there hold for all $i \in \llbracket  N^{1-c}, N - N^{1-c}\rrbracket$ if $c>0$ is chosen sufficiently small. 
However, we leave this as an open question for future work.

\subsection{Proof Strategy}
Previous works determining the fluctuations in QUE have all followed the dynamical approach to random matrix universality (surveyed in \cite{erdHos2017dynamical}). This approach uses the following three steps.

\begin{enumerate} 

\item Establish various \emph{a priori} estimates on the eigenvalues and eigenvectors of Wigner matrices, such as \eqref{intro:deloc}, which are used as input in the following steps. 

\item Determine the fluctuations in QUE for random matrices of the form $H + \sqrt{ t} W$, where $H$ is an arbitrary Wigner matrix, $W$ is a Gaussian Wigner matrix, and $t \approx N^{-c}$ for some $c>0$. This is done by recognizing $H + \sqrt{ t} W$ as the evolution of a matrix Brownian motion with initial data $H$ until time $t$. Under this stochastic process, the moments of the QUE observable in \eqref{i:CLT} evolve according to a parabolic differential equation known as the \emph{eigenvector moment flow}. A detailed analysis of this evolution shows that these moment observables converge to their equilibrium states, the Gaussian moments, after time $t$. This convergence in moments establishes \eqref{i:CLT} for the matrix $H + \sqrt{ t} W$.

\item Transfer the conclusion from the previous step to all Wigner matrices. Given an arbitrary Wigner matrix $H$, there exists a Wigner matrix $H'$ such that the first three moments of $H$ and $H' + \sqrt{t} W$ match exactly, and the difference of the fourth moments is order $t$. By a moment matching argument similar to the one used in Lindeberg's proof of the central limit theorem (see \cite[Section 11]{benaych2016lectures}), one can show that this moment condition is enough to establish that $H$ has the same fluctuations in QUE as $H' + \sqrt{t} W$, completing the proof. 
\end{enumerate}

Thus far, obstacles related to Step 2 of the dynamical approach have blocked a proof of \eqref{i:CLT} for $|\mathcal I|$ proportional to $N$ at the spectral edge. The analysis of the eigenvector moment flow in \cite{BenLop22QUE} was applicable throughout the entire spectrum, but is only effective for index sets $\mathcal I$ with cardinality $| \mathcal I | \ll N$.  The works \cite{cipolloni2022normal,cipolloni2022rank} analyzed a variation of the eigenvector moment flow introduced in \cite{marcinek2022high}, called the \emph{colored eigenvector moment flow}, which allow them to access $\mathcal I$ with $|\mathcal I|$ proportional to $N$. However, these works depend on an intricate analysis of the colored evolution dynamics presented in \cite{marcinek2022high}, which was only given in the bulk. In principle, such an analysis could also be carried out at the edge. However, given the length and sophistication of \cite{marcinek2022high}, and additional complications that arise at the edge due to the curvature of the spectral density (the semicircle law, given in \eqref{e:semicircle} below), this extension seems far from straightforward. 

Instead, we adopt an argument that has no dynamical component, and uses only moment matching. 
We draw inspiration from \cite{KnoYin13} and \cite{BloKnoYauYin16}, which characterize the joint eigenvector--eigenvalue distribution of Wigner matrices at the edge (see \cite[Remark 8.5]{BloKnoYauYin16}). Specifically, the authors show that given any finite collection of edge eigenvalues and entries of the corresponding eigenvectors, their joint distribution is asymptotically the same as the one for a Gaussian ensemble. In particular, any finite collection of such eigenvector entries is asymptotically distributed as independent Gaussians. The proof proceeds by a ``two-moment matching'' argument, which shows that two random matrix ensembles whose entries are independent, centered, and have the same variance matrix also have the same eigenvector--eigenvalue statistics at the edge. As an immediate consequence, the edge statistics of any Wigner matrix match those of a Gaussian Wigner matrix, which may be computed explicitly. The decay of spectral density at the edge is crucial to the proof, and renders it inapplicable to the bulk, where the full dynamical approach is necessary. 

We now give an overview of our proof. Let $H$ be a Wigner matrix. Our first step is to regularize the QUE observable in \eqref{i:CLT}. Let $f_n(H)$ denote a smooth approximation to the $n$-th moment of the observable on the left side of \eqref{i:CLT}, 
corresponding to some eigenvector $\bm u$,
%and associated eigenvalue $\lambda_\ell$, 
which is differentiable in the matrix entries.\footnote{The observable itself is not differentiable in the matrix entries, which necessitates the smoothing.} 
We wish to proceed as follows. Fix indices $a,b\in\llbracket 1,N\rrbracket$.
Let $W$ denote the matrix such that $w_{ij} = h_{ij}$ for all $i,j \in \llbracket 1, N \rrbracket$ such that $(i,j)\notin\{ (a,b), (b,a)\}$, and such that $w_{ab} = w_{ba} = g$, where $g$ is a Gaussian variable with mean $0$ and variance $(1 + \delta_{ab})N^{-1}$. We observe that the first two moments of $g$ match those of $h_{ab}$. 
Finally, let $Q$ denote the matrix such that $q_{ij} = h_{ij}$ for all $i,j \in \llbracket 1, N \rrbracket$ such that  $(i,j)\notin\{ (a,b), (b,a)\}$, where $q_{ab} = q_{ba} = 0$. Then by Taylor expansion, we have 
\begin{equation*}
f_n(H) = f_n(Q) + \partial_{ab}f_n (Q)h_{ab} + \frac{1}{2}\partial_{ab}^2f_n (Q)h_{ab}^2 + \frac{1}{6}\partial_{ab}^3f_n (Q)h_{ab}^3 + \frac{1}{24}\partial_{ab}^4f_n(Q)h_{ab}^4 + X_H,
\end{equation*}
where $X_H$ is the error term in expansion. Subtracting the analogous expansion for $f_n(W)$, and taking expectations, we obtain
\begin{align*}
\E\big[f_n(H) -& f_n(W)\big] = \E\big[\partial_{ab}f_n (Q) (h_{ab} - w_{ab})\big] + \frac{1}{2}\E\big[\partial_{ab}^2f_n (Q)( h_{ab}^2 - w_{ab}^2)\big]\\
 +& \frac{1}{6}\E\big[\partial_{ab}^3f_n (Q)( h_{ab}^3 - w_{ab}^3)\big] +\frac{1}{24}\E\left[\partial_{ab}^4f_n (Q)(h_{ab}^4-w_{ab}^4)\right] + \E\big[(X_H - X_W)\big]
\\ =& \frac{1}{6}\E\big[\partial_{ab}^3f_n (Q)\big] \E\big[h_{ab}^3 - w_{ab}^3\big] +\frac{1}{24}\E\left[\partial_{ab}^4f_n (Q)\right]\E\left[h_{ab}^4-w_{ab}^4\right] + \E\big[(X_H - X_W)\big].
\end{align*}
In the previous equation, we observed that $Q$ is independent from $h_{ab}$ and $w_{ab}$, and used \[\E\big[\partial_{ab}f_n(Q)(h_{ab} - w_{ab})\big] = 0,\] which follows from $\E[h_{ab}] = \E[w_{ab}]$. We also used the analogous reasoning for the second-moment term. 

We consider the third-moment and fourth-moment terms, and neglect the error term for now. From the definition of a Wigner matrix, we have $\E\big[h_{ab}^3 - w_{ab}^3\big] = O(N^{-3/2})$ and $\E\left[h_{ab}^4-w_{ab}^4\right]=O(N^{-2})$. If we had the estimates 
\begin{align}\label{eqn:desired-estimates}
\begin{split}
    \E\big[\partial_{ab}^3f_n (Q)\big] \ll N^{-1/2},\quad \E\big[\partial_{ab}^4f_n(Q)\big]\ll1,
\end{split}
\end{align}
then we could conclude that $\E\big[f_n(H) - f_n(W)\big] \ll N^{-2}$. This estimates the error accrued when exchanging one entry of $W$ for a Gaussian. Since we need to exchange $O(N^2)$ entries, the total error will be $o(1)$, and the moments $\E[f_n(H)]$ will match those of a Gaussian random matrix in the large $N$ limit. Because \eqref{i:CLT} can directly be established for Gaussian matrices, this would complete the proof. 

The crux of the problem is then to produce a suitable regularization $f_n$ and demonstrate that its derivatives decay suitably in $N$ near the edge of the spectrum. While regularizations of \eqref{i:CLT} have appeared before, the necessary decay at the edge has not been established. For example, the regularized QUE observable in \cite{BenLop22QUE} was only shown to satisfy $\E\big[\partial_{ab}^3f_n(Q)\big] = O(1)$. %Instead, we adopt the regularization procedure used for individual eigenvector entries in \cite{KnoYin13,BloKnoYauYin16}. 
To illustrate how our regularization works, and how we achieve the additional gain at the edge, we begin by describing the regularization of a single eigenvector entry, as accomplished in \cite{KnoYin13,BloKnoYauYin16}.

Let $\bm{u}_\ell$ be an eigenvector with associated eigenvalue $\lambda_\ell$, and let $\eta >0$ be chosen so that $\eta\ll\Delta_\ell$, where $\Delta_\ell$ is the typical size of the eigenvalue gap $\lambda_{\ell +1} - \lambda_{\ell}$. Recall the Poisson kernel identity
\begin{equation*}\frac{\eta}{\pi} \int_{\mathbb R} \frac{\d E}{(E - \lambda_\ell )^2 + \eta^2} = 1.
\end{equation*}
Fix $k\in\llbracket1,N\rrbracket$. We have the high probability estimate 
\begin{equation}\label{introreg}
\bm u_\ell(k)^2 = \frac{\eta}{\pi} \int_{\mathbb R} \frac{ \bm u_\ell(k)^2 \, \d E }{(E - \lambda_\ell)^2 + \eta^2 } \approx \frac{\eta}{\pi} \int_{I} \frac{ \bm u_\ell(k)^2 \, \d E }{(E - \lambda_\ell)^2 + \eta^2 } \approx \frac{\eta}{\pi} \int_{I}  \sum_{i=1}^N \frac{ \bm u_i(k)^2 \, \d E }{(E - \lambda_i)^2 + \eta^2 },
\end{equation}
where $I$ is any interval centered at $\lambda_\ell$ such that $\eta\ll| I |\ll \Delta_{\ell} $, and we used $$ \max( \lambda_{\ell +1} - \lambda_{\ell},\lambda_{\ell} - \lambda_{\ell-1})\gg \eta$$ to neglect the terms with $i\neq \ell$ in the sum. Letting $G = (H - E - \mathrm{i} \eta)^{-1}$ denote the resolvent of $H$, the spectral theorem implies that 
\begin{equation}\label{i:resolventreg}
\frac{\eta}{\pi} \int_{I}  \sum_{i=1}^N \frac{ \bm u_i(k)^2 \, \d E }{(E - \lambda_i)^2 + \eta^2 } =\frac{\eta}{\pi} \int_I (G\bar G)_{kk}\, \d E.
\end{equation}
For illustrative purpose, let us treat $I$ as a deterministic interval. Then, we see that 
\begin{align*}
\begin{split}
   f_n(H) \defeq  \left(\frac{N\eta}{\pi} \int_I (G\bar G)_{kk}\, \d E\right)^{n} \approx \left(\sqrt{N}\bm u_\ell(k)\right)^{2n},
\end{split}
\end{align*}
 is a smooth function of the matrix entries. We multiplied $\bm u_{\ell}(k)$ by $\sqrt{N}$ to make it an $O(1)$ quantity.

Letting $R=(Q-E-\iu\eta)^{-1}$ denote the resolvent of $Q$ and taking derivatives, one readily finds that\footnotemark
\footnotetext{There are multiple terms as a result of applying product rule, so we focus on one representative term for clarity.}
\begin{equation}\label{i:derivative}
    \partial_{ab}^m f_n(Q)\approx n(n-1)\cdots (n-m+1)(\sqrt{N}\bm u_\ell(k))^{2n-m} \int_{I}N\tilde P_m\,\d E+\cdots,
\end{equation}
where $\tilde P_m$ is a polynomial with constant number of terms, and  each term consists of one $(R\bar R)_{**}$ factor and $m$ $R_{**}$'s or $\bar R_{**}$'s.
Here $*\in\{a,b,k\}$ and different appearances of $*$ may take different values.

So far, we have not used the fact that $\lambda_\ell$ is an edge eigenvalue. The crucial use of this fact is that we are able to choose the spectral parameter $\eta$ such that $1\ll \Delta_\ell\eta^{-1}\ll (N\eta)^{1/4}$. Indeed, if $\lambda_\ell$ were in the bulk, we would have $\Delta_\ell=O\left(N^{-1}\right)$, and such choice would not be possible. Combined with the standard local law for resolvents of Wigner matrices (see \eqref{eqn:2-resolvents} below), for any $\epsilon>0$, we have
\begin{align}\label{i:two-res-bound}
\begin{split}
    (R\bar R)_{ij}\leq N^{\epsilon}(N\eta)^{-2}
\end{split}
\end{align}
with high probability for all $i,j \in \llbracket 1, N \rrbracket$. Therefore, we have the bound
\begin{align}\label{i:p-integral}
\begin{split}
    \int_IN\tilde P_m \leq |I|N^{1+\epsilon} (N\eta)^{-2}\leq (N\eta)^{-1/2}\ll 1,
\end{split}
\end{align}
by the choice of $I$ and $\eta$. Inserting this into \eqref{i:derivative}, we have
\begin{align}
\begin{split}
    \big|\partial_{ab}^mf_n(Q)\big|\ll 1,
\end{split}
\end{align}
with high probability. Upon taking expectation, this establishes the second inequality in \eqref{eqn:desired-estimates}.

For the first inequality in \eqref{eqn:desired-estimates}, we need to exploit an additional cancellation that is introduced when taking expectation. To uncover this cancellation, we use the \textit{polynomialization} technique, which first appeared systematically in \cite{erdHos2013averaging}, and was further developed in \cite{yin2014local,BloKnoYauYin16}. The main idea is to write $\tilde P_m$ in the form
\begin{equation}
    \tilde P_m \approx \sum_{i_1,\ldots,i_d\neq a}\tilde P_{m,i_1,\ldots,i_d}^{(a)}h_{ai_1}\cdots h_{ai_d},
\end{equation}
where $\tilde P^{(a)}_{m,i_1,\ldots,i_d}$ is independent of $a$-th row and column of Q. When $d$ is an odd number, we have
\begin{align}\label{i:additional-gain}
\begin{split}
    \left|\E\big[\tilde P_m\big] \right|\lesssim N^{-1/2}\sqrt{\E\big[|\tilde P_m|^2\big]}.
\end{split}
\end{align}
To see why \eqref{i:additional-gain} is true, consider a simple case, where $\mathcal P=\sum_{i_1,i_2,i_3\neq a}h_{ai_1}h_{ai_2}h_{ai_3}$. Taking expectation forces $i_1,i_2,i_3$ to coincide, and therefore
\begin{align*}
\begin{split}
    \big|\E [\mathcal P]\big|=\left|\E\left[\sum_i h_{ai}^3\right]\right|\lesssim N^{-1/2}=N^{-1/2}\sqrt{\E\left[\sum_{i_1,i_2,i_3}h_{ai_1}^2h_{ai_2}^2h_{ai_3}^2\right]}\leq N^{-1/2}\sqrt{\E\left[|\mathcal P|^2\right]}.
\end{split}
\end{align*}
More generally, it can be shown that $\tilde P_m$ can be approximated by an odd polynomial as long as $m$ is odd and $a,b,k$ are distinct indices. Combining \eqref{i:additional-gain} with \eqref{i:derivative} and \eqref{i:p-integral}, we obtain the first inequality in \eqref{eqn:desired-estimates} for all indices $a,b$, except for the $O(n)$ pairs such that $a=b,a=k$ or $b=k$, which is sufficient for our purpose.

Regularizing the QUE observable in \eqref{i:CLT} can be accomplished similarly by replacing each term $\langle \bm q_\alpha, \bm u_\ell \rangle^2$ appearing there by the regularization given in \eqref{i:resolventreg}. However, to appropriately control the size of the resulting moments, we need to detect additional cancellations in the sum; it is not enough to bound each term individually. For this, we use \emph{multi-resolvent local laws}, which bound the quantities $(GA\overline{G})_{cd}$ and $(GA \overline{G} G)_{cd}$ for any choice of $c,d \in \llbracket 1 , N \rrbracket$ and  deterministic $N\times N$ matrix $A$ such that $\| A \| \le 1$ and $\Tr A = 0$; see Lemma~\ref{lem:three-resolvent-local-law} below. While such laws have been established previously \cite{cipolloni2022rank,CipErdSch22optimal,cipolloni2021eigenstate}, we prove a new version with improved error terms at the spectral edge. These improved estimates allow us to obtain sharper bounds in the moment matching argument, which are necessary to complete the proof. 

In summary, our argument involves three interlocking technical components: eigenvector regularization at the edge, two-moment matching, and multi-resolvent local laws. While the first two elements have been applied previously to characterize eigenvalue statistics at the edge \cite{KnoYin13,BloKnoYauYin16}, we deal here with more general statistics $\langle \bm u, A \bm u \rangle$, which present new challenges, including a more complicated set of error terms that must be bounded using the polynomialization technique to enforce the appropriate regularization. As mentioned previously, simpler regularization schemes, such as the one considered in \cite{BenLop22QUE}, do not seem to suffice.

To implement our argument, our new multi-resolvent local law (Lemma~\ref{lem:three-resolvent-local-law}) is a crucial technical input; previous local laws do not provide the necessary bounds at the spectral edge. We remark that after the first version of this paper appeared, a different multi-resolvent local law at the edge was proved in \cite[Theorem 2.4]{cipolloni2023eigenstate}, which implies a strong form of eigenvector thermalization. However, this result does not seem to suffice for our purpose, since it does not reproduce the bounds in Lemma~\ref{lem:three-resolvent-local-law}. %The reason for this is that Lemma \ref{lem:three-resolvent-local-law} includes the additional hypothesis that $\| A \| \le 1$, which permits more refined bounds for certain observables.

It is natural to ask whether \Cref{thm:CLT} extends to matrices $A$ such that $\Tr (A^2) \gg 1$. While this broader conclusion is likely true, there appears to be an intrinsic difficulty in extending our two-moment matching approach to prove it. 
We explain this point in Remark~\ref{r:theproblem} below.

\subsection{Outline}
Section~\ref{s:preliminaries} introduces our notational conventions and states several preliminary results from previous works that are used throughout this paper. Section~\ref{sec:regularized-observables} defines the smoothed QUE observables needed for our moment matching argument. Section~\ref{s:conclusion} proves our main result, Theorem~\ref{thm:CLT}, assuming two preliminary lemmas, Lemma~\ref{lem:three-resolvent-local-law}, and Lemma~\ref{lem:third-moment-terms}. Lemma~\ref{lem:third-moment-terms} is proved in Section~\ref{sec:poly}. \Cref{s:goeclt} establishes the analogue of our main result for Gaussian random matrices, and 
Appendix~\ref{sec:local-law} contains the proof of Lemma~\ref{lem:three-resolvent-local-law}. 
We comment on how to extend our main result to the joint distribution of edge eigenvectors in \Cref{a:joint}.

\subsection{Acknowledgments}

The authors thank Antti Knowles for several helpful conversations and Giorgio Cipolloni for many useful comments on \cite{{cipolloni2023eigenstate}}. 
They also grateful to the anonymous referees for the detailed feedback, which substantially improved this article. 
Patrick Lopatto was supported by the NSF postdoctoral
fellowship DMS-2202891. 
Xiaoyu Xie was supported by NSF grants DMS-1954351 and DMS-2246838. 
Lucas Benigni and Patrick Lopatto also wish to acknowledge the NSF
grant DMS-1928930. This grant supported their participation in the Fall 2021 semester program at
MSRI in Berkeley, California titled, ``Universality and Integrability in Random Matrix Theory and
Interacting Particle Systems,'' where this project began.
\section{Preliminaries}\label{s:preliminaries}

%Throughout this work, we suppress the dependence of various constants in our results on the constants  $\mu_p$ in
%Definition~\ref{def:wigner}. This dependence does not affect our arguments in any substantial way.  
%\subsection{Outline of the paper}

\subsection{Conventions}\label{s:conventions}
For the remainder of the paper, we fix an arbitrary constant $\tau\in(0,1)$, a sequence of deterministic traceless matrices $A=A_N\in\mathbb R^{N\times N}$ such that $A= A^*$, $\|A\|\leq 1$,  and $\Tr(A^2)\geq N^{1-\delta} $, where $\delta=\delta(\tau)>0$ will be defined in \Cref{para}, and a sequence of deterministic indices $\ell=\ell_N \in \llbracket 1, N^{1-\tau} \rrbracket\cup\llbracket N-N^{1-\tau},N\rrbracket$.  Without loss of generality, we always assume that $\ell \in \llbracket 1, N^{1-\tau} \rrbracket$.  We also fix a sequence of positive reals $(\mu_p)_{p=1}^\infty$. We assume that all Wigner matrices mentioned below satisfy Definition~\ref{def:wigner} with this sequence of constants. Our claims hold for any choices of $\tau$, $A$, $\ell$, and $(\mu_p)_{p=1}^\infty$.

We also define the spectral domain
\begin{equation}\label{eqn:spectral-domain}
\bm S= \bm S( N)=\left\{z=E+\mathrm{i} \eta \in \mathbb{C}:|E| \leqslant \frac{10}{\tau}, N^{-1+\tau/10} \leqslant |\eta| \leqslant \frac{10}{\tau} \right\}.
\end{equation}
%In the following theorem, and all other stochastic domination statements in this work involving elements of $\bm S_\omega$, the numbers $N_0$ in Definition~\ref{def:stochasticDomination} may additionally depend on $\omega$. \red{(this $\omega$ will ultimately depend on $\tau$ as in Section \ref{s:conventions})}

%Finally, we fix $\omega=\tau/10$ in the definition of $\bm S_\omega$ \eqref{eqn:spectral-domain}, and use the notation $\bm S$ for this choice of spectral domain.

Throughout this article, we typically suppress the dependence of various constants in our results on the choices of $\tau$ and $(\mu_p)_{p=1}^\infty$. These dependencies do not affect our arguments in any substantial way. Additionally, we focus on the case of real symmetric Wigner matrices in our proof of Theorem~\ref{thm:CLT}. The details for the complex Hermitian case are nearly identical, and hence omitted. 

\subsection{Notations and Definitions}

Let $\operatorname{Mat}_N$ be the set of $N \times N$ real symmetric matrices and $\{\e_i\}_{i=1}^N$ be the standard basis of $\R^N$. Let $\| M \|$ denote the spectral norm of $M$. We index the eigenvalues of matrices $M \in \operatorname{Mat}_N$ in increasing order, and denote them $\lambda_1 \leqslant \lambda_2 \leqslant \ldots \leqslant \lambda_N$. For $z \in \mathbb C\backslash\R$, the resolvent of $M \in \operatorname{Mat}_N$ is given by $G(z)=(M-z)^{-1}$. %, where $\mathrm{Id}= \mathrm{Id}_N$ is the identity matrix. 
The Stieltjes transform of $M$ is
$$
m_N(z)= \frac{1}{N} \operatorname{Tr} G(z)=\frac{1}{N} \sum_i \frac{1}{\lambda_i-z} .
$$
The resolvent has the spectral decomposition
$$
G(z)=\sum_{i=1}^N \frac{\bm{u}_i \bm{u}_i^*}{\lambda_i-z},
$$
where we let $\bm{u}_i$ denote the eigenvector corresponding to the eigenvalue $\lambda_i$ of $M$ such that $\left\|\bm{u}_i\right\|_2=1$. We fix the sign of $\bm{u}_i$ arbitrarily by demanding that $\bm{u}_i(1) \geqslant 0$. For deterministic vectors $\bm x,\bm y$, we abbreviate $\langle\bm x,M\bm y\rangle$ by $M_{\bm x\bm y}$ and we abbreviate $M_{\bm x\bm y}$ further by $M_{i\bm y}$ or $M_{\bm xj}$ if $\bm x=\bm e_i$ or $\bm y=\bm e_j$ respectively.

The semicircle density and its Stieltjes transform are
\begin{equation}\label{e:semicircle}
\rho_{\mathrm{sc}}(E)=\frac{\sqrt{(4-E^2)_{+}}}{2 \pi}\, \mathrm{d} E, \qquad \m(z)=\int_{\mathbb{R}} \frac{\rho_{\mathrm{sc}}(x)}{x-z} \, \mathrm{d} x = \frac{-z + \sqrt{z^2 -4}}{2},
\end{equation}
for $E \in \mathbb{R}$ and $z \in \mathbb{C}\backslash \mathbb{R}$. The square root in $\sqrt{z^2 -4}$ is defined with a branch cut in $[-2,2]$, so that $\Im \m(z) >0$ for $\Im z>0$. 

For $i \in \llbracket 1, N \rrbracket$, we denote the $i$-th $N$-quantile of the semicircle distribution by $\gamma_i$ and define it implicitly by
\begin{equation}
\label{gammaidef}
\frac{i}{N}=\int_{-2}^{\gamma_i} \rho_{\mathrm{sc}}(x) \, \mathrm{d} x .
\end{equation}
We will often differentiate functions of a matrix $M\in \matn$ with respect to some entry $m_{ab}$ of $M$. For example, we will consider quantities such as $\partial_{ab} f(M)$, where $\partial_{ab}$ means that we consider $M$ as a function of its upper-triangular elements $\{ m_{ij} \}_{1 \le i \le j \le N}$ and differentiate with respect to $m_{ab}$ when $a \le b$, or with respect to $m_{ba}$ when $b \le a$. Most commonly, we take $f$ to be the resolvent $f(M) = (M -z)^{-1}$, or some product of resolvents.

Finally, we adopt the convention that $\mathbb{N} = \{1,2,3,\dots\}$.
\subsection{Local law for resolvent and multi-resolvent}\label{sec:multi-resolvent-local-law}

We require the isotropic local law proved in \cite{alex2014isotropic} and the multi-resolvent local law proved in \cite{CipErdSch22optimal}. We begin by recalling the notion of stochastic domination (which was introduced in \cite{erdos2013local}).

\begin{definition}[Stochastic domination]\label{def:stochasticDomination}
Let
$$
X=\left(X^{(N)}(u): N \in \mathbb{N}, u \in U^{(N)}\right), \quad Y=\left(Y^{(N)}(u): N \in \mathbb{N}, u \in U^{(N)}\right)
$$
be two families of nonnegative random variables, where $U^{(N)}$ is a possibly $N$-dependent parameter set. We say that $X$ is stochastically dominated by $Y$, uniformly in $u$, if for all (small) $\varepsilon>0$ and (large) $D>0$ there exists $N_0(\varepsilon, D) >0$ such that
$$
\sup _{u \in U^{(N)}} \mathbb{P}\left[X^{(N)}(u)>N^{\varepsilon} Y^{(N)}(u)\right] \leqslant N^{-D}
$$
for all $N \geqslant N_0(\varepsilon, D)$. Unless stated otherwise, throughout this paper the stochastic domination will always be uniform in all parameters apart from  $\delta$, $\tau$, and the constants $\mu_p$ (which were fixed in Section~\ref{s:conventions});
%in \eqref{finite-moment} and $\tau,\delta$ in the statement of Theorem \ref{thm:CLT};
thus, $N_0(\varepsilon, D)$ also depends on $\mu_p,\tau,\delta$. If $X$ is stochastically dominated by $Y$, uniformly in $u$, we use the notation $X \prec Y$. Moreover, if for some complex family $X$ we have $|X| \prec Y$ we also write $X=O_{\prec}(Y)$. The notion of stochastic domination can be trivially extended to deterministic quantities $A=A^{(N)}$ and $B=B^{(N)}$ with the understanding that $A\prec B$ implies that for all $\epsilon>0$, we have $A\leq N^\epsilon B$ for  all $N\geq N_0(\epsilon)$. In this case, we also write $A=O_{\prec}(B)$ if $|A|\prec B$.
\end{definition}

We first introduce the isotropic local law for a single resolvent. 
\begin{theorem}[Isotropic local law]\label{thm:iso-single}
    Let $H$ be a Wigner matrix, and let $G=(H-z)^{-1}$ be its resolvent. Then %, and fix $\omega >0$. Then 
    \begin{equation}\label{isotropic}
    \sup_{z\in \bm S}
\big|\langle\bm x, G(z) \mathbf{y}\rangle-\langle\bm x, \mathbf{y}\rangle \m(z)\big| \prec \sqrt{\frac{|\operatorname{Im} \m(z)|}{N \eta}}+\frac{1}{N \eta}
\end{equation}
%uniformly for $z=E+\iu\eta\in\bm S$ and
for any choice of deterministic vectors $\bm x,\bm y\in\mathbb S^{N-1}$, where $\eta = |\Im z|$. %In \eqref{isotropic}, the numbers $N_0$ in the definition of the $\prec$ notation (recall  Definition~\ref{def:stochasticDomination}) may additionally depend on $\omega$.
\end{theorem}
\begin{remark}\label{rmk:stochastic-continuation}
    In this and all following local laws, the high probability bound may be strengthened to hold simultaneously for all $z$ in the specified domain. For instance, \eqref{isotropic} may be strengthened to 
    \begin{align}
        \mathbb{P}\left[\bigcap_{z \in \mathbf{S}}\left\{\left|\langle\mathbf{x}, G(z) \mathbf{y}\rangle-\m(z)\langle\mathbf{x}, \mathbf{y}\rangle\right| \leqslant N^{\varepsilon}\left(\sqrt{\frac{\operatorname{Im} \m(z)}{N \eta}}+\frac{1}{N \eta}\right)\right\}\right] \geqslant 1-N^{-D},
    \end{align}
    for all $\epsilon>0,D>0$ and $N\geq N_0(\epsilon,D)$. It follows from a straightforward lattice argument combined with the Lipschitz continuity of $G,\m$ on $\bm S$. See \cite[Remark~2.7]{benaych2016lectures} for details.
\end{remark}

We next present the multi-resolvent local law. Observe that \Cref{thm:iso-single} establishes the deterministic approximation $G(z) \approx \m(z) I$, where $I \in \matn$ is the identity matrix. The multi-resolvent law identifies deterministic approximations to the more general quantities 
\begin{equation}\label{theproduct}
G(z_1) A_1 G(z_2) A_2 \cdots G_k(z_k) A_k G_{k+1}(z_{k+1}),
\end{equation}
where $z_1, \dots, z_{k+1} \in \bm S$ may be distinct and $A_1, \dots, A_k \in \matn$ are deterministic matrices. These deterministic approximations are defined using the notion of \emph{free cumulants} from free probability. We take a combinatorial approach to their definition and refer the reader to \cite[Section 4]{speicher1994multiplicative} for more on their origin in free probability.

Recall that for any random variable $X$, its moments $\mu^{(r)}(X)$ and cumulants $\kappa^{(r)}(X)$ satisfy the relation
\[
\mu^{(n)} = \sum_{\pi \in \Pi_n} \prod_{B \in \pi} \kappa^{(|B|)}
\]
for all $n \in \N$, where $\Pi_n$ is the set of all partitions of $\{1,2\dots, n\}$, the product is over all blocks $B$ of the partition $\pi$, and $|B|$ denotes the number of elements in $B$. For example, the partition $(145)(26)(3)$ has three blocks.
The free cumulants are represented similarly in terms of  \emph{non-crossing partitions}, which we now define. We follow the notation of \cite{GW16}; see also \cite{kreweras1972partitions}.

\begin{definition}
For all $k \in \mathbb{N}$, let $[k]$ denote the set $\{1,2,\dots, k\}$. A \emph{set partition} of $[k]$ is a set $\pi$ of disjoint subsets of $[k]$ whose union is $[k]$. The elements of $\pi$ are called \emph{blocks}. Given a set partition $\pi$, a \emph{bump} is an ordered pair $(i_1, i_2)$ such that $i_1$ and $i_2$ lie in the same block of $\pi$,  $i_1 < i_2$, and there is no $j$ in the same block with $i_1 < j < i_2$. We say that $\pi$ is a \emph{noncrossing partition} if for every pair of bumps $(i_1, i_2)$ and $(j_1, j_2)$ in $\pi$, it is not the case that $i_1 < j_1 < i_2 < j_2$. We let $\operatorname{NC}[k]$ denote the set of non-crossing partitions of $[k]$. %=\{1,\ldots,k\}$ arranged in in increasing order,
\end{definition}

We also need the notion of the Kreweras complement of a partition.
It relies on the following geometric description of non-crossing partitions: a partition $\pi$ of $\{1,\dots,n \}$ is non-crossing if and only if when the elements of $\{1,\dots,n \}$ are arranged in order on a circle, so that they divide the circle into equal arcs, the set of polygons $\{P_B\}$ given by the convex hulls of the points in each block $B$ are pairwise disjoint.

\begin{definition}
Arrange the points in $[k]$ equidistantly on the boundary of the unit disk $\mathbb D$, with labels increasing counterclockwise. Label the arcs between adjacent points so that arc $i$ connects point $i$ to its neighbor in the counterclockwise direction. Given $\pi \in \operatorname{NC}[k]$, we define the \emph{Kreweras complement} $K(\pi)\in \operatorname{NC}[k]$ of $\pi$ to be the partition such that two points $x,y \in [k]$ belong to the same block of $K(\pi)$ if and only if the arcs $x,y$ are in the same connected component of $\mathbb D \setminus \cup_{B \in \pi} P_B$, where $P_B$ denotes the convex hull of the vertices in the block $B$.
\end{definition}

%Let $\operatorname{NC}[k]$ denote the non-crossing partitions of the set $[k]=\{1,\ldots,k\}$ arranged in in increasing order, and 

%$K(\pi)$ denote the Kreweras complement of $\pi$ \cite[Definition~2.4]{GW16}. 
Further, for all $\pi \in \operatorname{NC}[k]$, and matrices $A_1, \dots, A_{k-1} \in \operatorname{Mat}_N$, we define the partial trace $\operatorname{pTr}_\pi$ associated to partition $\pi$ to be the element of $\matn$ given by
\begin{align}
\begin{split}
    \operatorname{pTr}_\pi\left(A_1, \ldots, A_{k-1}\right)=\frac{1}{N} \left( \prod_{j \in B(k) \backslash\{k\}} A_j\right) \prod_{B \in \pi \backslash B(k)}\left[ \Tr\left(\prod_{j \in B} A_j\right) \right],
\end{split}
\end{align}
where $B(k)\in\pi$ denotes the unique block containing $k$. We recall that by convention, an empty product is equal to $1$.

For any subset $B\subset [k]$ we define 
\begin{align}
\begin{split}
    m[B]\defeq m_{\operatorname{sc}}\big[\{z_i\mid i\in B\}\big]=\int_{-2}^2\rho_{\operatorname{sc}}(x)\prod_{i\in B}\frac{1}{x-z_i}\, \mathrm{d}x.
\end{split}
\end{align}
For every $k \in \mathbb{N}$, let $m_\circ[\cdot]\colon 2^{[k]} \rightarrow \mathbb{C}$ denote the free-cumulant transform of $m[\cdot]$, which is defined implicitly by requiring that the relation
\begin{align}
\begin{split}
    m[B]=\sum_{\pi \in \operatorname{NC}(B)} \prod_{B^{\prime} \in \pi} m_{\circ}\left[B^{\prime}\right], \quad \forall B \subset[k]
\end{split}
\end{align}
holds for all $k$. 
For example, when $k=1$, we have $m_\circ[i] =m[i]$, and for $k=2$ we have $m_\circ[i,j] = m[i,j] -m[i]m[j]$. 
For further details, see the discussion following \cite[Definition 2.3]{cipolloni2022thermalisation}. We now define the deterministic equivalent for \eqref{theproduct}.

\begin{definition}
For arbitrary deterministic matrices $A_1,\ldots,A_{k-1}\in \operatorname{Mat}_N$ and spectral parameters $z_1,\ldots,z_k \in \mathbb{C} \backslash \mathbb{R}$, define
\begin{align}\label{Mdefinition}
\begin{split}
    M(z_1, A_1, \ldots, A_{k-1}, z_k):=\sum_{\pi \in \operatorname{NC}[k]} \operatorname{pTr}_{K(\pi)}\left(A_1, \ldots, A_{k-1}\right) \prod_{B \in \pi} m_{\circ}[B].
\end{split}
\end{align}
\end{definition}
We are now ready to state the multi-resolvent local laws necessary for our work.
%For positive quantities $f,g>0$, we write $f\lesssim g$ if there exists a constant $C>0$  such that $ f\le Cg$. 
\begin{lemma}[{\cite[Lemma 2.4]{CipErdSch22optimal}}]
Fix $k, m\in \mathbb{N}$ with $m\le k$ and a constant $C_0 >0$. Let $A_1,\dots, A_k \in \operatorname{Mat}_N$ be deterministic matrices such that $\left\|A_i\right\| \le C_0$ for all $1\le i \le k$, and suppose that 
%If %at least $m$  of these $k$ matrices %$B_1, \ldots, B_k$ with $\left\|B_i\right\| \lesssim 1$ 
%are traceless, i.e. 
$\Tr A_j=0$ holds for at least $m$ distinct indices $j$. Then there exists a constant $C=C(C_0,k)>0$ such that%(for some $0 \leq a \leq k$ ), then it holds that
\begin{equation}
\left| \Tr\left(M\left(z_1, A_1, \ldots, z_{k-1}, A_{k-1}, z_k\right) A_k\right)\right| \le \begin{cases} CN \eta^{-(k-1-\lceil m / 2\rceil)} & d \leq 1 \\
CN d^{-k} & d \geq 1\end{cases}\label{eqn:center-avg}
\end{equation}
and
\begin{equation}
\left\|M\left(z_1, A_1 \ldots, z_k, A_k, z_{k+1}\right)\right\| \le \begin{cases}C\eta^{-(k-\lceil m / 2\rceil)} %\frac{C}{\eta^{k-\lceil m / 2\rceil}} 
& d \leq 1 \\
%\frac{C}{d^{k+1}} 
C d^{-k-1}
& d \geq 1,\end{cases} \label{eqn:center-iso}
\end{equation}
where $\eta:=\min _j\left|\Im z_j\right|$ and $d:=\min _j \operatorname{dist}\left(z_j,[-2,2]\right)$.
\end{lemma}

\begin{theorem}[{\cite[Theorem 2.5]{CipErdSch22optimal}}]
Let $H$ be an $N\times N$ Wigner matrix and let $G=(H-z)^{-1}$ be its resolvent. Fix $m, k\in \mathbb{N}$ with $m \le k$ and $z_1, \ldots, z_{k+1} \in \bm S$. Fix $C_0>0$, and let $A_1, \ldots, A_k$ be deterministic matrices such that $\left\|A_j\right\| \le C_0$ for all $1 \le j \le k$, and $\Tr A_j=0$ for at least $m$ distinct indices $j$. 
%and such that at least $m$ of them are traceless for some $0 \leq m \leq k$. 
 Then %, we have %the optimal averaged local law
\begin{align}\label{eqn:avg}
\left| \Tr\left(G_1 A_1 \cdots G_k A_k-M\left(z_1, A_1, \ldots, A_{k-1}, z_k\right) A_k\right)\right| \prec \begin{cases}
%\frac{1}{ \eta^{k-m / 2}} 
 \eta^{-(k-m / 2)}
& d \leq 1 \\ 
d^{-(k+1)}
%\frac{1}{ d^{k+1}} 
& d \geq 1,\end{cases}
\end{align}
and for any deterministic vectors $\boldsymbol{x}, \boldsymbol{y} \in \mathbb{R}^N$ such that $\|\boldsymbol{x}\|+\|\boldsymbol{y}\| \le C_0$, we have %the optimal isotropic local law
\begin{align}\label{eqn:iso}
\left|\left\langle\boldsymbol{x},\left(G_1 A_1 \cdots G_k A_k G_{k+1}-M\left(z_1, A_1, \ldots, A_k, z_{k+1}\right)\right) \boldsymbol{y}\right\rangle\right| \prec \begin{cases}
%\frac{1}{\sqrt{N} \eta^{k-m / 2+1 / 2}} 
N^{-1/2} \eta^{-(k-m / 2+1 / 2)}
& d \leq 1 \\ 
N^{-1/2} d^{-(k+2)}
%\frac{1}{\sqrt{N} d^{k+2}} 
& d \geq 1.\end{cases}
\end{align}
Here $G_j:=G\left(z_j\right)$, $\eta:=\min _j\left|\Im z_j\right|$, and $d:=\min_j\operatorname{dist}\left(z_j,[-2,2]\right)$. In \eqref{eqn:avg} and \eqref{eqn:iso}, the numbers $N_0$ in the definition of the $\prec$ notation (recall  Definition~\ref{def:stochasticDomination}) may depend on $k$ and $C_0$. 
\end{theorem}
\subsection{Central Limit Theorem for GOE}
We require the following central limit theorem for eigenvector statistics of Gaussian random matrices. It is proved in Appendix~\ref{s:goeclt}. We recall that the Gaussian Orthogonal Ensemble (GOE) is a Wigner matrix with Gaussian entries (with variance matrix as in \Cref{def:wigner}) .
%\red{Need to change to corresponding CLT for the new overlapping.}
%We recall the following result from \cite[Theorems~2.3 and 2.4]{o2016eigenvectors}.
\begin{theorem}[Central Limit Theorem for GOE]\label{thm:GOE-CLT}
Let $H$ be a GOE matrix and fix $\delta\in(0,1)$. 
Let $A=A_N\in\R^{N\times N}$ be a deterministic sequence of traceless matrices such that $A=A^*$, $\|A\|\leq 1$ and $\Tr(A^2)\geq N^{1-\delta}$. % let $\I=\I_N\in\llbracket N^{1-\delta},N\rrbracket$ be a deterministic sequence of subsets, 
Let $\ell=\ell_N\in\llbracket1,N\rrbracket$ be a deterministic sequence of indices, and let $\bm u=\bm u_{\ell}^{(N)}$ be the corresponding sequence of $\ell^2$-normalized eigenvectors of $H$. 
% Let $\left(\bm q_{\alpha}^{(N)}\right)_{\alpha\in\I}=\left(\bm q_\alpha\right)_{\alpha\in\I}$ be a deterministic sequence of sets of orthogonal vectors in $\mathbb S^{N-1}$. 
Then
\begin{equation}
\sqrt{\frac{ N^2}{2\Tr(A^2)}}\left\langle \bm u,A\bm u\right\rangle \rightarrow \mathcal{N}(0,1),
\end{equation}
with convergence in distribution.
\end{theorem}
\subsection{Eigenvector Thermalization}
We also recall the following eigenvector thermalization bound from \cite[Theorem~2.2]{cipolloni2021eigenstate}.
\begin{theorem}
Let $H$ be a Wigner matrix. 
Let $A=A_N\in\R^{N\times N}$ be a deterministic sequence of traceless matrices such that $\|A\|\leq 1$, let $\ell=\ell_N\subset \llbracket 1, N \rrbracket$ be a deterministic sequence of indices, and let $\bm u=\bm u_{\ell}^{(N)}$ be the corresponding sequence of eigenvectors of $H$. 
% Let $\left(\mathbf{q}_\alpha\right)_{\alpha \in I}=\left(\mathbf{q}_\alpha^{(N)}\right)_{\alpha \in \I}$ be a deterministic sequence of sets of orthogonal vectors in $\mathbb{S}^{N-1}$. 
Then
\begin{align}\label{eqn:que}
\begin{split}
    \left|\langle \bm u,A\bm u\rangle \right|\prec N^{-1/2}.
\end{split}
\end{align}
\end{theorem}

\section{Regularized observables}\label{sec:regularized-observables}

We retain the conventions stated in \Cref{s:conventions}.

\subsection{Definitions}
We begin by defining notation for the self-overlaps of eigenvectors and typical eigenvalue spacings.
\begin{definition}[Self-overlaps and spacings]\label{def:selfoverlap}
Let $H$ be an $N\times N$ Wigner matrix and let $A\in\mathbb R^{N\times N}$ be a deterministic traceless matrix. 
Define the self-overlap %of $k$-th $\ell^2$-normalized 
of the eigenvector $\bm u_\ell$ by
\begin{align}
\begin{split}
    p_\ell=p_\ell(A)\coloneqq \langle \bm u_\ell,A\bm u_\ell\rangle.
\end{split}
\end{align}
Denote the normalized overlap $p_\ell$ by
\begin{align}\label{eq:notpl}
    \widehat{p}_\ell = \sqrt{\frac{N^2}{2\Tr(A^2)}}\cdot p_\ell.
\end{align}
Denote the typical size of the $\ell$-th eigenvalue gap by
\begin{align}
\begin{split}
    \Delta_\ell=N^{-2/3}\ell^{-1/3}.
\end{split}
\end{align}
\end{definition}
We now prepare to define $v_\ell$, which serves as a smooth regularization of $\widehat p_\ell$. 
\begin{definition}[Smoothed indicator function]
For any $E_1,E_2 \in\R$ with $E_1<E_2$, and $\eta >0$,
let $f_{E_1,E_2, \eta}$ denote %the characteristic function on $[a,b]$ smoothed on scale $\eta$, i.e.
a function such that
$f_{E_1,E_2,\eta}=1$ on $[E_1,E_2]$, $f_{E_1,E_2,\eta}=0$ on $\R\backslash[E_1-\eta,E_2+\eta]$, and $|f_{E_1,E_2,\eta}'|\leq C\eta^{-1}, |f_{E_1,E_2,\eta}''|\leq C\eta^{-2}$ on $\mathbb{R}$.
\end{definition}
For the next definition, recall that $\gamma_\ell$ denotes the typical location of the $\ell$-th smallest eigenvalue and was defined in \eqref{gammaidef}.
\begin{definition}[Regularized self-overlap]\label{def:regularized}
%Retain the notation and assumptions of Definition~\ref{def:selfoverlap}.

Let $\delta_i>0$ for $1 \le i \le 5$ be parameters, and let $H$ be a Wigner matrix. Define 
\begin{equation}\label{regularizednotation}
\begin{gathered}
 \eta_\ell\coloneqq\Delta_\ell N^{-\delta_1},\quad I_\ell\coloneqq\left[\gamma_\ell-\Delta_\ell N^{\delta_2}, \gamma_\ell+\Delta_\ell N^{\delta_2}\right],\quad E^{\pm}\coloneqq E \pm \Delta_\ell N^{-\delta_3}, \\
 \nu \coloneqq \Delta_\ell N^{-\delta_4}, \quad\tilde{\eta}_\ell\coloneqq\Delta_\ell N^{-\delta_5}.
\end{gathered}
\end{equation}
Also, set
\[
\varpi \coloneqq  \frac{1}{2} \left(\frac{\ell}{N}\right)^{2/3}, \quad 
\tilde f\coloneqq f_{-\varpi,\varpi,\varpi},  \quad q\coloneqq f_{\ell-1/3 ,\ell+1/3,1/3},\]
and
\begin{equation}\label{vartheta}
\begin{aligned}
    \vartheta\coloneqq -2-N^{-2/3+\delta_1},\quad f_E\coloneqq f_{\vartheta, E^{+},\nu} .
\end{aligned}
\end{equation}
Define 
\begin{equation}\label{eqn:regularized-x}
\begin{aligned}
x(E) \equiv x_\ell(E) =\frac{\eta_\ell}{\pi} \sum_{i} \frac{\widehat{p}_i}{\left(\lambda_i-E\right)^2+\eta_\ell^2}  = \frac{\eta_\ell}{\pi}\sqrt{\frac{N^2}{2\Tr(A^2)}} \Tr(GA\bar G)
\end{aligned}
\end{equation}
and
\begin{equation}\label{eqn:regularized-y}
\begin{aligned}
y(E) \equiv y_\ell(E)&=  \frac{1}{2 \pi} \int_{\mathbb{R}^2} \mathrm{i} \sigma f_E^{\prime \prime}(e) \tilde f(\sigma) \operatorname{Tr} G(e+\mathrm{i} \sigma) \bm{1}\left(|\sigma|>\tilde{\eta}_\ell\right) \mathrm{d} e \, \mathrm{d} \sigma \\
&\quad+\frac{1}{2 \pi} \int_{\mathbb{R}^2}\left(\mathrm{i} f_E(e) \tilde f^{\prime}(\sigma)-\sigma f_E^{\prime}(e) \tilde f^{\prime}(\sigma)\right) \operatorname{Tr} G(e+\mathrm{i} \sigma)\, \mathrm{d} e \, \mathrm{d} \sigma.
\end{aligned}
\end{equation}
Finally, set $\bm \delta = (\delta_1, \dots, \delta_5)$ and define the regularized observable %of suitably scaled $p_\ell$ as
\begin{equation}\label{eqn:regularized-observable}
v_\ell\equiv v_\ell(\bm \delta, A)=\int_{I_\ell} x(E) q\left(y_E\right) \mathrm{d} E.
\end{equation}
\end{definition}
\begin{remark}
The definition of $x(E)$ is analogous to the regularization \eqref{introreg} given in the introduction, with the eigenvector entry $\bm u_\ell(k)^2$ there replaced here by the self-overlap.  The definition of $y(E)$ is more subtle, and comes from using the Helffer--Sj\H{o}strand formula to provide a smooth approximation 
to  $\Tr f_E(H)$.  %of the indicator function $\bm 1(\lambda_{\ell}\leq E^+\leq \lambda_{\ell+1})$. 
We refer the reader to the proof of \Cref{lem:reg-ob-3} to see how this specific form of $y(E)$ arises. 

Below, we choose the parameters $\delta_i$ so that
\[
\delta_2 < \delta_3 < \delta_1 < \delta_4 < \delta_5. 
\]
In particular, $f_E$ is a step function regularized on scale smaller than $\eta_\ell$, and $|I_\ell| \gg \Delta_\ell \gg \eta_\ell$. 
\end{remark}

Before stating the main lemma in this section, we need the following several results.
\begin{theorem}[Eigenvalue rigidity {\cite[Theorem 2.2]{erdos2012rigidity}}] Let $H$ be a Wigner matrix. For all $i\in\llbracket1,N\rrbracket$, we have
\begin{equation}\label{eqn:rig}
|\lambda_i-\gamma_i|  \prec\Delta_i.
\end{equation}
\end{theorem}
\begin{proposition}[Level repulsion at the edge {\cite[Proposition 5.7]{BL22LInfinity}}] \label{prop:rigidity}
Let $H$ be a Wigner matrix. Then there exists $\epsilon_0>0$ such that for all $\epsilon\in(0,\epsilon_0)$, there exists a constant $C=C(\epsilon)$ such that for all $i\in\llbracket1,\lfloor N/2\rfloor\rrbracket$,
\begin{align}\label{eqn:level-repulsion}
\begin{split}
    \P\left(\lambda_{i+1}-\lambda_i<N^{-2/3-\epsilon}i^{-1/3}\right)\leq CN^{-\epsilon}.
\end{split}
\end{align}
\end{proposition}
\begin{lemma}[{{\cite[Lemma~4.9]{BL22LInfinity}}}]With the definitions in Definition \ref{def:regularized}, for all $\epsilon>0$, we have\footnotemark
\begin{align}\label{eqn:small-remaining-terms}
\sum_{i:|i-\ell| \geqslant N^{\epsilon}} \frac{1}{\left(\lambda_i-\lambda_\ell\right)^2}  \prec N^{4 / 3-\epsilon} \ell^{2 / 3}.
\end{align}
\footnotetext{There is a misprint in \cite[Lemma~4.9]{BL22LInfinity}. The sign of the $\omega$ on the right-hand side of the inequality should be negative.}
\end{lemma}
We now fix the parameters used in the definition of \eqref{eqn:regularized-observable} for the rest of the paper. 
\begin{definition}[Parameters]\label{para}
Recall the parameters $\tau\in(0,1)$ in Theorem \ref{thm:CLT}. Suppose that $\lambda_\ell$ satisfies level repulsion estimate \eqref{eqn:level-repulsion} with 
\begin{equation*}
\epsilon=\epsilon_1\defeq\min\left\{\frac{\epsilon_0}{2},10^{-9}\tau\right\}.
\end{equation*}
Fix the parameters appearing in Definition \ref{regularizednotation} to be
\begin{align}\label{deltachoices}
%\begin{split}
    \delta_1=2\epsilon_1,\quad \delta_2=10^{-2}\epsilon_1,\quad \delta_3=\frac{\epsilon_1}{2},\quad\delta_4=6\epsilon_1,\quad\delta_5=8\epsilon_1,
%\end{split}
\end{align}
and fix the parameter $\delta$ in \Cref{thm:CLT} to be 
\begin{align*}
\begin{split}
    \delta=10^{-2}\epsilon_1.
\end{split}
\end{align*}
\end{definition}
\begin{lemma}\label{lem:bound-x-integral}
Under the assumptions of Theorem \ref{thm:CLT}, we have
\begin{align}\label{eqn:bound-x-integral}
    \int_{I_\ell}|x(E)|\chi(E)\,\d E\prec N^{\delta_2+\delta/2},\quad 
   %\beb \int_{I_\ell}|x(E)| \,\d E\prec N^{\delta_2+\delta/2},\eeb
\end{align}
where $\chi(E)=\bm 1(\lambda_{\ell}\leq E^+\leq \lambda_{\ell+1})$.
\end{lemma}
\begin{proof}By the QUE bound \eqref{eqn:que} and the assumption $\Tr(A^2)\geq N^{1-\delta}$, it suffices to show
\begin{align}\label{eqn:bound-integral}
\begin{split}
    \sum_{i}\int_{I_\ell}\frac{\eta_\ell}{\pi}\frac{1}{(\lambda_i-E)^2+\eta_\ell^2}\chi(E)\,\d E\prec N^{\delta_2}.
\end{split}
\end{align}
We break the sum \eqref{eqn:bound-integral} into two parts and find that it equals
\begin{align}\label{eqn:integral-break}
\begin{split}
    \sum_{i:|i-\ell|< N^{\delta_2}}\int_{I_\ell}\frac{\eta_\ell}{\pi}\frac{1}{(\lambda_i-E)^2+\eta_\ell^2}\chi(E)\, \d E+\sum_{i:|i-\ell|\geq N^{\delta_2}}\int_{I_\ell}\frac{\eta_\ell}{\pi}\frac{1}{(\lambda_i-E)^2+\eta_\ell^2}\chi(E)\, \d E.
\end{split}
\end{align}
For the first term in \eqref{eqn:integral-break}, using the integral
\begin{align}\label{poisson-identity}
\begin{split}
    \int_{\mathbb R}\frac{\eta_\ell}{E^2+\eta_\ell^2}\, \d E=\pi,
\end{split}
\end{align}
we bound it by 
\begin{align}\label{eqn:bound-integral-break-1}
\begin{split}
    \sum_{i:|i-\ell|< N^{\delta_2}}\int_{I_\ell}\frac{\eta_\ell}{\pi}\frac{1}{(\lambda_i-E)^2+\eta_\ell^2}\chi(E)\, \d E<2N^{\delta_2}.
\end{split}
\end{align}
For the second term in \eqref{eqn:integral-break}, using \eqref{eqn:small-remaining-terms}, rigidity \eqref{eqn:rig}, and the definition of $\chi(E)$, it follows that
\begin{align}\label{eqn:bound-integral-break-2}
\begin{split}
    \sum_{i:|i-\ell|\geq N^{\delta_2}}\int_{I_\ell}\frac{\eta_\ell}{\pi}\frac{1}{(\lambda_i-E)^2+\eta_\ell^2}\chi(E)\, \d E\prec N^{-\delta_2-\delta_1}.
\end{split}
\end{align}
Combining \eqref{eqn:bound-integral-break-1} and \eqref{eqn:bound-integral-break-2} completes the proof of the first bound in \eqref{eqn:bound-x-integral}. %\beb The proof of the second bound is similar, with the only difference being that \eqref{eqn:bound-integral-break-2} becomes
%\[
%\sum_{i:|i-\ell|\geq N^{\delta_2}}\int_{I_\ell}\frac{\eta_\ell}{\pi}\frac{1}%{(\lambda_i-E)^2+\eta_\ell^2}\, \d E\prec N^{-\delta_2-\delta_1}.\eeb
%\]
\end{proof}
The following lemma is our main comparison result for the smoothed observable $v_\ell$. 
\begin{lemma}\label{lem:regularized-observable}
Let $H$ be a Wigner matrix, and let the parameters $\epsilon_1>0$ and $\delta_1,\ldots,\delta_5$ be chosen as in Definition \ref{para}.
Let $g:\R\rightarrow\R$ be a compactly supported smooth function.
Then there exists a constant $c(\tau,g)>0$ such that 
\begin{align*}
\begin{split}
    \left| \E \Big[g\big(\widehat{p}_\ell(A)\big)\Big]-\E \big[g\big(v_\ell(\bm \delta, A)\big) \Big] \right|\le c^{-1} N^{-c}.
\end{split}
\end{align*}
\end{lemma}
% \blue{I showed it for $N^{-c}=N^{-\delta_2}$, but I'm not sure if I should just leave it as $c$ for all the lemmas below. I left maybe a bit too much details for now because I'm not sure what information is necessary to keep. One idea of the simplification is to absorb all the $N^{\delta_0,\delta_1}$ into big $O$ or constant $C$, or simply use $O_\prec$.}\\
Lemma \ref{lem:regularized-observable} is an immediate consequence of Lemmas \ref{lem:reg-ob-1}, \ref{lem:reg-ob-2} and \ref{lem:reg-ob-3} below, which we now state and prove. Analogously to our definition of $E^+$ and $E^-$ in Definition~\ref{def:regularized}, we define 
\[
\lambda_i^+ = \lambda_i + \Delta_i N^{-\delta_3}, \qquad \lambda_i^- = \lambda_i  - \Delta_i N^{-\delta_3}.
\]
We also recall the integral 
\begin{equation}\label{poissonintegral}
\int_{-\infty}^u \frac{ y\, \d x}{x^2 + y^2} = \frac{\pi}{2} + \arctan \left( \frac{u}{y} \right),
\end{equation}
along with the facts
\[
\arctan(x) + \arctan( x^{-1} ) = \operatorname{sgn}(x) \frac{\pi}{2} , \qquad \big| \arctan(x) \big| \le 2|x|.
\]

% We need the following theorems for the proof. 
% \begin{theorem}[Eigenvalue Rigidity {\cite[Theorem 2.2]{erdos2012rigidity}}] Let $H$ be a Wigner matrix. For all $i\in\llbracket1,N\rrbracket$, we have
% \begin{equation}\label{eqn:rig}
% |\lambda_i-\gamma_i|  \prec\Delta_i.
% \end{equation}
% \end{theorem}

% \begin{proposition}[Level repulsion at the edge {\cite[Proposition 5.7]{BL22LInfinity}}] 
% Let $H$ be a Wigner matrix. Then there exists $\epsilon_0>0$ such that for all $\epsilon\in(0,\epsilon_0)$, there exists constants $C=C(\epsilon)$ such that for any $i\in\llbracket1,\lfloor N/2\rfloor\rrbracket$,
% \begin{align}\label{eqn:level-repulsion}
% \begin{split}
%     \P\left(\lambda_{i+1}-\lambda_i<N^{-2/3-\epsilon}i^{-1/3}\right)\leq CN^{-\epsilon}.
% \end{split}
% \end{align}
% \end{proposition}

\begin{lemma}\label{lem:reg-ob-1}
    Maintain the notation and assumptions of Lemma \ref{lem:regularized-observable}.
    Recalling Definition \ref{def:regularized}, we have
\begin{align*}
\begin{split}
    \E \big[g(\widehat{p}_\ell)\big]-\E \left[g\left(\int_{I_\ell} x(E) \chi(E)\right)\right]=O(N^{-\epsilon_1/4}),
\end{split}
\end{align*}
where $\chi(E):= \bm1\left(\lambda_\ell\leq E^+ \leq \lambda_{\ell+1}\right)$.
\end{lemma}
\begin{proof}
We first write 
\begin{equation}\label{e0}
   \widehat{p}_\ell=\frac{\eta_\ell}{\pi}\int_{\mathbb R}\frac{\widehat{p}_{\ell}}{(E-\lambda_\ell)^2+\eta_\ell^2}\;\mathrm{d} E. 
\end{equation}
%Let 
%\begin{equation}
%a \;\leq\; \lambda_\alpha^-,\;\;\;\;\;\; b \;\geq\; \lambda_\alpha^+\,,
%\end{equation}

We suppose without loss of generality that $\ell \le N^{1-\tau}$. By the assumption on $g$, \eqref{poissonintegral}, rigidity \eqref{eqn:rig}, and the bound \eqref{eqn:que}, we write
\begin{align}\label{e0.1}
\begin{split}
    \E\big[g(\widehat{p}_\ell)\big]&=\E\left[g\left(\frac{\eta_\ell}{\pi}\int_{E_1}^{E_2}\frac{\widehat{p}_\ell}{(E-\lambda_\ell)^2+\eta_\ell^2}\d E\right)\right]+O_\prec(N^{-\delta_1+\delta_3 +\delta/2})\\
    &=\E\left[g\left(\frac{\eta_\ell}{\pi}\int_{E_1}^{E_2}\frac{\widehat{p}_\ell}{(E-\lambda_\ell)^2+\eta_\ell^2}\d E\right)\right]+O(N^{-\epsilon_1/2}),
\end{split}
\end{align}
where 
\begin{align*}
\begin{split}
    E_1=\lambda_\ell^-,\quad E_2= \max\left\{\lambda_\ell^{+},\lambda_{\ell+1}^-\right\}.
\end{split}
\end{align*}
We now show that the integral over $[E_1,E_2]$ can be approximated by integrating over $[\lambda_\ell^-, \lambda_{\ell+1}^-]$. From \eqref{eqn:level-repulsion} and the parameter choice $\delta_3=\epsilon_1/2$, we have
\begin{align}\label{eqn:level-repulsion-specific}
\begin{split}
\P\left(\lambda_{\ell+1}^-\leq \lambda_\ell^+\right)\leq C N^{-\epsilon_1/2}.
\end{split}
\end{align}
Decomposing the integral
\begin{align*}
\begin{split}
    \int_{E_1}^{E_2}=\int_{\lambda_{\ell}^-}^{\lambda_{\ell+1}^{-}}+\bm 1\left(\lambda_{\ell+1}^-\leq \lambda_\ell^+\right)\int_{\lambda_{\ell+1}^-}^{\lambda_{\ell}^+},
\end{split}
\end{align*}
we have from \eqref{e0.1} that
\begin{align*}
\begin{split}
    \E\big[g(\widehat{p}_\ell)\big]&=\E\left[g\left(\frac{\eta_\ell}{\pi}\int_{\lambda_\ell^-}^{\lambda_{\ell+1}^-}\frac{\widehat{p}_\ell}{(E-\lambda_\ell)^2+\eta_\ell^2}\, \d E\right)\right]+ N^{\delta/2} \cdot O_\prec \big(N^{\delta/2}\P(\lambda_{\ell+1}^-\leq \lambda_\ell^+)\big)+O(N^{-\epsilon_1/2})\\
    &=\E\left[g\left(\frac{\eta_\ell}{\pi}\int_{\lambda_\ell^-}^{\lambda_{\ell+1}^-}\frac{\widehat{p}_\ell}{(E-\lambda_\ell)^2+\eta_\ell^2}\, \d E\right)\right]+O(N^{-\epsilon_1/4}),
\end{split}
\end{align*}
where we used  \eqref{eqn:que} and $\Tr(A^2)\geq N^{1-\delta}$ in the first step and \eqref{eqn:level-repulsion-specific} in the second step.
By rigidity \eqref{eqn:rig},
\begin{align*}
\begin{split}
    |\lambda_\ell-\gamma_\ell|\prec \Delta_\ell,\quad |\lambda_{\ell+1}-\gamma_{\ell+1}|\prec \Delta_\ell,
\end{split}
\end{align*}
which implies using the definitions of $I_\ell$ and $\chi(E)$ that 
\begin{align}\label{e0.2}
\begin{split}
    \E\left[g(\widehat{p}_\ell)\right]=\E\left[g\left(\frac{\eta_\ell}{\pi}\int_{I_\ell}\frac{\widehat{p}_\ell}{(E-\lambda_\ell)^2+\eta_\ell^2}\chi(E)\, \d E\right)\right]+O(N^{-\epsilon_1/4}).
\end{split}
\end{align}
Our next goal is to replace the first term on the right-hand side of \eqref{e0.2} by
\begin{align}%\label{eq1}\notag
    &\E \left[g\left(\int_{I_{\ell}} x(E) \chi(E)\right)\right] \\ & \quad =\E \left[g \left(\frac{\eta_\ell}{\pi}\int_{I_\ell}\frac{\widehat{p}_\ell}{(E-\lambda_\ell)^2+\eta_\ell^2}\chi(E) \;\mathrm{d} E+\frac{\eta_\ell}{\pi}\sum_{i\neq\ell} \int_{I_{\ell}} \frac{\widehat{p}_i}{\left(\lambda_i-E\right)^2+\eta_\ell^2}\chi(E) \;\mathrm{d} E \right)\right]. \notag
\end{align}
Using mean value theorem, \eqref{eqn:que}, and \eqref{e0.2} in the first step, and \eqref{eqn:que} and \eqref{eqn:small-remaining-terms} in the second step, we have
\begin{align}
    &\left|\E\left[g(\widehat{p}_\ell)\right]-\E\left[g\left(\int_{I_\ell}x(E)\chi(E)\, \d E\right)\right]\right|\notag \\
    &\leq O_\prec\left(N^{\delta/2}\right)\E\left[\sum_{i\neq \ell}\int_{I_\ell}\frac{\eta_\ell}{\pi}\frac{1}{(\lambda_i-E)^2+\eta_\ell^2}\chi(E)\, \d E\right]+O(N^{-\epsilon_1/4})\notag \\
    &= O_\prec\left(N^{\delta/2}\right)\E\left[\sum_{i:1\leq|i-\ell|<N^{\delta_2}}\int_{I_\ell}\frac{\eta_\ell}{\pi}\frac{1}{(\lambda_i-E)^2+\eta_\ell^2}\chi(E)\, \d E\right]+O_\prec(N^{-\delta_1 - \delta_2 - \delta_3 + \delta/2})+O(N^{-\epsilon_1/4}) \notag \\
    &=O_\prec\left(N^{\delta/2}\right)\E\left[\sum_{i:1\leq|i-\ell|<N^{\delta_2}}\int_{I_\ell}\frac{\eta_\ell}{\pi}\frac{1}{(\lambda_i-E)^2+\eta_\ell^2}\chi(E)\, \d E\right]+O(N^{-\epsilon_1/4}).\label{e1}
\end{align} 
Next, we would like to bound the expectation term in  \eqref{e1}. 
We decompose it
into two parts. 

Firstly, for $i>\ell$, we have
\begin{align}
\E\left[\sum_{i:1\leq i-\ell<N^{\delta_2}}\int_{I_\ell}\frac{\eta_\ell}{\pi}\frac{1}{(\lambda_i-E)^2+\eta_\ell^2}\chi(E)\, \d E\right]& \leq N^{\delta_2}\E\left[\int_{-\infty}^{\lambda_{\ell+1}^-}\frac{\eta_\ell}{\pi}\frac{1}{(\lambda_{\ell+1}-E)^2+\eta_\ell^2}\, \d E\right]\notag \\ 
&\prec N^{\delta_2-\delta_1+\delta_3}.\label{eqn:small-terms-1}
\end{align}
Suppose now that $i<\ell$.
On the event $\mathcal B\defeq\{\lambda_\ell-\lambda_{\ell-1}>4\Delta_{\ell}N^{-\delta_3}\}$, we have
\begin{align*}
\begin{split}
    \chi(E)(E-\lambda_i)^2\geq (\lambda_{\ell}-\lambda_{i})^2-2\Delta_\ell N^{-\delta_3}(\lambda_\ell-\lambda_i)\geq \frac{1}{2}(\lambda_\ell-\lambda_i)^2\geq\frac{1}{2}(\lambda_\ell-\lambda_{\ell-1})^2.
\end{split}
\end{align*}
Therefore, 
\begin{align*}
\begin{split}
    \bm 1(\mathcal B)\chi(E)\frac{1}{(E-\lambda_i)^2+\eta_{\ell}^2}\leq \bm 1(\mathcal B)\chi(E)\frac{2}{(\lambda_\ell-\lambda_{\ell-1})^2+\eta_\ell^2}\leq \chi(E)\frac{N^{2\delta_3}}{8\Delta_\ell^2}.
\end{split}
\end{align*}
Now we have
\begin{align}\label{eqn:small-terms-2}
\begin{split}
    \E\left[\sum_{i:1\leq \ell-i<N^{\delta_2}}\int_{I_\ell}\frac{\eta_\ell}{\pi}\frac{1}{(\lambda_i-E)^2+\eta_\ell^2}\chi(E)\, \d E\right]\prec\, N^{\delta_2}\P\left(\mathcal B^c\right)+N^{\delta_2}\eta_\ell\Delta_\ell\frac{N^{2\delta_3}}{\Delta_\ell^2}\leq\,N^{\delta_2-\delta_3},
\end{split}
\end{align}
where we used \eqref{poisson-identity}  and $\lambda_{\ell+1}-\lambda_{\ell}\prec \Delta_\ell$ in the first inequality (to control the length of the interval in the definition of $\chi(E)$), and \eqref{eqn:level-repulsion} in the second inequality.
Combining \eqref{eqn:small-terms-1} and \eqref{eqn:small-terms-2}, we have
\begin{align}\label{eqn:small-terms-all}
\begin{split}
    \E\left[\sum_{i:1\leq|i-\ell|<N^{\delta_2}}\int_{I_\ell}\frac{\eta_\ell}{\pi}\frac{1}{(\lambda_i-E)^2+\eta_\ell^2}\chi(E)\, \d E\right]\prec N^{\delta_2-\delta_3}.
\end{split}
\end{align}

Inserting \eqref{eqn:small-terms-all} into \eqref{e1}, we have
\begin{align*}
\begin{split}
    \left|\E\left[g(\widehat{p}_\ell)\right]-\E\left[g\left(\int_{I_\ell}x(E)\chi(E)\, \d E\right)\right]\right|=O(N^{-\epsilon_1/4}).
\end{split}
\end{align*}
\end{proof}

\begin{lemma}\label{lem:reg-ob-2}
    Maintain the assumptions of Lemma \ref{lem:regularized-observable} and recall Definition \ref{def:regularized}. We have 
\begin{equation}
    \E \left[g\left(\int_{I_\ell} x(E) \chi(E)\; \mathrm{d} E\right)\right] - \E \left[g\left(\int_{I_\ell} x(E) q\left(\Tr f_E(H)\right)\, \mathrm{d} E\right)\right]=O(N^{-\epsilon_1/2}), \label{e4}
\end{equation}
where $\chi(E):= \bm1\left(\lambda_\ell\leq E^+ \leq \lambda_{\ell+1}\right)$.
\end{lemma}
\begin{proof} Recall that $\vartheta=-2-N^{-2/3+\delta_1}$ from \eqref{vartheta}.
    Let $\theta= \mathbf{1}_{[\vartheta,E^+]}$ and $\mathcal B$ denote the event $\{\lambda_1\geq \vartheta\}$. By the definition of $\chi(E)$, %the first term becomes \ber first term of what? \eer 
    \begin{align*}
    \begin{aligned}
       \bm 1(\mathcal B) \int_{I_\ell} x(E) \chi(E)\; \mathrm{d} E &= \bm 1(\mathcal B)\int_{I_\ell} x(E) \;\textbf{1}(\lambda_{\ell}\leq E^{+} \leq {\lambda_{\ell+1}}) \; \mathrm{d} E\\ &=  \bm 1(\mathcal B)\int_{I_\ell} x(E)\; \textbf{1}({\cal N}(-\infty, E^+) = \ell) \;\mathrm{d} E\\
        &=\bm 1(\mathcal B)\int_{I_\ell} x(E) \; q(\Tr \theta(H)) \;\mathrm{d} E,
    \end{aligned}
    \end{align*}
where $\mathcal N(E_1,E_2)$ denotes the number of eigenvalues in $[E_1,E_2]$.

By the definition of $f_E$ in  \eqref{vartheta},
\begin{equation}\label{e11}
    \bm 1(\mathcal B)\big|\Tr \theta(H) - \Tr f_E(H)\big| \leq {\cal N}(E^+, E^++\Delta_\ell N^{-\delta_4})=\sum_{i}\bm 1\big(|\lambda_i-E^+|\leq\Delta_\ell N^{-\delta_4}\big) .
\end{equation}
Hence we have
\begin{align}\label{eqn:smoothed-chi-1}
\begin{aligned}
&\bm 1(\mathcal B)\left|\int_{I_\ell} x(E)\chi(E)\;\mathrm{d}E-\int_{I_\ell} x(E)\; q(\Tr f_E(H))\; \mathrm{d} E\right| \\
&=  \bm 1(\mathcal B)\left|\int_{I_\ell} x(E)\; \big[q\big(\Tr \theta(H)\big)-q\big(\Tr f_E(H)\big)\big]\; \mathrm{d} E\right|\\
&\leq C\bm 1(\mathcal B)\int_{I_\ell}\big|x(E)\big|\;\big|\Tr \theta(H) - \Tr f_E(H)\big|\; \mathrm{d} E\\
&\leq C\sum_{i}\int_{I_\ell}\big|x(E)\big|\bm 1\big(|\lambda_i-E^+|\leq \Delta_\ell N^{-\delta_4}\big)\,\d E\\
&\prec N^{-\delta_4/2}\Delta_\ell \sup_{E\in I_\ell}\big|x(E)\big|.
\end{aligned}
\end{align}
In the last step, we integrated in $E$ and used the rigidity estimate \eqref{eqn:rig} to show that only $\prec N^{\delta_2}$ eigenvalues contribute to the sum. By \eqref{eqn:que} and \eqref{eqn:small-remaining-terms} with $\epsilon=2\delta_2$, we have
\begin{align}\label{eqn:smoothed-chi-2}
\begin{split}
    \sup_{E\in I_\ell}|x(E)|\prec \Delta_{\ell}^{-1}N^{\delta_1+2\delta_2+\delta/2}.
\end{split}
\end{align}
Combining \eqref{eqn:smoothed-chi-1} and \eqref{eqn:smoothed-chi-2}, we obtain the desired bound on $\mathcal B$. On $\mathcal B^c$, we simply use the rigidity estimate \eqref{eqn:rig} and \eqref{eqn:smoothed-chi-2}. The proof is complete.
\end{proof}

\begin{lemma}\label{lem:reg-ob-3}
 Maintain the same assumptions as in Lemma \ref{lem:regularized-observable}. 
 We have
 %Using Definition \ref{def:regularized}, the following holds
    \begin{equation*}
        \E \left[g\left(\int_{I_\ell} x(E) q(\Tr f_E(H))\right)\right]-\E \left[g\left(\int_{I_\ell} x(E) q(y_E) \right)\right]=O(N^{-\epsilon_1/2}).
    \end{equation*}
\end{lemma}
\begin{proof}
    We first express $f_E(H)$ in terms of Green functions using the Helffer--Sj\"{o}strand functional calculus (see equation (B.12) of \cite{uni-sine-kernel}):
    \begin{equation*}
    f_E (\lambda) \;=\; \frac1{2\pi}\int_{\R^2}\frac{\mathrm{i} \sigma  f_E ''(e)\tilde f(\sigma)+ \mathrm{i}  f_E (e) \tilde f'(\sigma)-\sigma 
    f_E '(e)\tilde f'(\sigma)}{\lambda-e-\mathrm{i} \sigma} \; \mathrm{d} e \, \mathrm{d} \sigma.
\end{equation*}

Let $G(z) = (H-z)^{-1}$ and recall $\tilde \eta_\ell=\Delta_\ell N^{-\delta_5}$. Then we have
 \begin{align*}
 \begin{aligned}
    \Tr f_E (H) &= \frac{1}{2\pi}\int_{\R^2} \left(\mathrm{i} \sigma  f_E ''(e)\tilde f(\sigma)+ \mathrm{i}  f_E (e) \tilde f'(\sigma)-\sigma f_E '(e)\tilde f'(\sigma) \right) \Tr G(e+\mathrm{i} \sigma) \; \mathrm{d} e \, \mathrm{d} \sigma\\
    &= \frac{1}{2\pi}\int_{\R^2} \left(\mathrm{i}  f_E (e) \tilde f'(\sigma)-\sigma f_E '(e)\tilde f'(\sigma) \right) \Tr G(e+\mathrm{i} \sigma)\; \mathrm{d} e \, \mathrm{d} \sigma\\
    &\quad+\frac{1}{2\pi}\int_{|\sigma|>\tilde{\eta}_\ell}\int \mathrm{i} \sigma  f_E ''(e)\tilde f(\sigma) \Tr G(e+\mathrm{i} \sigma) \;\mathrm{d} e \, \mathrm{d} \sigma\\
    &\quad+\frac{1}{2\pi}\int_{|\sigma|<\tilde{\eta}_\ell}\int \mathrm{i} \sigma  f_E ''(e)\tilde f(\sigma) \Tr G(e+\mathrm{i} \sigma)\; \mathrm{d} e \, \mathrm{d} \sigma\\
    &= \frac{1}{2\pi}\int_{\R^2} \left(\mathrm{i}  f_E (e) \tilde f'(\sigma)-\sigma f_E '(e)\tilde f'(\sigma) \right) \Tr G(e+\mathrm{i} \sigma)\; \mathrm{d} e \, \mathrm{d} \sigma\\
    &\quad+\frac{1}{2\pi }\int_{|\sigma|>\tilde{\eta}_\ell}\int \mathrm{i} \sigma  f_E ''(e)\tilde f(\sigma) \Tr G(e+\mathrm{i} \sigma) \;\mathrm{d} e \, \mathrm{d} \sigma\\
    &\quad -\frac{1}{2\pi}\int_{|\sigma|<\tilde{\eta}_\ell}\int \sigma  f_E ''(e)\tilde f(\sigma) \Im\Tr G(e+\mathrm{i} \sigma)\; \mathrm{d} e \, \mathrm{d} \sigma,
 \end{aligned}
 \end{align*}
where in the last step we use the fact that the left-hand side is real.

We show that the last term is negligible. From {\cite[Lemma 5.1]{KnoYin13}}, we know $\sigma \Im\Tr G(e+\mathrm{i} \sigma) = O_\prec(1)$. 
 Since $\int |f_E ''(e) | = O(\Delta_\ell^{-1}N^{\delta_4})$ and $|\tilde f| \le 1$, the last term is bounded by
\begin{equation*}
    \left|\frac{1}{2 \pi}\int_{|\sigma|<\tilde{\eta}_\ell}\int  \sigma  f_E ''(e)\tilde f(\sigma) \Im\Tr G(e+\mathrm{i} \sigma)\; \mathrm{d} e \, \mathrm{d} \sigma\right| = O_\prec(N^{\delta_4-\delta_5}).
\end{equation*}

Hence by the mean value theorem applied to $q$, and the definition of $y_E$ (see \eqref{eqn:regularized-y}), we know 
\[q\big(\Tr f_E(H)\big)- q(y_E)=O_\prec(N^{\delta_4-\delta_5}).\]
Now by mean value theorem applied to $g$ and \eqref{eqn:bound-x-integral}, we have
\begin{align*}
\begin{split}
    \mathbb{E}\left[g\left(\int_{I_{\ell}} x(E) q\left(\operatorname{Tr} f_E(H)\right)\right)\right]-\mathbb{E}\left[g\left(\int_{I_{\ell}} x(E) q\left(y_E\right)\right)\right]=O_{\prec}\left(N^{\delta_2+\delta_4-\delta_5+\delta/2}\right),
\end{split}
\end{align*}
which completes the proof by the choice of parameters in Definition~\ref{regularizednotation}.
\end{proof}

%\newpage
\section{Two moment comparison and proof of the main theorem}\label{s:conclusion}

We retain the conventions stated in Section~\ref{s:conventions} and the choice of parameters made in \Cref{para}.
 \subsection{Preliminary Lemmas}

 For $z\in \mathbb{C}\backslash\R$, we define the control parameter
\begin{gather}\label{controlparameter}
\Psi(z)=\sqrt{\frac{\operatorname{Im} m_{\mathrm{sc}}(z)}{N \operatorname{Im} z}}+\frac{1}{N |\operatorname{Im} z|}.
\end{gather}

Let $H$ be an $N \times N$ Wigner matrix.
Fix $a,b\in\llbracket1,N\rrbracket$, and define 
$$Q = Q(a,b) = \{q_{ij}\}_{1 \le i,j \le N}\in\matn$$
as follows. Set $q_{ij} = h_{ij}$ if $(i,j) \notin \{ (a,b), (b,a) \}$, and set $q_{ij} = 0$ otherwise. In other words, $Q$ is the matrix obtained by starting with $H$ and replacing entries $h_{ab}$ and $h_{ba}$ by zeros. Given $z \in \mathbb{C} \setminus \mathbb{R}$, set 
\begin{equation}\label{GR}
G = (H - z)^{-1}, \qquad R = (Q - z )^{-1}.
\end{equation}
Let $x^G,x^R,y^G,y^R$ be the quantities in \eqref{eqn:regularized-x} and \eqref{eqn:regularized-y} defined using the resolvents $G$ or $R$, as indicated by the superscript. Finally, let $U=H-Q$.

We summarize some preliminary bounds in the following lemma.
\begin{lemma}\label{lem:prelim-bound}Let $H=\{h_{ij}\}_{i,j=1}^N,G,R,\Psi$ be as defined above. We have
\begin{align}
    |h_{ij}|&\prec N^{-1/2},\label{eqn:entry-bound}\\
    \|G(z)\|&\leq \frac{1}{\eta},\label{eqn:absolute-bound-G}\\
    \|R(z)\|&\leq \frac{1}{\eta},\label{eqn:absolute-bound-R}\\
    C^{-1}\tau^{1/4}N^{-1/2}&\leq\Psi(z)\leq C\tau^{-1/4}N^{-\tau/20},\quad\forall\,z\in\bm S,\label{eqn:psi-bound}
\end{align}
for some constant $C>0$, where $\eta=|\Im z|$.
Moreover, when $z=E+\iu\eta_\ell$ with \[E\in [\gamma_\ell-2\Delta_\ell N^{\delta_2}, \gamma_\ell+2 \Delta_\ell N^{\delta_2}],\] there exists constant $C>0$ such that
\begin{align}\label{eqn:psi-bound-specified}
\begin{split}
    \Psi(z)\leq \frac{C}{N\eta_\ell}.
\end{split}
\end{align}
Further, the analogous claim holds with $\tilde \eta_\ell$ replacing $\eta_\ell$. 
\end{lemma}
\begin{remark}
The interval $[\gamma_\ell-2\Delta_\ell N^{\delta_2}, \gamma_\ell+2 \Delta_\ell N^{\delta_2}]$ is a slightly enlarged version of the interval $I_\ell$ from Definition \ref{def:regularized}.
\end{remark}
\begin{proof}
    Fix any $\epsilon>0,D>0$. By Markov's inequality and the moment assumption \eqref{finite-moment} on $h_{ij}$, we have
    \begin{equation*}
        \P\left(|\sqrt{N}h_{ij}|>N^\epsilon \right)=\P\left(|\sqrt Nh_{ij}|^p>N^{p\epsilon}\right)\leq \mu_pN^{-p\epsilon}\leq N^{-D},
    \end{equation*}
    for large enough $p$ and $N>N_0(\epsilon,D)$. This proves \eqref{eqn:entry-bound}.

    By spectral decomposition, we have
    \begin{equation*}
        G(z)=\sum_i\frac{\bm u_i\bm u_i^\trans}{\lambda_i-z},
    \end{equation*}
    where $\{\lambda_i\}_{i=1}^N$ and $\{\bm u_i\}_{i=1}^N$ are the corresponding eigenvalues and (unit) eigenvectors of $H$. Observe that 
    \begin{equation*}
        \left|\frac{1}{\lambda_i-z}\right|\leq \frac{1}{\eta}.
    \end{equation*}
     Pick any unit vectors $\bm x, \bm y$. We have 
    \begin{equation*}
        |\bm x^\trans G\bm y|\leq \frac{1}{\eta}\bm \sum_i\bm x^\trans \bm u_i\bm u_i^\trans \bm y\le \frac{1}{2\eta} \sum_i (\bm x^\trans \bm u_i)^2 + (\bm u_i^\trans \bm y)^2 \le \frac{1}{\eta}.
    \end{equation*}
    This proves \eqref{eqn:absolute-bound-G}. The inequality \eqref{eqn:absolute-bound-R} follows similarly.

    To obtain the lower bound of $\Psi(z)$, we use the following result (see \cite[Lemma~3.3]{benaych2016lectures}):
    \begin{equation}\label{eqn:bound-im-m}
        \begin{aligned}
        &c\sqrt{\kappa+\eta}\leq|\operatorname{Im} \m(z)|\leq c^{-1}\sqrt{\kappa+\eta} ,&\text{ if }|E|\leq 2,\\
        &\frac{c\eta}{\sqrt{\kappa+\eta}}\leq|\Im \m(z)|\leq\frac{c^{-1}\eta}{\sqrt{\kappa+\eta}},&\text{ if }|E|\geq 2,
        \end{aligned}
    \end{equation}
    for some constant $c>0$, where $E=\Re z$ and $\kappa\equiv\kappa(E)\defeq ||E|-2|$. 

    For $|E|\leq 2$, we have 
    \begin{equation*}
    \begin{aligned}
        &\Psi(z)\geq \sqrt{\frac{c\sqrt{\kappa+\eta}}{N\eta}}+\frac{1}{N\eta}\geq CN^{-1/2}\eta^{-1/4}\geq C\tau^{1/4}N^{-1/2},\\
        &\Psi(z)\leq \sqrt{\frac{c^{-1}\sqrt{\kappa+\eta}}{N\eta}}+\frac{1}{N\eta}\leq C\sqrt{\frac{\tau^{-1/2}}{NN^{-1+\tau/10}}}+\frac{1}{NN^{-1+\tau/10}}\leq C\tau^{-1/4}N^{-\tau/20},
    \end{aligned} 
    \end{equation*}
    where we used $z\in\bm S$ in the last step of the first line and in the second inequality of the second line.

    For $|E|\geq 2$, we have
    \begin{equation*}
    \begin{aligned}
        &\Psi(z)\geq\sqrt{\frac{c}{\sqrt{\kappa+\eta}N}}+\frac{1}{N\eta}\geq\left\{ \begin{aligned}
            CN^{-1/2}\eta^{-1/4}&\text{ if }\kappa\leq \eta\\
            CN^{-1/2}\kappa^{-1/4}&\text{ if }\kappa\geq \eta
        \end{aligned}\right\}\geq C\tau^{1/4}N^{-1/2},\\
        &\Psi(z)\leq \sqrt{\frac{c^{-1}}{\sqrt{\kappa+\eta}N}}+\frac{1}{N\eta}\leq C\sqrt{\frac{1}{N^{-1/2+\tau/20}N}}+\frac{1}{NN^{-1+\tau/10}}\leq CN^{-1/2-\tau/20},
    \end{aligned}
    \end{equation*}
    where we used $z\in\bm S$ in the last step of the first line and in the second inequality of the second line. This completes the proof of \eqref{eqn:psi-bound}.

    Finally, for $z=E+\iu\eta_\ell$ with $E\in I_\ell$, we have
    \begin{align*}
    \begin{split}
        \Psi(z)\leq C\sqrt{\frac{\sqrt{\kappa+\eta_\ell}}{N\eta_\ell}}+\frac{1}{N\eta_\ell}.
    \end{split}
    \end{align*}
    Note that $\kappa\leq C(\ell/N)^{2/3}+N^{\delta_2}\Delta_\ell$. It is not hard to check that $\kappa+\eta_\ell\leq C(N\eta_\ell)^{-2}$. This completes the proof of \eqref{eqn:psi-bound-specified}.
\end{proof}

In the next lemma, we collect several local laws, which will be used frequently in the current section and Section \ref{sec:poly}.
\begin{lemma}\label{lem:three-resolvent-local-law}
Let $H$ be a Wigner matrix, and let $S$ be either $G$ or $R$, as defined in \eqref{GR}.
\begin{enumerate}
\item For all $z \in \bm S$, we have 
\begin{equation}
\big|\langle\bm x,S\bm y\rangle- \langle\bm x,\bm y\rangle \m\big|\prec\Psi,\quad 
|(SS)_{\bm x\bm y}|\prec N\Psi^2,\quad |(S\bar S)_{\x\y}|\prec N\Psi^2\label{eqn:2-resolvents}
\end{equation}
uniformly over all $\bm x, \bm y \in \mathbb{S}^{N-1}$.
\item If $z=E+\iu\eta \in \bm S$ satisfies
$
    E\in I_\ell$ and $\eta=\eta_\ell$,
then for any deterministic $A \in \matn$ such that $\|A\|\leq 1$ and $\Tr A=0$, we have
\begin{align}
|(SA\bar S)_{cd}|&\prec N^{1/2}\Psi,\label{eqn:2-resolvents-traceless}\\
|(SA\bar SS)_{cd}|&\prec N^{3/2}\Psi^{9/4},\label{eqn:3-resolvents-traceless}
\end{align}
uniformly over all $c,d\in\llbracket1,N\rrbracket$.
\end{enumerate}
\end{lemma}
The proof of Lemma \ref{lem:three-resolvent-local-law} is postponed to Appendix \ref{sec:local-law}.

The resolvent expansion formula
\begin{align}\label{eqn:resolvent-expansion-original}
\begin{split}
G&=R-R U R+R U R U R-R U R U R U R+(R U)^4 G
\end{split}
\end{align}
follows immediately from the definitions of $G$, $R$, and $U$. 
It implies 
%From \eqref{eqn:resolvent-expansion-original}, we have
\begin{align}\label{eqn:resolvent-expansion}
\begin{split}
\begin{aligned}
G   A\bar{G}-R   A\bar{R}= & -R  A \overline{R U R}-R U R  A \bar{R} \\
& +R   A\overline{R U R U R}+R U R   A\overline{R U R}+R U R U R   A\bar{R} \\
& -R   A\overline{R U R U R U R}-R U R   A\overline{R U R U R}-R U R U R   A\overline{R U R}-R U R U R U R   A\bar{R} \\
& +R   A\overline{(R U)^4 G}+R U R   A\overline{(R U)^3 R}+(R U)^2 R   A\overline{(R U)^2 R}+(R U)^3 R   A\overline{R U R}+(R U)^4 G   A\bar{R}.
\end{aligned}
\end{split}
\end{align}
These identities facilitate resolvent expansions for the terms $x^G$  and $y^G$, which are stated in the following lemma. We recall that $\ell$ was fixed earlier in Section~\ref{s:conventions} and that $I_\ell$ was defined in \eqref{regularizednotation}.
\begin{lemma}\label{lem:resolvent-expansion}
Let $H$ be an $N\times N$ Wigner matrix. 
\begin{enumerate}
\item We have %the following resolvent expansion for $x^G-x^R$:
\begin{equation*}
x^G-x^R=\sum_{r=1}^3 x_r h_{a b}^r+x_{\mathrm{err}},
\end{equation*}
where
\begin{equation*}
 \left|x_i(E)\right| \prec N^{1+\delta/2}\Psi(E + \iu \eta_\ell)^{5/4}, \quad i=1,2,3, \quad
 \left|x_{\mathrm{err}}(E)\right| \prec N^{-1+\delta/2}\Psi(E + \iu \eta_\ell)^{5/4},
\end{equation*}
uniformly over $E\in I_\ell$. Moreover 
\begin{align}\label{xRbound}
\begin{split}
    |x^R(E)|\prec N^{1+\delta/2}\Psi(E + \iu \eta_\ell)^{1/2},
\end{split}
\end{align}
uniformly over $E\in I_\ell$.
\item We have
\begin{equation*}
\operatorname{Tr} G-\operatorname{Tr} R=\sum_{r=1}^3 J_r h_{a b}^r+J_{\mathrm{err}}
\end{equation*}
where
\begin{equation}\label{theji}
%\begin{aligned}
|J_i(z)|\prec N \Psi(z )^2, \quad i=1,2,3, \quad
 |J_{\mathrm{err}}(z)|\prec N^{-1} \Psi(z)^2 ,
%\end{aligned}
\end{equation}
uniformly for $z \in \bm S$.
\item We have
\begin{equation*}
y^G-y^R=\sum_{r=1}^3 y_r h_{a b}^r+y_{\mathrm{err}}
\end{equation*}
where
\begin{equation*}
%\begin{aligned}
 \left|y_i(E)\right| \prec N^{6 \epsilon_1} \Psi(E + \iu \eta_\ell), \quad i=1,2,3, \quad
 \left|y_{\mathrm{err}}(E)\right|\prec N^{-2  + 6 \epsilon_1}\Psi(E + \iu \eta_\ell),
%\end{aligned}
\end{equation*}
uniformly over $E\in I_\ell$.
\end{enumerate}
\end{lemma}
\begin{proof}
We first present the proofs for the claims about $x_1$ and $x^R$. The others are similar. 

By \eqref{eqn:regularized-x} and \eqref{eqn:resolvent-expansion}, we have
\begin{align*}
\begin{split}
    x_1=\frac{\eta_\ell}{\pi}\sqrt{\frac{N^2}{2\Tr(A^2)}}\frac{1}{1+\delta_{ab}}[(RA\bar RR)_{ab}+(RA\bar RR)_{ba}]+[C],
\end{split}
\end{align*}
where $[C]$ denotes the complex conjugate of the previous terms. 
By \eqref{eqn:psi-bound-specified}, Lemma \ref{lem:three-resolvent-local-law}, and the assumption $\Tr(A^2)\geq  N^{1-\delta}$, we have that $|x_1|\prec N^{1+\delta/2}\Psi^{5/4}$.

For $x^R$ we have
\begin{equation*}
    x^R=\frac{\eta_\ell}{\pi}\sqrt{\frac{N^2}{2\Tr(A^2)}}\Tr(RA\bar R).
\end{equation*}
By definition, we have $M(z,A,\bar z)=|\m(z)|^2A$, which is traceless. From \eqref{eqn:avg} with $A_1 = A$ and $A_2 = I$, we have
%\begin{align}
%\begin{split}
    $|x^R|\prec N^{1+\delta/2}\Psi^{1/2}.$
%\end{split}
%\end{align}

By the resolvent expansion formula \eqref{eqn:resolvent-expansion-original}, we have 
\begin{align*}
\begin{split}
    \operatorname{Tr} G -\operatorname{Tr} R&=-\operatorname{Tr} R U R+\operatorname{Tr} R U R U R-\operatorname{Tr} R U R U R U R+\operatorname{Tr}(R U)^4 G=:\sum_{i=1}^3 J_i h_{ab}^i +J_\err.
\end{split}
\end{align*}
% Therefore,
% neglecting lower order terms, we have
% \begin{equation*}
% \begin{aligned}
% & J_1=-2\left(R^2\right)_{a b}, \quad J_2\approx\left(R^2\right)_{a a} R_{b b}+\left(R^2\right)_{b b} R_{a a}, \quad J_3\approx-2\left(R^2\right)_{a b} R_{a a} R_{b b}, \\
% & J_{\mathrm{err}}\approx\left[(G R)_{a a}\left(R_{b b}\right)^2 R_{a a}+(G R)_{b b}\left(R_{a a}\right)^2 R_{b b}\right] h_{a b}^4
% \end{aligned}
% \end{equation*}
Using the first and second high probability bounds in \eqref{eqn:2-resolvents} together with \eqref{eqn:entry-bound}, we have 
\begin{equation}\label{secondpartJ}
\begin{aligned}
 |J_i|\prec N \Psi^2, \quad i=1,2,3, \quad{\text{and}}\quad
 |J_{\mathrm{err}}|\prec N^{-1}\Psi^2 .
\end{aligned}
\end{equation}
For example, for $J_1$, we have 
\begin{equation*}
\operatorname{Tr} RUR = 2(RR)_{ab}h_{ab}
%R_{aa} (RR)_{bb} + 2R_{ab} (RR)_{ab} + R_{bb} (RR)_{aa},
\end{equation*}
and $(RR)_{jk} \prec N \Psi^2$ for $j,k \in \{a,b\}$, by \eqref{lem:three-resolvent-local-law}. The bounds for $J_2$ and $J_3$ are similar. For $J_\err$, we additionally use \eqref{eqn:entry-bound} to gain a factor of $N^{-2}$ from the expectation of $h^4_{ab}$. 
For example, one term arising in $J_\err$ (which has a leading-order contribution) is 
\[
(GR)_{ab} R_{aa}^2 R_{bb}^2 h_{ab}^2,
\]
and we use $(GR)_{ab} \prec N\Psi^2$, $R_{aa} \prec 1$, $R_{bb} \prec 1$, and $ h_{ab}^2 \prec N^{-2}$. The bound on $(GR)_{ab}$ comes from noting that
%We also note that
\[
(RG)_{jk} = (RR)_{jk}  + \big(R (G - R) \big)_{jk},
\]
and bounding the second term by using \eqref{eqn:resolvent-expansion-original} to expand  $G-R$.

By definition,
\begin{align}
\begin{aligned}\label{theyi}
y_i= & \frac{1}{2 \pi} \int_{\mathbb{R}^2} \mathrm{i} \sigma f_E^{\prime \prime}(e) \widetilde{f}(\sigma) J_i(e+\iu\sigma) \mathbf{1}\left(|\sigma|>\widetilde{\eta}_{\ell}\right) \mathrm{d} e \, \mathrm{d} \sigma \\
& +\frac{1}{2 \pi} \int_{\mathbb{R}^2}\left(\mathrm{i} f_E(e) \widetilde{f}^{\prime}(\sigma)-\sigma f_E^{\prime}(e) \widetilde{f}^{\prime}(\sigma)\right) J_i(e+\iu\sigma)\, \mathrm{d} e \, \mathrm{d} \sigma,\\
y_\err=& \frac{1}{2 \pi} \int_{\mathbb{R}^2} \mathrm{i} \sigma f_E^{\prime \prime}(e) \widetilde{f}(\sigma) J_\err(e+\iu\sigma) \mathbf{1}\left(|\sigma|>\widetilde{\eta}_{\ell}\right) \mathrm{d} e \, \mathrm{d} \sigma \\
& +\frac{1}{2 \pi} \int_{\mathbb{R}^2}\left(\mathrm{i} f_E(e) \widetilde{f}^{\prime}(\sigma)-\sigma f_E^{\prime}(e) \widetilde{f}^{\prime}(\sigma)\right) J_\err(e+\iu\sigma)\, \mathrm{d} e \, \mathrm{d} \sigma,
\end{aligned}
\end{align}
Then using the bound on $J_i$ from \eqref{secondpartJ},
the proof of the bounds on $y_i$ and $y_{\err}$ follows from the proof of \cite[Lemma~7.7]{BloKnoYauYin16}. For completeness, we give the details here.\footnote{The derivation of \cite[Lemma~7.7]{BloKnoYauYin16} appears to contain a misprint. The indicator function $f_E$ there is supported on an interval with constant length, which seems too large to obtain the indicated bounds. We therefore define $f_E$ as in \eqref{vartheta} instead.} %\beb Without loss of generality, we assume that $\ell\in  \llbracket 1, N^{1-\tau} \rrbracket$.\eeb

We first consider the $y_i$ and begin with the first term in the expression for $y_i$ in \eqref{theyi}. We integrate by parts in $e$, use the Cauchy--Riemann equations in the form
\[
\partial_e J_i(e + \iu \sigma) = - \iu \partial_\sigma J_i(e + \iu \sigma),
\]
and then integrate by parts again in $\sigma$. 
This yields 
\begin{align}
& \int_{\mathbb{R}^2} \mathrm{i} \sigma f_E^{\prime \prime}(e) \widetilde{f}(\sigma) J_i(e+\iu\sigma) \mathbf{1}\left(|\sigma|>\widetilde{\eta}_{\ell}\right) \mathrm{d} e \, \mathrm{d} \sigma\notag \\
&=  -  \int_{\mathbb{R}^2} \mathrm{i} \sigma f_E^{\prime }(e) \widetilde{f}(\sigma) \partial_e J_i(e+\iu\sigma) \mathbf{1}\left(|\sigma|>\widetilde{\eta}_{\ell}\right) \mathrm{d} e \, \mathrm{d} \sigma\notag \\
& =  -   \int_{\mathbb{R}^2}  \sigma f_E^{\prime }(e) \widetilde{f}(\sigma) \partial_\sigma J_i(e+\iu\sigma) \mathbf{1}\left(|\sigma|>\widetilde{\eta}_{\ell}\right) \mathrm{d} e \, \mathrm{d} \sigma\notag \\
& =  \int_{\mathbb{R}^2}  \big(\sigma \widetilde{f}'(\sigma)  + \widetilde{f}(\sigma) \big) f_E^{\prime }(e)  J_i(e+\iu\sigma) \mathbf{1}\left(|\sigma|>\widetilde{\eta}_{\ell}\right) \mathrm{d} e \, \mathrm{d} \sigma %\\
%& \quad 
+ \sum_{\pm} \mp \int_{\R} \tilde \eta_\ell 
 f_E^{\prime }(e) \widetilde{f}(\pm \eta_\ell ) J_i(e \pm \iu\eta_\ell)  \,  \mathrm{d} e \label{newterms}
\end{align}
For the first term in \eqref{newterms}, we use  \eqref{theji},  $|\widetilde{f}(\sigma)|\le 1$, $| \sigma \widetilde{f}'(\sigma)| \le 2$, and the definition of $f_E'(e)$,  %\eqref{eqn:psi-bound-specified}, \eqref{eqn:bound-im-m} in the region $|E| \ge 2$, and the fact that $\Psi$ is decreasing for $\eta > 0$ 
to get 
\begin{comment}
\begin{align}
\begin{split}\label{split1}
 &\left| \int_{\mathbb{R}^2}  \big(\sigma \widetilde{f}'(\sigma)  + \widetilde{f}(\sigma) \big) f_E^{\prime }(e)  J_i(e+\iu\sigma) \mathbf{1}\left(|\sigma|>\widetilde{\eta}_{\ell}\right) \mathrm{d} e \, \mathrm{d} \sigma \right| \\
  & \le 
  2 \int_{\R} \int_{\tilde \eta_\ell}^\infty \big( | \sigma \widetilde{f}'(\sigma)|  +  |\widetilde{f}(\sigma)| \big) |f_E^{\prime }(e) |
  \big| J_i(e+\iu\sigma)\big| \, \mathrm{d} \sigma \, \mathrm{d} e \\
   & \le 
  6 N \int_{\R} \int_{\tilde \eta_\ell}^\varpi \big|f_E^{\prime }(e)\big| 
  \Psi^2(e + \iu \sigma)   \, \mathrm{d} \sigma \, \mathrm{d} e \\
  & \prec 
    N \eta_\ell \cdot \frac{1}{(N \tilde \eta_\ell)^2} \int_{E^+}^{E^+ + \nu}  |f_E^{\prime }(e) |
     \,  \mathrm{d} e 
   + N \eta_\ell \cdot N^{-2/3} \int_{\vartheta - \nu}^{\vartheta} |f_E^{\prime }(e) |
    \, \mathrm{d} e \\ 
    & \le  
    N \eta_\ell \left( \frac{1}{(N \tilde \eta_\ell)^2} 
    + N^{-2/3} \right) 
    %N \eta_\ell \int_{\vartheta - \nu}^{\vartheta}  |f_E^{\prime }(e) |
   %\Psi^2  \, \mathrm{d} e 
   \\
   & \le \frac{ 2 N^{2(\delta_5 - \delta_1)}}{N \eta_\ell}\\
   &\le N^{2 ( \delta_5 - \delta_1)} \Psi.
   \end{split}
\end{align}
\end{comment}
\begin{align}
\begin{split}\label{split1}
 &\left| \int_{\mathbb{R}^2}  \big(\sigma \widetilde{f}'(\sigma)  + \widetilde{f}(\sigma) \big) f_E^{\prime }(e)  J_i(e+\iu\sigma) \mathbf{1}\left(|\sigma|>\widetilde{\eta}_{\ell}\right) \mathrm{d} e \, \mathrm{d} \sigma \right| \\
  & \le 
  2 \int_{\R} \int_{\tilde \eta_\ell}^\infty \big( | \sigma \widetilde{f}'(\sigma)|  +  |\widetilde{f}(\sigma)| \big) |f_E^{\prime }(e) |
  \big| J_i(e+\iu\sigma)\big| \, \mathrm{d} \sigma \, \mathrm{d} e \\
   & \le 
  6 N \int_{\R} \int_{\tilde \eta_\ell}^{2\varpi} \big|f_E^{\prime }(e)\big| 
  \Psi^2(e + \iu \sigma)   \, \mathrm{d} \sigma \, \mathrm{d} e \\
  & \le 
    6 N \left(  \int_{E^+}^{E^+ + \nu} \int_{\tilde \eta_\ell}^{2\varpi} |f_E^{\prime }(e) | \Psi^2(e + \iu \sigma)   \, \mathrm{d} \sigma
     \,  \mathrm{d} e 
   +  \int_{\vartheta - \nu}^{\vartheta} \int_{\tilde \eta_\ell}^1 |f_E^{\prime }(e) | \Psi^2(e + \iu \sigma)   \, \mathrm{d} \sigma
    \, \mathrm{d} e \right).%\\ 
   % & \le  
   % N \eta_\ell \left( \frac{1}{(N \tilde \eta_\ell)^2} 
    %+ N^{-2/3} \right) 
    %N \eta_\ell \int_{\vartheta - \nu}^{\vartheta}  |f_E^{\prime }(e) |
   %\Psi^2  \, \mathrm{d} e 
  % \\
   %& \le \frac{ 2 N^{2(\delta_5 - \delta_1)}}{N \eta_\ell}\\
  % &\le N^{2 ( \delta_5 - \delta_1)} \Psi.
   \end{split}
\end{align}
Using \eqref{eqn:bound-im-m}, note that
\begin{equation}\label{e:418bound}
\Psi^2(e + \iu \sigma) \le C \left( \frac{ \sqrt{\sigma} + \sqrt{\kappa}}{N\sigma}
+ \frac{1}{N^2\sigma^2}\right).
\end{equation}
We insert this bound for $\Psi^2$ into \eqref{split1} and bound the resulting terms. We begin with 
\begin{align}\label{e:combine1}
\begin{split}
%6 N \left(  \int_{E^+}^{E^+ + \nu} \int_{\tilde \eta_\ell}^1 |f_E^{\prime }(e) | \Psi^2(e + \iu \sigma)   \, \mathrm{d} \sigma
 %    \,  \mathrm{d} e 
%   +  \int_{\vartheta - \nu}^{\vartheta} \int_{\tilde \eta_\ell}^1 |f_E^{\prime }(e) | \Psi^2(e + \iu \sigma)   \, \mathrm{d} \sigma
%    \, \mathrm{d} e \right)\\
& N
\left(  \int_{E^+}^{E^+ + \nu} \int_{\tilde \eta_\ell}^{2\varpi} |f_E^{\prime }(e) | \frac{1}{N^2\sigma^2}   \, \mathrm{d} \sigma
     \,  \mathrm{d} e 
   +  \int_{\vartheta - \nu}^{\vartheta} \int_{\tilde \eta_\ell}^{2\varpi} |f_E^{\prime }(e) | \frac{1}{N^2\sigma^2}  \, \mathrm{d} \sigma
    \, \mathrm{d} e \right)\\
    &\le C (N \tilde \eta_\ell)^{-1}\\
    &= C N^{6 \epsilon_1}  (N  \eta_\ell)^{-1}\\
    &\le N^{6 \epsilon_1}  \Psi( E + \iu \eta_\ell),
 \end{split}
\end{align}

where the last inequality follows from the definition of $\Psi$ in \eqref{controlparameter}. 
Using $\kappa(e) \le 2 N^{\delta_2}\varpi$ for $e \in [E^+,  E^+ + \nu]$ and $\varpi^{1/2} \le (N\eta_\ell)^{-1}$, we get
\begin{align}
\begin{split}
\label{e:combine2}
&N \int_{E^+}^{E^+ + \nu} \int_{\tilde \eta_\ell}^{2\varpi} |f_E^{\prime }(e) | \left(\frac{\sqrt{\sigma} + \sqrt{\kappa} }{N\sigma} \right)  \, \mathrm{d} \sigma
     \,  \mathrm{d} e \\
&\le
 \int_{\tilde \eta_\ell}^{2\varpi} \sigma^{-1/2} \, \mathrm{d} \sigma  +  2 \int_{\tilde \eta_\ell}^{2\varpi} \frac{N^{\delta_2/2 } \sqrt{\varpi}}{\sigma} \, \mathrm{d} \sigma\\
 &\le \varpi^{1/2} +  N^{\delta_2/2}\varpi^{1/2} \big( \log(\varpi) + \log (\tilde \eta_\ell^{-1}) \big)  \\
 &\le  C N^{\delta_2/2 } \log (N) (N\eta_\ell)^{-1}\\
 &\le N^{\delta_2} \Psi(E + \iu \eta_\ell).
\end{split}
\end{align}
%where we used $\ell \le N^{1-\tau}$ and the definition of $\eta_\ell$ in the last line. 
We further have 
\begin{align}
\begin{split}
\label{e:combine3}
& N  \int_{\vartheta - \nu}^{\vartheta} \int_{\tilde \eta_\ell}^{2\varpi} |f_E^{\prime }(e) | \left(\frac{\sqrt{\sigma} + \sqrt{\kappa} }{N\sigma} \right)  \, \mathrm{d} \sigma
    \, \mathrm{d} e \\
&\le N   \int_{\tilde \eta_\ell}^{2\varpi} \left(\frac{\sqrt{\sigma} + N^{-1/3 + \delta_1/2} }{N\sigma} \right)  \, \mathrm{d} \sigma \\
&\le \varpi^{1/2}  + N^{-1/3 + \delta_1/2} \log(N)\\
&\le C (N \eta_\ell)^{-1} \\
&\le C \Psi(E + \iu \eta_\ell).
\end{split}
\end{align}

For the second term in \eqref{newterms}, we note using  similar reasoning to the first term that
\begin{align}
\begin{split}\label{split2}
& \left| \int_{\R} \tilde \eta_\ell 
 f_E^{\prime }(e) \widetilde{f}( \eta_\ell ) J_i(e + \iu\eta_\ell)  \,  \mathrm{d} e \right| \\ &\le 
  N \tilde \eta_\ell \left( \int_{E^+}^{E^+ + \nu} | f_E^{\prime }(e)| \Psi^2( E + \iu \eta_\ell) \, \mathrm{d} e
  +  \int_{\vartheta - \nu}^{\vartheta} | f_E^{\prime }(e)| \Psi^2( E + \iu \eta_\ell) \, \mathrm{d} e\right) 
 \\
 &\le  C N \tilde \eta_\ell   \cdot \frac{1}{(N \eta_\ell)^2} \\
 & \le C \Psi.
 \end{split}
\end{align}
The other term in the summation is  bounded the same way.
\eeb

Next, we consider the second term in the expression for $y_i$ in \eqref{theyi}. %Using \eqref{eqn:bound-im-m}, note that
%\[
%\Psi^2(z) \le C \left( \frac{ \sqrt{\eta} + \sqrt{\kappa}}{N\eta}
%+ \frac{1}{N^2\eta^2}\right).
%\]
For the first part of the integrand,  using \eqref{e:418bound} and $\kappa(e) \le 2 N^{\delta_1}\varpi$ for $e \in [\vartheta-\nu,  E^+ + \nu]$  (to control both the $\kappa$ in the bound for $\Psi^2$ and the size of the interval of integration in $e$), we have 
\begin{align}
\begin{split}\label{split3}
\left|
\int_{\mathbb{R}^2} \mathrm{i} f_E(e) \widetilde{f}^{\prime}(\sigma)  J_i(e+\iu\sigma)\, \mathrm{d} e \, \mathrm{d} \sigma\right| %\\
&\le N \int_{\mathbb{R}^2}  f_E(e) \big|\widetilde{f}^{\prime}(\sigma) \big| \Psi^2(e+\iu\sigma) \, \mathrm{d} e \, \mathrm{d} \sigma \\
&\le 2 N\int_\varpi^{2\varpi} \int_{\vartheta-\nu}^{E^+ + \nu}  \big|\widetilde{f}^{\prime}(\sigma) \big| \Psi^2(e+\iu\sigma)\, \mathrm{d} e \, \mathrm{d} \sigma \\
&\le C N^{1+\delta_1}\varpi
\left( \frac{ N^{\delta_1} \sqrt{\varpi} }{N\varpi}
+ \frac{1}{N^2\varpi^2}\right)\\
&\le  C N^{ 2\delta_1}\sqrt{\varpi}\\
& \le C N^{ 2\delta_1} \Psi(E + \iu \eta_\ell).
\end{split}
\end{align}
Similarly,
\begin{align}\begin{split}\label{split4}
\left|
\int_{\mathbb{R}^2}\sigma f_E^{\prime}(e) \widetilde{f}^{\prime}(\sigma)  J_i(e+\iu\sigma)\, \mathrm{d} e \, \mathrm{d} \sigma
\right| %\\
&\le N \int_{\mathbb{R}^2}  |f_E^{\prime}(e) | \big|\sigma \widetilde{f}^{\prime}(\sigma) \big| \Psi^2(e+\iu\sigma) \, \mathrm{d} e \, \mathrm{d} \sigma \\
 & \le C N^{1+\delta_2}\varpi
\left( \frac{ N^{\delta_1} \sqrt{\varpi} }{N\varpi}
+ \frac{1}{N^2\varpi^2}\right)\\
& \le N^{2 \delta_1} \Psi(E + \iu \eta_\ell).
\end{split}
\end{align}
Combining \eqref{split1}, \eqref{e:combine1}, \eqref{e:combine2}, \eqref{e:combine3}, \eqref{split2}, \eqref{split3}, and \eqref{split4}, and using the definition of  $\delta_1$ in \eqref{deltachoices}, we obtain the desired conclusion.

The argument for $y_{\mathrm{err}}(E)$ is essentially the same as for the $y_i(E)$ (using the second bound in \eqref{theji}), so we omit it.
\end{proof}

Using \eqref{eqn:entry-bound}, \eqref{eqn:psi-bound}, and Lemma \ref{lem:resolvent-expansion}, and a Taylor expansion of $q(y^G)$ around $y^R$ up to order 3, we have
\begin{align}\label{resolvent-expansion-integral}
\begin{split}
    &\int_{I_\ell}x^Gq(y^G)\, \d E-\int_{I_\ell} x^Rq(y^R)\, \d E\\
    &=\int_{I_\ell}\left(x^R+\sum_{r=1}^3 x_r h_{a b}^r+x_{\mathrm{err}}\right)\cdot
    \left(q(y^R)+\sum_{k=1}^3\frac{q^{(k)}(y^R)}{k!}\left(\sum_{r=1}^3y_rh_{ab}^r+y_\err\right)^k\hspace{-.5em}+ O_\prec(N^{-2+24\epsilon_1}\Psi^4)
    \right) \d E\\
    &\quad -\int_{I_\ell}x^Rq(y^R)\, \d E\\
    &=\sum_{\bm{l} \in \mathcal{L}} P_{\bm l} h_{a b}^{|\bm l|}+O_\prec(N^{-2-c}),
\end{split}
\end{align}
where we define
\begin{equation}
\begin{aligned}
 \mathcal{L}&\coloneqq\left\{\bm{l}=\left(l_0, \ldots, l_k\right) \in \llbracket 0,3 \rrbracket \times \llbracket 1,3 \rrbracket^k: 0 \le k \le 3 , 1 \leqslant|\bm{l}| \leqslant 3\right\} \setminus \{(0,0) \},\\
|\bm{l}|&\coloneqq\sum_{i=0}^k l_i, \\
 P_{\bm{l}}&\coloneqq\int_{I_\ell} \frac{q^{(k)}\left(y^R\right)}{k !} x_{l_0} y_{l_1} \cdots y_{l_k} \mathrm{~d} E,
\end{aligned}
\end{equation}
and $c>0$ depends only on $\tau$. Here we denote $x_0\defeq x^R$.  For the error term in the last equality, we used \eqref{eqn:psi-bound-specified} and \eqref{xRbound} to show that
\begin{align*}
\int_{I_\ell} \Psi(E + \iu \eta_\ell)^{4}  \big|x^R (E) \big|\,  \d E &\le N^{-2/3+ \delta_2} \ell^{-1/3} \cdot N^{1 + \delta/2} \sup_{E \in I_\ell} \Psi^{9/2}(E + \iu \eta_\ell) \\
& \le \Delta_\ell  N^{ 1 + \delta_2 + \delta /2} ( N \eta_\ell)^{-9/2}\\
&\le \Delta_\ell^{-7/2} N^{-7/2} N^{9 \delta_1/2 + \delta_2 + \delta/2}\\
&\le N^{-c}.
\end{align*}
The error terms involving the products of the $x_i$ and $x_\err$ with terms in the expansion of $q(y^G)$ are handled similarly.

Using \eqref{resolvent-expansion-integral}, for smooth and compactly supported $g$, we have
\begin{multline}\label{eqn:resolvent-expansion-integral-poly}
 \mathbb{E} \left[g\left(\int_{I_\ell} x^G q\left(y^G\right) \mathrm{d} E\right)\right]-\mathbb{E} \left[g\left(\int_{I_\ell} x^R q\left(y^R\right) \mathrm{d} E\right)\right]
 =\mathbb{E} \left[\mathcal{A}\right]+O_{\prec}\left(N^{-2-c}\right) \\
 \quad+\mathbb{E} \left[h_{a b}^3\right] \mathbb{E} \left[\sum_{k=1}^3 \frac{1}{k !} g^{(k)}\left(P_{\bm{0}}\right) \sum_{\bm{l}_1, \ldots, \bm l_k \in \mathcal{L}} \bm{1}\left(\sum_{i=1}^k\left|\bm{l}_i\right|=3\right) \prod_{i=1}^k P_{\bm{l}_i}\right],
\end{multline}
where $P_{\bm 0}\defeq \int_{I_\ell}x^Rq(y^R)~\d E$, and $\E[\mathcal A]$ depends only on $R$ and the first two moments of $h_{ab}$.

We need the following lemma, which is proved in Section \ref{sec:poly}.
\begin{lemma}\label{lem:third-moment-terms}
There is a constant $c(\tau)>0$ such that the following holds.
Fix indices $a,b$ such that $a\neq b$, and let $Y$ denote any of the following terms:
\begin{equation}\label{eqn:third-moment-terms}
\begin{aligned}
& g^{(3)}\left(P_{\bm{0}}\right) P_{(1)}^m P_{(0,1)}^n \quad(m+n=3), \\
& g^{(2)}\left(P_{\bm{0}}\right)\left(P_{(2)}+P_{(0,2)}+P_{(1,1)}+P_{(0,1,1)}\right)\left(P_{(1)}+P_{(0,1)}\right), \\
& g^{(1)}\left(P_{\bm{0}}\right)\left(P_{(3)}+P_{(0,3)}+P_{(1,2)}+P_{(2,1)}+P_{(0,1,2)}+P_{(1,1,1)}+P_{(0,1,1,1)}\right).
\end{aligned}
\end{equation}
Then 
\begin{align}
\begin{split}
\big|\mathbb{E}[Y ]\big|\prec N^{-1/2-c}(1 + |A_{ab}| \Psi^{-1}).
\end{split}
\end{align}
\end{lemma}

\subsection{Resolvent Comparison}\label{subsec:resolventcomparison}
Given Lemma~\ref{lem:third-moment-terms}, we can conclude by a standard resolvent comparison argument. 
\begin{proof}[Proof of Theorem \ref{thm:CLT}]
Let $W$ be drawn from the Gaussian Orthogonal Ensemble, and let $H$ be a Wigner matrix. We first show that, for every smooth and compactly supported $g$, we have
\begin{align}\label{eqn:comparison}
\begin{split}
\left|\mathbb{E}\left[g(v_{\ell}^W)-g(v_{\ell}^H) \right]\right|\leq c^{-1}N^{-c}
\end{split}
\end{align}
for some constant $c>0$ (depending on $\tau$ and $g$), where $v_{\ell}^W$ and $v_{\ell}^H$ are corresponding regularized observables (see \eqref{eqn:regularized-observable}) for $W$ and $H$ respectively.

To this end, we fix a bijection
\begin{equation*}
\psi:\{(i, j): 1 \leqslant i \leqslant j \leqslant N\} \rightarrow \llbracket 1, \xi_N \rrbracket,
\end{equation*}
where $\xi_N=N(N+1)/2$,
and define the interpolating matrices $H^0,H^1,H^2,\ldots,H^{\xi_N}$ by
\begin{equation*}
h_{i j}^\xi= \begin{cases}h_{i j} & \text { if } \psi(i, j) >\xi, \\ w_{i j} & \text { if } \psi(i, j)\leq\xi,\end{cases}
\end{equation*}
for $i\leq j$. Therefore, $H^0=H$ and $H^{\xi_N}=W$. We may rewrite \eqref{eqn:comparison} as a telescopic summation,
\begin{align}\label{eqn:comparison-telescopic}
\begin{split}
\left|\mathbb{E}\left[g(v_{\ell}^W)-g(v_{\ell}^H) \right]\right|\leq \sum_{\xi=1}^{\xi_N}\left|\mathbb{E}\left[g\left(v_\ell^{H^{\xi}}\right)-g\left(v_\ell^{H^{\xi-1}}\right) \right]\right|.
\end{split}
\end{align}

Fix some $\xi\in\llbracket 1,\xi_N\rrbracket$ and consider the indices $(a,b)$ such that $\psi(a,b)=\xi$. Let $Q^{\xi}$ be the matrix obtained from $H^{\xi}$ by setting $h^\xi_{ab}$ and $h^{\xi}_{ba}$ to zero. Note that $Q^{\xi}$ can also be obtained from $H^{\xi-1}$ by setting $h^{\xi-1}_{ab}$ and $h^{\xi-1}_{ba}$ to zero. We consider the following two cases.
\item First, suppose $a=b$. Lemma \ref{lem:bound-x-integral} and Lemma \ref{lem:resolvent-expansion} imply that, with $Y$ denoting any term from \eqref{eqn:third-moment-terms},
%\begin{align}
%\begin{split}
    $|Y|\prec N^{-\tau/30}$,
%\end{split}
%\end{align}
where we use the upper bound on $\Psi(z)$ from \eqref{eqn:psi-bound-specified}.
Combining with \eqref{eqn:resolvent-expansion-integral-poly}, we have
\begin{align*}
\begin{split}
\left|\mathbb{E}\left[g(v_{\ell}^{H^{\xi}})-g(v_{\ell}^{H^{\xi-1}}) \right]\right|\leq\left|\mathbb{E}\left[g(v_{\ell}^{H^{\xi}})-g(v_\ell^{Q^\xi}) \right]+\left[g(v_\ell^{Q^\xi})-g(v_{\ell}^{H^{\xi-1}}) \right]\right|\leq N^{-3/2-c},
\end{split}
\end{align*}
where we use the fact that the first two moments of Wigner matrices $H^{\xi}$ and $H^{\xi-1}$ are the same, and therefore $\E[\mathcal A]$ in \eqref{eqn:resolvent-expansion-integral-poly} is the same for both cases.

Now, if $a\neq b$. Combining Lemma \ref{lem:third-moment-terms} and \eqref{eqn:resolvent-expansion-integral-poly}, we have
\begin{align*}
\begin{split}
\left|\mathbb{E}\left[g(v_{\ell}^{H^{\xi}})-g(v_{\ell}^{H^{\xi-1}}) \right]\right|\leq\left|\mathbb{E}\left[g(v_{\ell}^{H^{\xi}})-g(v_\ell^{Q^\xi}) \right]+\left[g(v_\ell^{Q^\xi})-g(v_{\ell}^{H^{\xi-1}}) \right]\right|\leq N^{-2-c}(1 + |A_{ab}| \Psi^{-1})
\end{split}
\end{align*}
for some $c>0$.
These two estimates conclude the proof of \eqref{eqn:comparison} for smooth and compactly supported $g$ in view of \eqref{eqn:comparison-telescopic} and the estimate
\begin{equation}
\sum_{ 1 \le j \le N, \, j\neq a} |A_{aj} | \le \sqrt{N} , 
\end{equation}
which is implied by $\| A \| \le 1$. 
Combining \eqref{eqn:comparison} with \Cref{thm:GOE-CLT} and \Cref{lem:regularized-observable}, we have proved 
\begin{equation}\label{smooth-convergence}
\begin{aligned}
    \lim_{N \rightarrow \infty} \E\big[g(\widehat p_\ell)-g(X)\big] = 0,
\end{aligned}
\end{equation}
for smooth and compactly supported $g$, where $X$ is a standard Gaussian random variable.

For compactly supported but not necessarily continuous $g$, \eqref{smooth-convergence} can be proved by approximating $g$ by the smooth function $g*\gamma_\epsilon$, where $\gamma : \R \rightarrow \R$ is any nonnegative, smooth, compactly supported function that integrates to one, $\gamma_\epsilon(x) = \epsilon^{-1} \gamma( x / \epsilon)$, and we take $\epsilon \rightarrow 0$.  
This implies that $\widehat p_\ell$ converges to standard Gaussian random variable in distribution (see  \cite[Theorem 13.16 (vii)]{klenke2013probability}).
\end{proof}

%\newpage
\section{Proof of Lemma \ref{lem:third-moment-terms}}\label{sec:poly}
We now fix indices $a,b\in\llbracket1,N\rrbracket$ such that $a\neq b$, and carry this choice throughout the current section. All of the bounds stated below are uniform in the choice of $a$ and $b$.

Recall from \eqref{e:semicircle} that $\m$ denotes the Stieltjes transform of the semicircle law, which is deterministic and satisfies
\begin{equation}\label{definingequation}
\m(z)+z+\frac{1}{\m(z)}=0.
\end{equation}
Recall from the discussion below \eqref{controlparameter} that $Q$ is the matrix obtained by setting $h_{a b}, h_{b a}$ in $H$ to 0, and $R$ is the resolvent of $Q$. Let $Q^{(a)}$ be the matrix obtained by setting $a$-th row and column of $H$ to $0$ and let $R^{(a)}$ be the resolvent of $Q^{(a)}$. Then it follows from the 
 fact that the inverse of a block matrix can be computed block-by-block that 
%Schur complement formula that
\begin{equation}\label{eqn:ra-def}
        \left(R^{(a)}\right)_{ij}=\begin{cases}
            0,&\text{ if exactly one of }i,j\text{ is }{a},\\
            -z^{-1},&\text{ if }i=j={a},\\
            W_{ij},&\text{ otherwise},
        \end{cases}
    \end{equation}
 
where $W$ is the resolvent of the $(N-1) \times (N-1)$ matrix with entries $(Q_{ij})_{i,j \in T}$ for $T = \{1,\dots, N\} \setminus \{ a \}$. We set $W_{ij} =0$ when at least one of $i$ and $j$ equals $a$. 

We now state some necessary local laws for $R^{(a)}$.
%We state the following local laws for $R^{(a)}$, whose proof follows directly from the corresponding local laws in Lemma \ref{lem:three-resolvent-local-law} and \eqref{eqn:ra-def}, and is therefore omitted.
\begin{lemma}\label{lem:three-resolvent-local-law-ra}
    Let $H$ be a Wigner matrix.
    \begin{enumerate}
        \item For all $z\in\bm S$, we have
        \begin{equation}
            \left|\ra_{\x\y}-\langle\boldsymbol{x}, \boldsymbol{y}\rangle m_{\mathrm{sc}}\right| \prec \Psi, \quad\left|(\ra\ra)_{\x \y}\right| \prec N \Psi^2, \quad\left|(\ra\rab)_{\x \y}\right| \prec N \Psi^2 ,\label{eqn:2-resolvents-ra}
        \end{equation}
        
        %uniformly over 
         for all  $\x,\y\in\mathbb S^{N-1}$ such that at least one of $\x,\y$ is $\e_s$ with some $s\neq a$, and $c,d\in\llbracket1,N\rrbracket$ such that $c,d\neq a$.
        \item Furthermore, if $z=E+\iu\eta\in\bm S$ satisfies
        \begin{equation*}
            E \in I_{\ell}, \quad \eta=\eta_{\ell} ,
        \end{equation*}
        then for any deterministic $A\in\operatorname{Mat}_N$ such that $\|A\|\leq 1$ and $\Tr A=0$, we have
        \begin{align}
            \left|(\ra A \rab)_{c d}\right|  &\prec N^{1 / 2} \Psi,\label{eqn:2-resolvents-traceless-ra} \\
            \left|(\ra A \rab \ra)_{c d}\right|  &\prec N^{3/2} \Psi^{9 / 4},\label{eqn:3-resolvents-traceless-ra}
        \end{align}
        uniformly over all $c,d\in\llbracket1,N\rrbracket$ such that $c\neq a$ and $d\neq a$.
    \end{enumerate}
\end{lemma}
\begin{proof}
%For the first inequality in \eqref{eqn:2-resolvents-ra}, 

Suppose without loss of generality that $\x = \e_s$ with $s \neq a$. 
%By the first estimate in \eqref{eqn:2-resolvents} and 
Using \eqref{eqn:ra-def}, we have 
\[
R^{(a)}_{\bm x \bm y}
= \sum_{r=1}^N \langle \e_s, R^{(a)} \e_r\rangle \langle \e_r, \y\rangle 
= 
\sum_{r=1}^N \langle \e_s, W \e_r\rangle \langle \e_r, \y\rangle 
= W_{\bm x \bm y} ,
\]
and the result follows from \eqref{eqn:2-resolvents}  applied to $W$ (after rescaling $\bm y$ appropriately, since $W_{\bm x \bm y}$ omits the $a$-th entry of $\bm y$).

Next, we have 
\begin{equation}\label{a3v1}
\begin{aligned}
    (R^{(a)}  R^{(a)})_{\x\y}=\sum_{k\neq a} R^{(a)}_{s k} R^{(a)}_{k\y} = (W W)_{s \y}
    %&\prec \left|\sum_k\langle\e_s,\e_k\rangle\langle\e_k,\y\rangle\right|+\left|\sum_k\langle\e_s,\e_k\rangle\right|\Psi+\left|\sum_k\langle\e_k,\y\rangle\right|\Psi+N\Psi^2\\
    %&\leq \big|\langle\e_s,\y\rangle\big|+2N^{1/2}\Psi+N\Psi^2\prec N\Psi^2,
\end{aligned}
\end{equation}
Applying \eqref{eqn:2-resolvents}, this proves the second claim in \eqref{eqn:2-resolvents-ra}, and the third follows similarly.

Turning to \eqref{eqn:2-resolvents-traceless-ra}, we write 
\[
(\ra A \rab)_{c d} = \sum_{i,j} R^{(a)}_{ci} A_{ij} \bar {R}^{(a)}_{jd} = \sum_{i,j} W_{ci} A_{ij} \bar W_{jd}  = (W  A'  \bar W)_{c d},
\]
where $A'$ is the $(N-1)\times (N-1)$ matrix obtained by deleting the $a$-th row and column of $A$. Fix an index $m\neq a,b,c,d$, and let $D$ be the diagonal matrix with $D_{mm} = A_{aa}$. We have
\begin{align}
(W  A'  \bar W)_{c d}
= \big(W  (A'  + D  )  \bar W\big)_{c d} -  (W  D  \bar W)_{cd}.
\end{align}
Since $\Tr (A'  + D)=0$, the first term is bounded using \eqref{eqn:2-resolvents-traceless}. The second term equals $W_{cm} \bar W_{md}$, which is $O_{\prec}(\Psi^2)$ by \eqref{eqn:2-resolvents-ra}, since these are off-diagonal resolvent entries.

Similarly, 
\begin{align*}
(\ra A \rab \ra)_{c d} &= \sum_{i,j,k} R^{(a)}_{ci} A_{ij} \bar {R}^{(a)}_{jk}R^{(a)}_{kd}\\
%&=\sum_{i, k\neq a, j} R_{ci} A_{ij} \overline{R}^{(a)}_{jk}R_{kd}\\
&=  \sum_{i,j, k\neq a } W_{ci} A_{ij} \overline{W}_{jk}W_{kd}\\% + \sum_{i \notin \{a,b\}}  R_{ci} A_{ia} ( - \overline{z}^{-1} ) R_{ad} \\ % + R_{ci} A_{ia} \overline{R}_{ak}R_{kd} \eer \\
& = (W A' \overline{W} W)_{c d} 
= (W (A' +D) \overline{W} W)_{c d}   - 
W_{cm} ( \overline{W} W)_{m d} 
.
\end{align*}
%where in the last line we used \eqref{eqn:2-resolvents}, \eqref{eqn:2-resolvents-ra}, \eqref{eqn:2-resolvents-traceless-ra}, and $|A_{ij}| \le 1$ for all $i,j$ (which follows from our hypothesis that $\| A \| \le 1$). 
We conclude using \eqref{eqn:3-resolvents-traceless} and \eqref{eqn:2-resolvents-ra}.

\end{proof}
% Finally, for conciseness, we let $A$ stand for the matrix defined in \eqref{eqn:traceless-A}, which satisfies
% \begin{align*}
% \begin{split}
% \|A\|\leq 1,\qquad \Tr(A)=0.
% \end{split}
% \end{align*}
The main goal of this section is to rewrite the resolvent expansion terms $x_i$ and $ y_i$ from Lemma~\ref{lem:resolvent-expansion} into a certain polynomial that allows us to take advantage of the cancellation mechanism noted in  
\eqref{i:additional-gain}.
%form as introduced in \cite[Section~7]{BloKnoYauYin16}. 
For instance, we want to rewrite
\begin{equation}\label{explainx1}
x_1 \approx\left(\sum_{i_1, \ldots, i_d\neq a} V_{i_1, \ldots, i_d} h_{i_1 a} \cdots h_{i_d a}\right) \cdot\left(\prod_{j=d+1}^{d+m} \sum_{i_j=1}^N V_{i_j}\left(h_{i_j a}^2-\frac{1}{N}\right)\right),
\end{equation}
where the $V$ terms are $Q^{(a)}$-measurable. The reason to write it into this form is that, whenever $d$ is odd, we gain a factor of $N^{-1 / 2}$ upon taking the expectation (see Lemma \ref{lem:odd-poly} for the precise statement), which is essential in the proof of Lemma \ref{lem:third-moment-terms}.

Fixing a universal constant $C_0>0$ and following the setup in \cite[Section~7]{BloKnoYauYin16}, we make the following definitions. 
\begin{definition}[Admissible weights]\label{admissibleweights}
Let $\varrho=\left(\varrho_i: i \in \llbracket 1, N \rrbracket\right)$ be a sequence of deterministic nonnegative real numbers. We say that $\varrho$ is an \emph{admissible weight} if
\begin{equation*}
\frac{1}{N^{1 / 2}}\left(\sum_{i=1}^N \varrho_i^2\right)^{1 / 2} \leqslant 1, \quad \frac{1}{N^{1 / 2}}\left(\sum_{i=1}^N \varrho_i^3\right)^{1 / 3} \leqslant N^{-1 / 6} .
\end{equation*}
\end{definition}

\begin{definition}[$O_{\prec,d}(\cdot)$]\label{def:od}
For a given degree $d \in \mathbb{N}$ let
\begin{equation*}
\mathcal{P}=\sum_{\substack{1\le i_1, \ldots, i_d \le N \\ i_1, \ldots, i_d \neq a}} V_{i_1 \cdots i_d} h_{a i_1} \cdots h_{a i_d}
\end{equation*}
be a polynomial in entries of the a-th row of $Q$. We write $\mathcal{P}=O_{\prec, d}(K)$ if the following conditions are satisfied.
\begin{enumerate}
\item $K$ is deterministic and $V_{i_1 \cdots i_d}$ is $Q^{(a)}$-measurable.
\item There exist admissible weights $\varrho^{(1)}, \ldots, \varrho^{(d)}$ such that
\begin{equation*}
\left|V_{i_1 \cdots i_d}\right| \prec K \varrho_{i_1}^{(1)} \cdots \varrho_{i_d}^{(d)} .
\end{equation*}
\item We have the deterministic bound $\left|V_{i_1 \cdots i_d}\right| \leqslant N^{C_0}$.
\end{enumerate}
The above definition also extends to $d=0$, where $\mathcal P=V$ is $Q^{(a)}$-measurable.
\end{definition}

\Cref{def:od} corresponds to the first term in our desired representation \eqref{explainx1}. 
The point of the growth condition in Definition~\ref{admissibleweights} is to ensure that we have $\mathcal P = O_{\prec}(K)$ whenever $\mathcal P = O_{\prec,d}(K)$, as noted in Remark~\ref{rmk:dominationRefinement} below. We next make a definition corresponding to the second term of \eqref{explainx1}.

\begin{definition}[$O_{\prec,\diamond}(\cdot)$]\label{def:odiamond}
Let $\mathcal{P}$ be a polynomial of the form
\begin{equation*}
\mathcal{P}=\sum_{i=1}^N V_i\left(h_{a i}^2-\frac{1}{N}\right).
\end{equation*}
We write $\mathcal{P}=O_{\prec, \diamond}(K)$ if $V_i$ is $Q^{(a)}$-measurable, $\left|V_i\right| \leqslant N^{C_0}$, and $\left|V_i\right| \prec K$ for some deterministic $K$.
\end{definition}

We finally define a class of terms that generalizes \eqref{explainx1}, and tracks whether $d$ is even or odd (since we expect additional cancellation when $d$ is odd).

\begin{definition}[Graded polynomials]\label{def:graded-poly} We write $\mathcal{P}=O_{\prec, \even}(K)$ if $\mathcal{P}$ is a sum of at most $C_0$ terms of the form
\begin{equation*}
K \mathcal{P}_0 \prod_{s=1}^n \mathcal{P}_i, \quad \mathcal{P}_0=O_{\prec, 2 d}(1), \quad \mathcal{P}_i=O_{\prec, \diamond}(1)
\end{equation*}
where $ 0 \le d, n \leqslant C_0 $ and $K$ is deterministic. Moreover, we write $\mathcal{P}=O_{\prec, \odd}(K)$ if $\mathcal{P}=\widehat{\mathcal{P}} \mathcal{P}_{\even}$, where $\widehat{\mathcal{P}}=O_{\prec, 1}(1)$ and $\mathcal{P}_\even=O_{\prec, \even}(K)$.
\end{definition}
\begin{remark}
    The graded polynomials satisfy simple algebraic rules by definition, which we state without proof:
    \begin{align*}
        \Oparity{*}{K_1}+\Oparity{*}{K_2}&=\Oparity{*}{K_1+K_2},\\
        \Oparity{*}{K_1}\Oparity{*}{K_2}&=\Oparity{\even}{K_1K_2},\\
        \Oparity{\odd}{K_1}\Oparity{\even}{K_2}&=\Oparity{\odd}{K_1K_2},
    \end{align*}
    after possibly increasing $C_0$. Here $*$ represents either $\odd$ or $\even$. It should be noted that all of these operations can be done for an arbitrary, but finite, number of times (independent of $N$).
\end{remark}
\begin{remark}\label{rmk:dominationRefinement}
    Definitions \ref{def:od}--\ref{def:graded-poly} refine the stochastic domination notation from \Cref{def:stochasticDomination}. More precisely, we have
    \begin{align}\label{eqn:refinement-implications}
        \mathcal{P}=O_{\prec, *}(K) \Longrightarrow \mathcal{P}=O_{\prec}(K),
    \end{align}
    where $*$ can represent $d$, $\diamond$, $\even$ or $\odd$. See lines under \cite[Equation~(7.56)]{BloKnoYauYin16} for details.
\end{remark}
We now state the following lemma proved in \cite[Lemma~7.13]{BloKnoYauYin16}, which formalizes that claim that we have additional cancellation for odd terms.
\begin{lemma}\label{lem:odd-poly}
Let $\mathcal{P}=O_{\prec, \odd}(K)$ for some deterministic $K \leqslant N^{C_0}$. Then for any fixed $D>0$, we have
\begin{equation*}
\big|\mathbb{E} [\mathcal{P}]\big| \prec N^{-1 / 2} K+N^{-D}.
\end{equation*}
\end{lemma}
The following resolvent identities are standard (see \cite[Lemma~3.5]{benaych2016lectures} and \cite[Equation~(4.1)]{benaych2016lectures}):

\begin{align}
R_{a a}&=\frac{1}{-z-\sum_{r, s \notin \{a,b\}} R_{r s}^{(a)} h_{a r} h_{a s}}, \label{eqn:neu-series-1}\\
 R_{a i}&=-R_{a a} \sum_{r \notin \{a,b\}} h_{a r} R_{ri}^{(a)} , \label{eqn:neu-series-2}\\
  R_{ia}&=-R_{a a} \sum_{r \notin \{a,b\}} R_{ir}^{(a)} h_{ra}  \\
R_{i j}&=R_{i j}^{(a)}+R_{a a}\left( \sum_{r \notin \{a,b\}} R_{i r}^{(a)} h_{ra} \right) \left( \sum_{s \notin \{a,b\}} h_{a s} R_{sj}^{(a)} \right) ,\label{eqn:neu-series-3}
\end{align}
where the second and third identities hold for any index $i\neq a$, and the fourth requires $i,j \neq a$.

\begin{remark}
In stating the above identities, we used that $R$ is the resolvent of $Q$, and $Q_{ab} = Q_{ba} = 0$. Hence, the terms corresponding to $h_{ab}$ and $h_{ba}$ are omitted in the summations. We will continue accounting for these omitted terms in the computations in the remainder of this section without mentioning it explicitly. 
\end{remark}
In the next several lemmas, we write resolvents and multi-resolvents in terms of graded polynomials.
\begin{lemma}\label{lem:parity-r}
Fix $D>0$. For every spectral parameter $z\in \bm S$ and index $c \neq a$, we have
\begin{align}
    R_{a a}&=O_{\prec, \even}(1)+O_{\prec}\left(N^{-D}\right),\label{eqn:parity-aa}\\
    R_{a c }&=O_{\prec, \odd}(\Psi)+\Od,\label{eqn:parity-ab}\\
    R_{cc  }&=O_{\prec, \even}(1)+\Od,\label{eqn:parity-bb}\\
    \Tr(R)&=\sum_{i\neq a }\ra_{ii}+\Oparity{\even}{N\Psi^2}+\Od.\label{eqn:parity-trace}
\end{align}
\end{lemma}
\begin{proof}
We begin with some preliminary claims. 
Using \eqref{eqn:2-resolvents-ra} and the definition of graded polynomial, we have
\begin{equation}\label{eqn:quad-parity}
\begin{aligned}
\sum_{r, s \notin \{a,b\}} R_{r s}^{(a)} h_{a r} h_{a s}-\m & =\sum_{r \notin \{a,b\}} R_{r r}^{(a)}\left(h_{a r}^2-\frac{1}{N}\right)+\left(\frac{1}{N} \sum_{r \notin \{a,b\}} R_{r r}^{(a)}-\m\right)+\sum_{\substack{r \neq s \\
r, s \notin \{a,b\}}} R_{r s}^{(a)} h_{a r} h_{a s} \\
& = O_{\prec, \diamond}(1)\eeb +O_{\prec, 0}(\Psi)+O_{\prec, 2}(\Psi) \\
& =O_{\prec, \even}(1).
\end{aligned}
\end{equation}
We also have 
\begin{equation}\label{eqn:quad}
\begin{aligned}
\sum_{r, s \notin \{a,b\}} R_{r s}^{(a)} h_{a r} h_{a s}-\m & =\sum_{r \notin \{a,b\}} R_{r r}^{(a)}\left(h_{a r}^2-\frac{1}{N}\right)+\left(\frac{1}{N} \sum_{r \notin \{a,b\}} R_{r r}^{(a)}-\m\right)+\sum_{\substack{r \neq s \\
r, s \notin \{a,b\}}} R_{r s}^{(a)} h_{a r} h_{a s} \\
& =O_{\prec}\left(N^{-1 / 2}\right)+O_{\prec}(\Psi)+O_{\prec}(\Psi) \\
& =O_{\prec}(\Psi) ,
\end{aligned}
\end{equation}
where we used the definition of graded polynomial, \eqref{eqn:refinement-implications} in the second step and \eqref{eqn:psi-bound} in the last step. For the second step, we also used a standard concentration bound on the first sum (see, e.g., \cite[Theorem B.1(i)]{erdHos2013delocalization}).

Note also that for all $i\neq a$,
\begin{equation}\label{eqn:linear-parity}
 \sum_{r \notin \{a,b\}} R_{i r}^{(a)} h_{r a}=
%R_{i i}^{(a)} h_{i a}+O_{\prec, \odd}(\Psi) 
%= O_{\prec, 1}(N^{-1/9}) +O_{\prec, \odd}(\Psi)
 O_{\prec, \odd}(\Psi) ,
\end{equation}
by \eqref{eqn:2-resolvents-ra} and \Cref{def:od}.
%\ber where we used \Cref{def:od} with the weights $\varrho_r = 1$ for $r\neq a$ and $\varrho_b = 0$. \eer 

For all $n \in \mathbb{N}$, we have by \eqref{definingequation} and \eqref{eqn:neu-series-1} that
\begin{equation}\label{eqn:expansion-aa}
\begin{aligned}
R_{a a}  =\frac{1}{-z-\sum_{r, s \notin \{a,b\}} R_{r s}^{(a)} h_{a r} h_{a s}}
& =\frac{1}{-z-\m+\left(\m-\sum_{r, s \notin \{a,b\}} R_{r s}^{(a)} h_{a r} h_{a s}\right)} \\
& =\frac{1}{1 / \m+\left(\m-\sum_{r, s \notin \{a,b\}} R_{r s}^{(a)} h_{a r} h_{a s}\right)} \\
& =\frac{\m}{1+\m\left(\m-\sum_{r, s \notin \{a,b\}} R_{r s}^{(a)} h_{a r} h_{a s}\right)} \\
& =\sum_{j=0}^n \m^{j+1}\left(\sum_{r, s \notin \{a,b\}} R_{r s}^{(a)} h_{a r} h_{a s}-\m\right)^j+O_{\prec}\left(\m^{{n+1}}\Psi^n\right) \\
& =O_{\prec, \even}(1)+O_{\prec}\left(\m^{{n+1}} \Psi^n\right),
\end{aligned}
\end{equation}
where we used \eqref{eqn:quad} in the second-to-last line and \eqref{eqn:quad-parity} in the last step.

By \eqref{eqn:neu-series-2},  
\eqref{eqn:linear-parity},    
and \eqref{eqn:expansion-aa}, we have
\begin{align}\label{eqn:expansion-ab}
\begin{split}
    R_{ac}={-}R_{aa}\sum_{r\notin \{a,b\}}\ra_{rc}h_{ar}=\Oparity{\odd}{\Psi}+O_\prec(\m^{{n+1}}\Psi^{n+1}).
\end{split}
\end{align}
By \eqref{eqn:neu-series-3}, \eqref{eqn:linear-parity},     and \eqref{eqn:expansion-aa}, we have
\begin{align}\label{eqn:expansion-bb}
\begin{split}
    R_{cc}=R_{cc}^{(a)}+R_{aa}\left(\sum_{r\notin \{a,b\}}\ra_{cr}h_{ar}\right)^2=\Oparity{\even}{1}+O_\prec(\m^{{n+1}}\Psi^{n+2}).
\end{split}
\end{align}
Summing over all $c  \neq a$, we get 
\begin{align}\label{e:tracemarch}
\sum_{c \neq a}
 R_{cc}
 = 
 \sum_{c \neq a}
 R^{(a)}_{cc} + R_{aa} \sum_{r,s \notin \{a,b\} } h_{ar} h_{as} (R^{(a)} R^{(a)} )_{rs}.
\end{align}

The lemma follows from \eqref{eqn:expansion-aa}, \eqref{eqn:expansion-ab},  \eqref{eqn:expansion-bb}, and \eqref{e:tracemarch} by choosing a sufficiently large $n\equiv n(\tau,D)$.
\end{proof}

\begin{lemma}\label{lem:parity-rr}
Fix $D>0$. For every spectral parameter $z\in\bm S$ and index $c\neq a$, we have
\begin{align}
    (R\bar R)_{aa}&=\Oparity{\even}{N\Psi^2}+O_\prec(N^{-D}),\label{eqn:parity-rr-aa}\\
    (R\bar R)_{ac}&=\Oparity{\odd}{N\Psi^2}+O_\prec(N^{-D}),\label{eqn:parity-rr-ab}\\
    (R\bar R)_{cc}&=\Oparity{\even}{N\Psi^2}+O_\prec\left(N^{-D}\right),\label{eqn:parity-rr-bb}\\
    \left(R^2\right)_{aa} &=\Oparity{\even}{N\Psi^2}+O_{\prec}(N^{-D}),\label{eqn:parity-r1r-aa}\\
    \left(R^2\right)_{ac} &=\Oparity{\odd}{N\Psi^2}+O_{\prec}(N^{-D}),\label{eqn:parity-r1r-ab}\\
    \left(R^2\right)_{cc} &=\Oparity{\even}{N\Psi^2}+O_{\prec}(N^{-D}).\label{eqn:parity-r1r-bb}
\end{align}
\end{lemma}
\begin{proof}We present only the proofs for  \eqref{eqn:parity-rr-aa} and \eqref{eqn:parity-rr-ab}; the others can be shown similarly. By the resolvent identity \eqref{eqn:neu-series-2},% and \eqref{eqn:neu-series-2}, we have
\begin{align}\label{eqn:expansion-rr-aa}
\begin{split}
    (R\bar R)_{aa}&=\sum_{i\notin \{a,b\}}\left(R_{ai}\bar R_{ia}\right)+|R_{aa}|^2\\
    =&|R_{aa}|^2\sum_{i\notin \{a,b\}}\left(\sum_{r\notin \{a,b\}}\ra_{ri}h_{ar}\right)\left(\sum_{s\notin \{a,b\}}\rab_{is}h_{as}\right)+|R_{aa}|^2\\
    =&|R_{aa}|^2\left(\sum_{r,s\notin \{a,b\}}\rr_{rs}h_{ar}h_{as}+1\right).
\end{split}
\end{align}
Combining \eqref{eqn:psi-bound}, \eqref{eqn:2-resolvents-ra}, \eqref{eqn:parity-aa}, \eqref{eqn:expansion-aa}, and \eqref{eqn:expansion-rr-aa}, we deduce \eqref{eqn:parity-rr-aa}.

Next, we consider \eqref{eqn:parity-rr-ab}. We have 
\begin{equation}
(R \bar R)_{ac} = 
    \sum_{i\notin \{a,b\}}\left(R_{ai}\bar R_{ic}\right)+R_{aa} \bar R_{{ac}},
\end{equation}
and $R_{{aa}} \bar R_{{ac}} =O_{\prec, \odd}(\Psi) + O_{\prec}(N^{-D})$ by \Cref{lem:parity-r}. Using \eqref{eqn:neu-series-3} and the resolvent identities used previously, we have 
\begin{align*}
\begin{split}
\sum_{i\notin \{a,b\}}\left(R_{ai}\bar R_{ic}\right)
&= 
- R_{aa} \sum_{i\notin \{a,b\}}
\left( \sum_{r\notin \{a,b\}} h_{ar} R_{ri}^{(a)} \right) 
\left( 
\bar R^{(a)}_{ic}  + \bar R_{aa} 
\left(
\sum_{s \notin \{a,b\}} \bar R_{{sc}}^{(a)} h_{sa}
\right)
\left(
\sum_{t \notin \{a,b\}} \bar R_{it}^{(a)} h_{{at}}
\right)
\right)
\\
&=- R_{aa} \sum_{r\notin \{a,b\}} h_{ar} \big (R^{(a)} \bar R^{(a)} \big)_{rc}
-{\vert} R_{aa}{\vert^2} \sum_{s,t \notin \{a,b\}} h_{{ar}}h_{sa}h_{{at}}
\big (R^{(a)} \bar R^{(a)} \big)_{rt}\bar R_{sc}^{(a)}.
\end{split}
\end{align*}
Applying \eqref{eqn:2-resolvents-ra} and \eqref{eqn:parity-r1r-aa} completes the proof of \eqref{eqn:parity-rr-ab}.
\end{proof}

\begin{lemma}\label{l:parity-ar}
Fix $D>0$. For all spectral parameters $z=E+\iu\eta$ satisfying $E\in I_\ell$ and $\eta=\eta_\ell$, and indices $c \neq a$,
\begin{align}
(AR)_{a a}  &=\Oparity{\even}{1 }+\Oparity{\odd}{1 }+O_{\prec}\left(N^{-D}\right),\label{e:ar-parity}\\
(AR)_{a c}  &=\Oparity{\odd}{\Psi }+
\Oparity{\even}{\Psi} + \Oparity{\even}{|A_{ac}|} + O_{\prec}\left(N^{-D}\right).\label{e:ar-odd}
\end{align}
\end{lemma}
\begin{proof}
We begin by noting that the first inequality in \eqref{eqn:2-resolvents-ra} implies that
\begin{align}\label{e:ar-march}
\big| (A\bar R^{(a)})_{as} \big|
= 
\big| \langle \bm e_a, A\bar R^{(a)} \bm e_s \rangle \big|
=\big| \langle  A \bm e_a, \bar R^{(a)} \bm e_s \rangle \big|\prec 1,
%\Psi + |A_{as}|,
%= \left| \sum_{r} A_{ar} R^{(a)}_{rs}  \right|
%\le 
\end{align}
since $A$ is deterministic and symmetric and $s\neq a$.

To prove \eqref{e:ar-parity}, we write
\begin{align}
(AR)_{a a}  = A_{aa} R_{aa}  + \sum_{i\neq a} A_{ai} R_{ia}.
\end{align}
The first term is $\Oparity{\even}{1 }$, by \Cref{lem:parity-r}. We expand the second term as 
\begin{align}\label{e:useabove}
\sum_{i\neq a } A_{ai} R_{ia}
= \sum_{i\neq a } A_{ai} \left( -R_{a a} \sum_{r \notin \{a,b\}} h_{ra} R_{ir}^{(a)} \right)
= - R_{a a} \sum_{r \notin \{a,b\}} h_{ra} (A R^{(a)})_{ar} = \Oparity{\odd}{1}
\end{align}
where the last bound uses \eqref{e:ar-march}. This shows \eqref{e:ar-parity}.

Next, we have 
\begin{align}
(AR)_{a c}  = A_{aa} R_{ac}  + \sum_{i\neq a } A_{ai} R_{ic}.
\end{align}
The first term is $\Oparity{\odd}{{\Psi} }$, by \Cref{lem:parity-r}. The sum is 
\begin{align}
\begin{split}
\sum_{i\neq a } A_{ai} R_{ic} &= 
\sum_{i\neq a }  A_{ai} \left(
R_{i c}^{(a)}+R_{a a}\left( \sum_{r \notin \{a,b\}} R_{i r}^{(a)} h_{ra} \right) \left( \sum_{s \notin \{a,b\}} h_{a s} R_{sc}^{(a)} \right)\right)\\
&= (A R^{(a)})_{ac} + \sum_{r,s \notin \{a,b \}}
h_{ra} h_{as}
(AR^{(a)})_{ar} R_{sc}^{(a)} \\
&= \Oparity{\even}{\Psi} + \Oparity{\even}{|A_{ac}|} + O_{\prec}\left(N^{-D}\right).
\end{split}
\end{align}
In the last line, we used \eqref{eqn:2-resolvents-ra} to estimate the first term and $(AR^{(a)})_{ar}$ in the sum.

\end{proof}
\begin{lemma}\label{lem:parity-rar-rarr}
Fix $D>0$. For all spectral parameters $z=E+\iu\eta$ satisfying $E\in I_\ell$ and $\eta=\eta_\ell$, and indices $c\neq a$, 
we have
\begin{align}
(R A \bar{R})_{a a}  &=\Oparity{\even}{N^{1/2}\Psi}+\Oparity{\odd}{1}+O_\prec(N^{-D}),\label{eqn:parity-rar-aa} \\
(R A \bar{R})_{a c}  &=\Oparity{\odd}{N^{1/2}\Psi}+\Oparity{\even}{\Psi} + 
 \Oparity{\even}{|A_{ac}|}+O_{\prec}(N^{-D}),\label{eqn:parity-rar-ab} \\
(R A \bar{R})_{cc}  &=\Oparity{\even}{N^{1/2}\Psi}+ \Oparity{\odd}{\Psi} +O_\prec(N^{-D}),\label{eqn:parity-rar-bb} \\
\Tr\left(RA\bar R\right)&=\Tr\rar+\Oparity{\even}{N^{3/2}\Psi^{9/4}}+\Oparity{\odd}{N\Psi^2}+O_\prec(N^{-D}),\label{eqn:parity-rar-trace}.
\end{align}
\end{lemma}
\begin{proof}%We present the proof of \eqref{eqn:parity-rar-aa}. The others can be shown similarly. At the end of the proof, we discuss the additional error terms in \eqref{eqn:parity-rarr-bb} and \eqref{eqn:parity-rarr-ba}.

 We begin with \eqref{eqn:parity-rar-aa}. 
By the resolvent identity \eqref{eqn:neu-series-2}, we have
\begin{align}
\label{eqn:expansion-rar-aa}
(RA\bar R)_{aa}=&|R_{aa}|^2\sum_{r,s\notin \{a,b\}}\left(R^{(a)}A\bar R^{(a)}\right)_{rs}h_{ar}h_{as}\\
&+|R_{aa}|^2\left(\sum_{s\notin \{a,b\}}\left(A\rab\right)_{as}h_{as}+\sum_{r\notin \{a,b\}}\left(\ra A\right)_{ra}h_{ar}+A_{aa}\right).\label{eqn:expansion-rar-aa2}
\end{align}
By the definition of graded polynomial (see Definition \ref{def:graded-poly}), \eqref{eqn:parity-aa}, and \eqref{eqn:2-resolvents-traceless-ra}, the term in \eqref{eqn:expansion-rar-aa} is
\begin{align}\label{eqn:rar-aa-first-line}
\begin{split}
|R_{aa}|^2\sum_{r,s\notin \{a,b\}}\left(R^{(a)}A\bar R^{(a)}\right)_{rs}h_{ar}h_{as}=\Oparity{\even}{N^{1/2}\Psi}+O_{\prec}(N^{-D}).
\end{split}
\end{align}

Recall \eqref{e:ar-march}, and note that similarly, we have
%\begin{align}
%\begin{split}
$\left(\ra A\right)_{ra}\prec 1.$
%\end{split}
%\end{align}
Moreover, $|A_{aa}|\leq 1$ as a consequence of $\|A\|\leq 1$. Therefore, by definition of graded polynomials, \eqref{eqn:expansion-rar-aa2} is
\begin{align}\label{eqn:rar-aa-second-line}
\begin{split}
&|R_{aa}|^2\left(\sum_{s\notin \{a,b\}}\left(A\rab\right)_{as}h_{as}+\sum_{r\notin \{a,b\}}\left(\ra A\right)_{ra}h_{ar}+A_{aa}\right)\\
&=\Oparity{\odd}{1}+\Oparity{\even}{1}+O_\prec(N^{-D}).
\end{split}
\end{align}
Now \eqref{eqn:parity-rar-aa} follows from \eqref{eqn:rar-aa-first-line} and \eqref{eqn:rar-aa-second-line}.

Next, for \eqref{eqn:parity-rar-ab}, we have by similar reasoning that
\begin{align}
(R A \bar R)_{ac} &= \sum_{i,j} R_{ai} A_{ij} {\bar R}_{jc}\\
&= 
\sum_{i,j\neq a } R_{ai} A_{ij}{\bar R}_{jc} + 
(RA)_{aa} {\bar R}_{ac} + R_{aa} (A {\bar R})_{ac} 
- R_{aa} A_{aa} {\bar R}_{ac}\\
&=
\sum_{i,j\neq a } R_{ai} A_{ij} {\bar R}_{jc} 
+ \Oparity{\odd}{\Psi} + 
 \Oparity{{even}}{|A_{ac}|}
 + \Oparity{\even}{\Psi}
+O_\prec(N^{-D})
\end{align}
%\Oparity{\odd}{1}+O_\prec(N^{-D})
Further, by \eqref{eqn:neu-series-3} and the resolvent identities used previously,
\begin{align}
\sum_{i,j\neq a } R_{ai} A_{ij} {\bar R}_{jc}
= - {R}_{aa} \sum_{r\notin \{a,b\}} h_{ar} ({{R}^{(a)}} A {\bar R}^{(a)})_{rc} - |{R}_{aa}|^2 \sum_{r,s,t\notin \{a,b\}} h_{ar} h_{as} h_{at} {\bar R}^{(a)}_{sc} ({{R}}^{(a)}  A {\bar R}^{(a)})_{rt},
\end{align}
and the claim follows from \eqref{eqn:2-resolvents-traceless-ra}.

For \eqref{eqn:parity-rar-bb}, we note that 
\begin{align}
(R A \bar R)_{cc} &= \sum_{i,j} R_{ci} A_{ij} {\bar R}_{jc}\\
&= 
\sum_{i,j\neq a } R_{ci} A_{ij} {\bar R}_{jc} + 
(RA)_{ca} {\bar R}_{ac} + R_{ca} (A {\bar R})_{ac} 
- R_{ca} A_{aa} {\bar R}_{ac}\\
&=
\sum_{i,j\neq a } R_{ai} A_{ij} {\bar R}_{jc} 
+ \Oparity{\odd}{\Psi} %+ 
 %\Oparity{\odd}{|A_{ac}|}
 + \Oparity{\even}{\Psi}
+O_\prec(N^{-D}),
\end{align}
as the leading-order term can be bounded as before.

Turning to \eqref{eqn:parity-rar-trace}, we note that 
\begin{align*}
\Tr\left(RA\bar R\right)
& = \sum_{i \neq a} \sum_{j,k} R_{ij} A_{jk} \bar R_{ki} + (RA \bar R)_{aa}\\
&=  \sum_{i, j,k} R^{(a)}_{ij} A_{jk} \bar R^{(a)}_{ki}-
R^{(a)}_{aa} A_{aa} \bar R^{(a)}_{aa}
+ (RA \bar R)_{aa}\\
&= \Tr\rar +
\Oparity{\even}{1}
+ 
\Oparity{\even}{N^{1/2}\Psi}+\Oparity{\odd}{1}+O_\prec(N^{-D}),
\end{align*}
where we used \eqref{eqn:parity-rar-aa} in the last line. Then \eqref{eqn:parity-rar-trace} follows after noting the errors above are bounded by the claimed error terms.
\end{proof}
\begin{lemma}\label{l:march-fix}
Fix $D>0$. For all spectral parameters $z=E+\iu\eta$ satisfying $E\in I_\ell$ and $\eta=\eta_\ell$, and indices $c \neq a$,
\begin{align}
(ARR)_{a a}  &=\Oparity{\even}{N\Psi^2 }+\Oparity{\odd}{N\Psi^2 }+O_{\prec}\left(N^{-D}\right),\label{e:arr-parity}
\\
(ARR)_{a c}  &=\Oparity{\even}{N\Psi^2 }+\Oparity{\odd}{N\Psi^2 }+O_{\prec}\left(N^{-D}\right).\label{e:arr-parity-2}
\end{align}
\end{lemma}
\begin{proof}
For the first estimate, we have
\begin{align}
(ARR)_{a a} = A_{aa} (RR)_{aa}  + \sum_{i\neq a} A_{ai} R_{ia} R_{aa} +
\sum_{i,j\neq a} A_{ai} R_{ij} R_{ja}.
\end{align}
We have 
\begin{align}
A_{aa} (RR)_{aa} = \Oparity{\even}{N\Psi^2}+O_{\prec}(N^{-D}),
\end{align}
by \Cref{lem:parity-rr}. We also have 
\begin{align}
R_{aa} \sum_{i \neq a} A_{ia} R_{ia}  = \Oparity{\odd}{1}
\end{align} 
by \eqref{e:useabove}. We expand 
\begin{align}
\begin{split}\label{e:arr-int-step}
\sum_{i,j\neq a} A_{ai} R_{ij} R_{ja}
& = 
\sum_{i,j\neq a} A_{ai} \left( R_{i j}^{(a)}+R_{a a} \sum_{r,s \notin \{a,b\}} R_{i r}^{(a)} R_{sj}^{(a)} h_{ra} h_{a s}  \right) \left(-R_{a a} \sum_{t \notin \{a,b\}} R_{jt}^{(a)} h_{ta}  \right)\\
&= - R_{aa} \sum_{t \neq a } h_{ta} (A R^{(a)}
R^{(a)})_{at}  - R_{aa}^{{2}} \sum_{r,s,t \neq a }
h_{ra} h_{sa} h_{ta} (A R^{(a)})_{ar} (R^{(a)}
R^{(a)})_{st}.
\end{split}
\end{align}
We observe that 
\begin{align}\label{e:bound_arr}
(AR^{(a)} R^{(a)})_{rs} \prec N\Psi^2 
\end{align}
for any $r,s$ with $s\neq a$ (and analogous claims with one or both of the resolvents conjugated). To justify it, note that 
\[
(AR^{(a)} R^{(a)})_{rs}  = 
\langle \e_r, AR^{(a)} R^{(a)} \e_s \rangle  = 
\langle A\e_r, R^{(a)} R^{(a)} \e_s \rangle,
\]
then recall from \eqref{eqn:2-resolvents-ra} that \[(R^{(a)} R^{(a)})_{\x s} \prec N\Psi^2\] for any $\x$ such that $\| \x \| \le 1$. Using \eqref{e:bound_arr} in \eqref{e:arr-int-step}, we obtain
\begin{align}
\sum_{i,j\neq a} A_{ai} R_{ij} R_{ja}
= \Oparity{\odd}{N\Psi^2 }.
\end{align}
This completes the proof of \eqref{e:arr-parity}.

For \eqref{e:arr-parity-2}, we write 
\begin{align}
(ARR)_{a c} = A_{aa} (RR)_{ac}  + \sum_{i\neq a} A_{ai} R_{ia} R_{ac} +
\sum_{i,j\neq a} A_{ai} R_{ij} R_{jc}.
\end{align}
We have 
\begin{align}
A_{aa} (RR)_{ac} = \Oparity{\odd}{N\Psi^2}+O_{\prec}(N^{-D}),
\end{align}
by \Cref{lem:parity-rr}. Further,
\begin{align}
\sum_{i\neq a} A_{ai} R_{ia} R_{ac}
= R_{ac} \sum_{i\neq a} A_{ai} R_{ia} = {\Oparity{\even}{\Psi}},
\end{align}
by \eqref{e:useabove} {and \eqref{eqn:parity-ab}}. Finally, we expand
\begin{align*}
&\sum_{i,j\neq a} A_{ai} R_{ij} R_{jc}\\
&=
\sum_{i,j\neq a} 
A_{ai} \left( R_{i j}^{(a)}+R_{a a}\left( \sum_{r \notin \{a,b\}} R_{i r}^{(a)} h_{ra} \right) \left( \sum_{s \notin \{a,b\}} h_{a s} R_{sj}^{(a)} \right) \right) \\
&\qquad \times 
\left( R_{jc}^{(a)}+R_{a a}\left( \sum_{t \notin \{a,b\}} R_{j t}^{(a)} h_{ta} \right) \left( \sum_{u  \notin \{a,b\}} h_{a u} R_{u c}^{(a)} \right) \right)\\
&=
(A R^{(a)}R^{(a)})_{ac} + R_{aa} \sum_{r,s \notin \{a,b\}} h_{ra} h_{sa} (A R^{(a)})_{ar} (R^{(a)}R^{(a)})_{sc}
+ R_{aa} \sum_{t,u \notin \{a,b\}} h_{ta} h_{ua}
 (A R^{(a)} R^{(a)})_{at} R^{(a)}_{uc}
\\ &\quad + R^2_{aa} \sum_{r,s,t,u \notin \{a,b\}} h_{ta} h_{ua}h_{ra} h_{sa} (A R^{(a)})_{ar} (R^{(a)}R^{(a)})_{st} R^{(a)}_{uc}.
\end{align*}
Bounding these terms as before completes the proof. 
\end{proof}

\begin{lemma}\label{lem:parity-rar-rarr-2}
Fix $D>0$. For all spectral parameters $z=E+\iu\eta$ satisfying $E\in I_\ell$ and $\eta=\eta_\ell$, and indices $c\neq a$,
we have
\begin{align}
(R A \bar{R} R)_{a a}  &=\Oparity{\even}{N^{3/2}\Psi^{9/4}}+\Oparity{\odd}{N\Psi^2}+O_{\prec}\left(N^{-D}\right),\label{eqn:parity-rarr-aa}\\
(R A \bar{R} R)_{a c}  &=\Oparity{\odd}{N^{3/2}\Psi^{9/4}}+\Oparity{\even}{N\Psi^2}+O_{\prec}\left(N^{-D}\right),\label{eqn:parity-rarr-ab}\\
(R A \bar{R} R)_{cc}  &=\Oparity{\even}{N^{3/2}\Psi^{9/4}}  +\Oparity{\odd}{N\Psi^3} + \Oparity{\odd}{N\Psi^2 |A_{ac}|} +O_\prec\left(N^{-D}\right),\label{eqn:parity-rarr-bb}\\
(RA\bar RR)_{ca}&=\Oparity{\odd}{N^{3/2}\Psi^{9/4}}+\Oparity{\even}{N\Psi^3} + \Oparity{\even}{N\Psi^2 |A_{ac}|} +O_\prec(N^{-D}).\label{eqn:parity-rarr-ba}
\end{align}
\end{lemma}
\begin{proof}
For \eqref{eqn:parity-rarr-aa}, we have 
\begin{align}
\begin{split}\label{e:rarr-prelim}
(R A \bar{R} R)_{a a}  &=
\sum_{i,j,k \neq a } R_{ai} A_{ij} \bar R_{jk} R_{ka}
+  \sum_{j,k  } R_{aa} A_{aj} \bar R_{jk} R_{ka}
\\
&\quad + \sum_{i \neq a } \sum_j R_{ai} A_{ij} \bar R_{ja} R_{aa} +  \sum_{i, k \neq a } R_{ai} A_{ia} \bar R_{ak} R_{ka}
\end{split}
\end{align}
We begin with the second, third, and fourth sums, with are lower-order. The second sum is
\begin{align}
R_{aa} \sum_{j,k} A_{aj} \bar R_{jk} R_{ka}
= R_{aa} (A \bar R R)_{aa} = 
\Oparity{\even}{ N \Psi^2 }
+ \Oparity{\odd}{N \Psi^2} + O_\prec(N^{-D}), 
\end{align}
by \eqref{eqn:parity-rr-aa} and \eqref{e:arr-parity}. 
The third sum is
\begin{align}
(RA{\bar R})_{aa} R_{aa} - R_{aa}(A{\bar R})_{aa} R_{aa} =\Oparity{\even}{N^{1/2}\Psi}+\Oparity{\odd}{1}+O_\prec(N^{-D}) , 
\end{align}
by \eqref{eqn:parity-rar-aa} and \Cref{l:march-fix}. 
The fourth sum is 
\begin{align}\label{e:fourthsum1}
\begin{split}
 \sum_{i, k \neq a } R_{ai} A_{ia} \bar R_{ak} R_{ka} &= \sum_{i, k \neq a }  \left(-R_{a a} \sum_{r \notin \{a,b\}} h_{a r} R_{ri}^{(a)}\right) A_{ia} \left| R_{a a} \sum_{s \notin \{a,b\}} h_{a s} R_{ks}^{(a)}\right|^2\\
 &= -R_{aa} |R_{aa}|^2 
\left(
\sum_{r\notin \{a,b\}} h_{ra} (R^{(a)} A)_{ra}
\right)
 \left( \sum_{s,t \notin \{a,b\}} h_{as} h_{at} (R^{(a)} \bar R^{(a)})_{st} \right)\\
 &= 
\Oparity{\odd}{N \Psi^2}. 
\end{split}
\end{align}

Continuing from \eqref{e:rarr-prelim},  the first (leading-order) sum is 
\begin{align}\label{e:march-mainterm}
&\sum_{i,j,k\neq a } R_{ai} A_{ij} \bar R_{jk} R_{ka}
\\
&= 
\sum_{i,j,k}
\left(-R_{a a} \sum_{r \notin \{a,b\}} h_{a r} R_{ri}^{(a)}\right) A_{ij}
\left( \bar R_{jk}^{(a)}+ \bar R_{a a} \sum_{s,t \notin \{a,b\}} \bar R_{j s}^{(a)}\bar R_{tk}^{(a)} h_{ra}  h_{a t}   \right)
\left(-R_{a a} \sum_{u \notin \{a,b\}} h_{a u} R_{ku}^{(a)}\right) \notag\\
&= R_{aa}^2 \sum_{r{,u}\notin \{a,b\}} h_{ar} h_{au}
( R^{(a)} A \bar R^{(a)} R^{(a)})_{ru}
+
|R_{aa}|^2 R_{aa}\hspace{-1.5em}\sum_{r, s, t, u \notin \{a,b\}} h_{ar} h_{as} h_{at} h_{au}
( R^{(a)} A \bar R^{(a)})_{rs} (\bar R^{(a)} R^{(a)})_{tu}.
\notag
\end{align}
These terms are all even, and can be bounded using \Cref{lem:parity-r}, \Cref{lem:parity-rr}, \Cref{lem:parity-rar-rarr},{and \eqref{eqn:3-resolvents-traceless-ra}}. This completes the argument for  \eqref{eqn:parity-rarr-aa}. The computation for \eqref{eqn:parity-rarr-ab} is extremely similar, and hence omitted.

Next, we prove \eqref{eqn:parity-rarr-ba}; the proof of \eqref{eqn:parity-rarr-aa} is analogous and omitted. 
We
\begin{align}
\begin{split}\label{e:rarr-prelim-2}
(R A \bar{R} R)_{c a}  &=
\sum_{i,j,k \neq a } R_{ci} A_{ij} \bar R_{jk} R_{ka}
+  \sum_{j,k  } R_{ca} A_{aj} \bar R_{jk} R_{ka}
\\
&\quad + \sum_{i \neq a } \sum_j R_{ci} A_{ij} \bar R_{ja} R_{aa} +  \sum_{i, k \neq a } R_{ci} A_{ia} \bar R_{ak} R_{ka}.
\end{split}
\end{align}
The first term is $\Oparity{\odd}{N^{3/2}\Psi^{9/4}}$, as can be shown nearly identically to the bound for \eqref{e:march-mainterm}. The second sum is
\begin{align}
R_{ac} \sum_{j,k} A_{aj} \bar R_{jk} R_{ka}
= R_{ac} (A \bar R R)_{aa} = 
\Oparity{\even}{ N \Psi^3 }
+ \Oparity{\odd}{N \Psi^3} + O_\prec(N^{-D}), 
\end{align}
by \eqref{eqn:parity-rr-ab} and \eqref{e:arr-parity}. The third sum is 
\begin{align}
(RAR)_{ca} R_{aa} - R_{ca}(AR)_{aa} R_{aa} = \Oparity{\odd}{N^{1/2} \Psi }+\Oparity{\even}{\Psi} + O_\prec(N^{-D}).
\end{align}
For the fourth sum, we have 
\begin{align}
\sum_{i, k \neq a } R_{ci} A_{ia} \bar R_{ak} R_{ka} =
\left(\sum_{i\neq a}R_{ci} A_{ia} \right) \left( \sum_{k\neq a} \bar R_{ak} R_{ka}\right)
\end{align}
It was shown in \eqref{e:fourthsum1} that 
\begin{align}
\sum_{k\neq a} \bar R_{ak} R_{ka}
= \Oparity{\even}{N \Psi^2}
\end{align}
Further, by \eqref{e:ar-odd},
\begin{align}
\sum_{i\neq a}R_{ci} A_{ia}
= (RA)_{ca} {-} R_{ca} A_{aa} = 
\Oparity{\odd}{{\Psi} }+
\Oparity{\even}{\Psi} + \Oparity{\even}{|A_{ac}|} + O_{\prec}\left(N^{-D}\right).
\end{align}
This completes the proof.
\end{proof}
\begin{remark}
    We note that the bounds we give for the even graded polynomials in \eqref{eqn:parity-rarr-ab} and \eqref{eqn:parity-rarr-ba} differ in the subleading error terms. The more refined bound for the latter quantity is needed in the proof of \Cref{lem:parity-x3}. 
\end{remark}

Using Lemma \ref{lem:parity-r}, Lemma \ref{lem:parity-rr}, Lemma \ref{lem:parity-rar-rarr} and the definitions of $x_i,y_i$ in Lemma \ref{lem:resolvent-expansion}, we have the following lemma. We omit the proof, since it is a straightforward adaptation of the proof of Lemma~\ref{lem:resolvent-expansion}. 
\begin{lemma}\label{lem:parity-xy}
Let $x_i,i=1,2$ and $y_i,i=1,2,3$ be defined as in Lemma \ref{lem:resolvent-expansion}. For spectral parameters $z=E+\iu\eta$ satisfying $E\in I_\ell$ and $\eta=\eta_\ell$,
we have
\begin{align*}
x_1&=\Oparity{\odd}{N^{1+\delta/2}\Psi^{5/4}}+\Oparity{\even}{N^{1/2+\delta/2}\Psi}+O_\prec(N^{-D}),\\
x^R,x_2&=\Oparity{\even}{N^{1+\delta/2}\Psi^{5/4}}+\Oparity{\odd}{N^{1/2+\delta/2}\Psi}+O_\prec(N^{-D}),\\
y_1,y_3&=\Oparity{\odd}{N^{6 \epsilon_1} \Psi}+O_\prec(N^{-D}),\\
y^R, y_2&=\Oparity{\even}{N^{6 \epsilon_1}  \Psi}+O_\prec(N^{-D}).
\end{align*}
\end{lemma}

To bound $x_3$, we need an extra lemma. 
\begin{lemma}\label{lem:parity-x3}
Denote by 
%\begin{align*}
%\begin{split}
    $\Oparity{\odd,b}{K}\,$%,\Oparity{\odd,b}{K}
%\end{split}
%\end{align*}
 the graded polynomial expanded in the $b$-th row and column of $Q$ instead of $a$-th row and column. For spectral parameters $z=E+\iu\eta$ satisfying $E\in I_\ell$ and $\eta=\eta_\ell$,
we have
%\begin{align*}
\begin{align}\label{eqn:parity-x3}
\begin{split}
    x_3=&\Oparity{\odd}{N^{1+\delta/2}\Psi^{5/4}}+\Oparity{\odd,b}{N^{1+\delta/2}\Psi^{5/4}}\\
    & + \Oparity{\even,b}{N^{1/2+\delta/2} \Psi^2  } + 
    \Oparity{\even}{N^{1/2+\delta/2} \Psi^2  }\\ &+ \Oparity{\even,b}{N^{1/2+\delta/2} \Psi |A_{ab}|   } + 
    \Oparity{\even}{N^{1/2+\delta/2}\Psi  |A_{ab}|   }+\Od
\end{split}
\end{align}
%
%\end{align*}
\end{lemma}
\begin{proof}
Note that, by the definition of $x_3$, the resolvent terms in $x_3$ come from the third line of \eqref{eqn:resolvent-expansion}. The resolvent terms in the third line of \eqref{eqn:resolvent-expansion} have one of the following forms:
\begin{equation}
 (R A \bar R R)_{**} R_{**} R_{**}
, \quad  
( \bar R R)_{**} R_{**} (R A \bar R)_{**}, \quad (RA \bar R )_{**} \bar R_{**} (\bar R R)_{**}  ,\quad 
(\bar R R A R )_{**} \bar R_{**} \bar R_{**}.
\end{equation}
Here, each $*$ denotes an index that is either $a$ or $b$. Further, in adjacent factors in the products, the second $*$ in the first factor must differ from the first $*$ in the second factor (and analogously for the first and last factors). We will discuss how to handle the contributions from the first two kinds of terms; the latter two kinds are analogous (since, up to conjugation and symmetry, they are the same as the others).

By \eqref{eqn:resolvent-expansion} and the definition of $x_3$, 
\begin{align}
%\begin{split}
    \frac{\pi \sqrt{\Tr(A^2)}}{N\eta_\ell} x_3
    =& -(RA\bar RR)_{ab}R_{aa}R_{bb}-(RA\bar RR)_{ba}R_{bb} R_{aa}\label{e:rarr1}\\
    & -(RA\bar RR)_{ab}R_{ab}R_{ab}-(RA\bar RR)_{ba}R_{ba} R_{ba}\label{e:rarr2}\\
    & -(RA\bar RR)_{aa}R_{bb}R_{ab}-(RA\bar RR)_{bb}R_{ab} R_{aa}\label{e:rarr3}\\
    &-(RA\bar R)_{ab}(\bar R R)_{aa}R_{bb}-(RA\bar R)_{ba}(\bar R R)_{bb}R_{aa}\label{e:rar1}\\ 
    &-(RA\bar R)_{ab}(\bar R R)_{ab}R_{ab}-(RA\bar R)_{ba}(\bar R R)_{ba}R_{ba}\label{e:rar2}\\ 
    &  -(RA\bar R)_{aa}(\bar R R)_{ba}R_{bb}-\left(RA\bar R\right)_{bb}(\bar R R)_{ab}R_{aa} +[\dots],\label{e:rar3}
%\end{split}
\end{align}
where $[\dots]$ denotes terms omitted according to the previous discussion. 
For the first term on the right-hand side of \eqref{e:rarr1}, using resolvent identities with respect to the $b$-th row and column as in the proof of Lemma \ref{lem:parity-r} {and Lemma \ref{lem:parity-rar-rarr-2}}, we have
\begin{align}
\begin{split}\label{e:alt-expand}
    (RA\bar RR)_{ab}&=\Oparity{\odd,b}{N^{3/2}\Psi^{9/4}}+\Oparity{\even,b}{N\Psi^3} + \Oparity{\even,b}{N\Psi^2|A_{ab}|} +\Od,\\
    R_{aa}&=\Oparity{\even,b}{1}+\Od,\\
    R_{bb}&=\Oparity{\even,b}{1}+\Od,
\end{split}
\end{align}
and combining these estimates gives the desired bound. The second term in \eqref{e:rarr1} is bounded similarly. The terms in lines \eqref{e:rarr2} and \eqref{e:rarr3} are bounded using Lemma \ref{lem:parity-r} {and Lemma \ref{lem:parity-rar-rarr-2}}; these are simpler to bound than the previous line, due to the presence of additional off-diagonal resolvent entries {and the fact that $(RA\bar RR)_{aa}$ and $(RA\bar RR)_{ab}$ have the same bound in $\Oparity{*,b}{}$ in Lemma \ref{lem:parity-rar-rarr-2}.}

Similarly, the lines \eqref{e:rar1}, \eqref{e:rar2}, and \eqref{e:rar3} 
are quickly bounded using \Cref{lem:parity-r}, \Cref{lem:parity-rr}, and \Cref{lem:parity-rar-rarr}. For the $(RA\bar R)_{aa}$ term in \eqref{e:rar3}, one uses the estimate analogous to \eqref{eqn:parity-rar-bb} coming from expanding in $b$ (as in \eqref{e:alt-expand}) and treats $a$ as an off-diagonal entry.
%The conclusion follows by applying Lemma \ref{lem:parity-r}, Lemma \ref{lem:parity-rr} and Lemma \ref{lem:parity-rar-rarr} to the remaining terms of \eqref{e:rar3}.
\end{proof}
\begin{remark}
    Note that it is still legitimate to apply Lemma \ref{lem:odd-poly} to \eqref{eqn:parity-x3}. By linearity of expectation, we may apply Lemma \ref{lem:odd-poly} to $\Oparity{\odd}{K}$ and $\Oparity{\odd,b}{K}$ separately.
\end{remark}

We are now ready for the proof of Lemma \ref{lem:third-moment-terms}.
\begin{proof}[Proof of \Cref{lem:third-moment-terms}]
Recall from \Cref{lem:parity-xy} that 
\begin{equation}\label{e:yRrenew}
y^R=\Oparity{\even}{N^{6 \epsilon_1}  \Psi}+O_\prec(N^{-D}).
\end{equation}
Note that for $N \ge N(\tau)$, this implies (recall \Cref{para}) that 
\begin{equation}\label{e:yR-small}
y^R = O_\prec(N^{-\epsilon_1}).
\end{equation}
By Taylor expansion around $0$, we have for every integer $K \ge 1$ that 
\begin{align}\label{e:q-taylor}
q(y^R)  = \sum_{j=0}^K \frac{ q^{(j)}(0) }{j!} (y^R)^j + E_K,
\end{align}
where $E_K$ is a $K$-dependent error term satisfying
\begin{equation}
| E_K | \le C_K \| q^{(K+1)}\|_\infty |y^R|^{K+1}.
\end{equation}
For $K \ge K_0(\tau,D)$, we have by \eqref{e:yR-small} that $|y^R|^{K+1} = O_\prec( N^{-D})$, and hence $E_K = O_\prec( N^{-D})$. Recalling \eqref{e:yRrenew} and \eqref{e:q-taylor}, this gives 
\begin{equation}\label{e:qyR}
q(y^R) = \Oparity{\even}{1} + O_\prec( N^{-D}),
\end{equation}
where the leading-order term comes from the $j=0$ term in \eqref{e:q-taylor}. 
Similarly, for $k=1,2,3$, 
\begin{equation}
q^{(k)}(y^R) = \Oparity{\even}{1} + O_\prec( N^{-D}).
\end{equation}

Next, note that by \Cref{lem:parity-xy} and \eqref{e:qyR},
\begin{align}
\int_{I_\ell}x^Rq(y^R)\, \d E&=\Oparity{\even}{N^{3\delta/2+{\delta_1}} \Psi^{1/4} }+\Oparity{\odd}{N^{-1/2+3\delta/2+{\delta_1}}}
\end{align}
{where we used the fact that 
\[
\vert I_\ell\vert = 2N^{\delta_2}\Delta_\ell = 2N^{\delta_2+\delta_1}\eta_\ell = O\left( \frac{N^{\delta+\delta_1}}{N\Psi}\right).
\]
by Definition \ref{def:regularized}, Definition \ref{para}, and Lemma \ref{lem:prelim-bound}}.
Recall that $g$ is smooth and compactly supported. 
By 
Taylor expansion around $0$ (as in the argument for \eqref{e:qyR}), for all $k=1,2,3$ we have
\begin{align}
g^{(k)}\left(\int_{I_\ell}x^Rq(y^R)\, \d E\right)&=\Oparity{\even}{1 }+\Oparity{\odd}{N^{-1/2+\delta/2}}+ O_\prec( N^{-D}).\label{eqn:parity-g}
\end{align}
The same graded polynomial expansion holds for the expansion with respect to the $b$-th row and column.

We now proceed to bound the terms identified in the lemma statement. Define
\begin{align*}
    \Psi_\ell=(N\Delta_\ell)^{-1} \qquad
    \widehat N=N^{3\delta/2 + {\delta_1} + 18 \epsilon_1} %\quad{\text{Could just put a power of }\varepsilon_1}
    , 
\end{align*}
and note that 
\begin{align*}\Psi_\ell \le N^{-\tau/3},\qquad \widehat N \le N^{\tau/100}. 
\end{align*}
Combining Lemma \ref{lem:parity-xy}, the definition of $P_{\bm l}$, Remark \ref{rmk:stochastic-continuation}, and \eqref{eqn:psi-bound-specified}, we have
\begin{equation*}
\begin{aligned}
P_{(1)},P_{(0,1)}&=\Oparity{\odd}{\widehat N\Psi_\ell^{1/4}}+\Oparity{\even}{\widehat NN^{-1/2}}+\Od,\\
P_{(2)},P_{(0,2)},P_{(1,1)},P_{(0,1,1)}&=\Oparity{\even}{\widehat N\Psi_\ell^{1/4}}+\Oparity{\odd}{\widehat NN^{-1/2}}+\Od.
\end{aligned}
\end{equation*}
We also have
\[
    P_{(0,3)},P_{(1,2)},P_{(2,1)},P_{(0,1,2)},P_{(1,1,1)},P_{(0,1,1,1)}=\Oparity{\odd}{\widehat N\Psi_\ell^{5/4}}+\Oparity{\even}{\widehat NN^{-1/2}\Psi_\ell}+\Od.
\]
By Lemma \ref{lem:parity-x3}, 
\begin{align*}
        %\begin{aligned}
    P_{(3)}=&\Oparity{\odd}{\widehat N\Psi_\ell^{1/4}}+\Oparity{\odd,b}{\widehat N\Psi_\ell^{1/4}}\\
    &+\Oparity{\even}{\widehat NN^{-1/2}|A_{ab}| }+\Oparity{\even,b}{\widehat NN^{-1/2}|A_{ab}|}\\
&+\Oparity{\even}{\widehat NN^{-1/2}\Psi_\ell}+\Oparity{\even,b}{\widehat NN^{-1/2}\Psi_\ell}+\Od.
    %\end{aligned}
\end{align*}
Combining these bounds with \eqref{eqn:parity-g} and Lemma \ref{lem:odd-poly}, we conclude that there exists a constant $c =  c(\tau)>0$ such that
\begin{align*}
\begin{split}
    \E [Y]\prec N^{-1/2-c}(1 + |A_{ab}| \Psi^{-1}) +N^{-D}\leq N^{-1/2-c}(1 + |A_{ab}| \Psi^{-1}),
\end{split}
\end{align*}
when $Y$ represents any term in \eqref{eqn:third-moment-terms}.
\end{proof}

\begin{remark}
We chose to prove \Cref{lem:third-moment-terms} using expansions based on the Schur complement formula for convenience. It likely also possible to prove it using cumulant expansions, as in \Cref{sec:local-law}, but we do not pursue this alternative here. 
\end{remark} 

\begin{remark}\label{r:theproblem}
We now comment on the hypothesis  $\Tr(A^2)\geq N^{1-\delta}$ in \Cref{thm:CLT}. The theorem should surely hold under a much weaker condition, for instance $\Tr(A^2)\geq N^\delta$. 
However, some of our technical inputs do not seem sharp enough to established this improved result. Consider, for instance, the estimate on $x_3$ in \eqref{e:rar3}. If $\Tr(A^2)$ is made smaller, this must be offset by an improved estimate on terms such as $(RAR)_{aa}$ to obtain the same bound for $x_3$. However, our estimates on the entries of $RAR$ are not sensitive to the size of $\Tr(A^2)$.

Specifically, the proof of \eqref{eqn:2-resolvents-traceless} in \Cref{sec:local-law} uses \eqref{eqn:center-iso} and \eqref{eqn:iso}. These bounds appear suboptimal for $A$ such that $\Tr(A^2)$ is small; consider, for example, the matrix $A$ with a single entry nonzero entry, $A_{12} = 1$. Then \eqref{eqn:2-resolvents-traceless} gives a bound of order $N^{1/2} \Psi$ for $(GAG)_{11}$, but $(GAG)_{11} = G_{11} G_{12}$ has order $\Psi$, by \eqref{eqn:2-resolvents}. 

\end{remark}

\appendix
%\newpage
\section{Proof of Theorem~\ref{thm:GOE-CLT}}\label{s:goeclt}
\begin{proof}[Proof of \Cref{thm:GOE-CLT}]

Because the distribution of $H$ is invariant under conjugation by orthogonal matrices, the eigenvector $\bm u$ of $H$ is uniformly distributed on $\mathbb S^{N-1}$.
Then by diagonalizing $A$, we may assume without loss of generality that $A$ is a diagonal matrix with diagonal entries $\bm \mu=(\mu_1,\ldots,\mu_N)$. By the assumptions of the theorem, we have 
\begin{align}
        \sum_{i=1}^N \mu_i&=0,\label{sum0}\\
        \sum_{i=1}^N \mu_i^2&\geq N^{1-\delta},\label{sum2}\\
        \max_{1\leq i\leq N}|\mu_i|&\leq 1\label{max1}.
\end{align}
By radial symmetry of the multi-dimensional gaussian distribution, it follows that
\begin{equation*}
    \bm u\stackrel{(d)}{=}\frac{\bm g}{\|\bm g\|},
\end{equation*}
where $\bm g=(g_1,\ldots,g_N)\in \mathbb R^N$ consists of independent standard gaussian random variables. Note that, by \eqref{sum0}, 
\[\sum_{i=1}^N\mu_ig_i^2=\sum_{i=1}^N\mu_i(g_i^2-1),\] so it suffices to show that
\begin{equation}\label{goal7}
    \frac{N}{\sqrt 2\|\bm \mu\|}\frac{\sum_{i=1}^N\mu_i(g_i^2-1)}{\|\bm g\|^2}\rightarrow \mathcal N(0,1)
\end{equation}
in distribution.

To this end, we check the \textit{Lindeberg's condition} for the sum
\begin{equation*}
\frac{1}{\sqrt 2\|\bm \mu\|}\sum_{i=1}^N\mu_i(g_i^2-1).
\end{equation*}
Fix any $\epsilon >0 $. For sufficiently large $N\ge N_0(\epsilon)$, we have
\begin{align*}
%\begin{split}
    &\sum_{i=1}^N\E\left[\frac{\mu_i^2\left(g_i^2-1\right)^2}{\|\bm \mu\|^2}\bm 1_{(\epsilon,\infty)}\left(\frac{|\mu_i(g_i^2-1)|}{\|\bm \mu\|}\right)\right]\\
    &=\sum_{i=1}^N\int_{\epsilon^2}^\infty \P\left(\mu_i^2\left(g_i^2-1\right)^2>x\|\bm \mu\|^2\right)dx+\sum_{i=1}^N\epsilon^2\P\left(\mu_i^2\left(g_i^2-1\right)^2>\epsilon^2\|\bm \mu\|^2\right)\\
    &\leq N\int_{\epsilon^2}^\infty \P\left((g_1^2-1)^2>N^{1-\delta}x\right)dx+N\epsilon^2\P\left((g_1^2-1)^2>N^{1-\delta}\epsilon^2\right)\\
    &\leq N\int_{\epsilon^2}^\infty \P\left(|g_1|>\frac{1}{2} N^{(1-\delta)/ {4}}x^{1/4}\right)dx +N\epsilon^2\P\left(|g_1|>\frac{1}{2} N^{(1-\delta)/ {4}}\epsilon^{1/2}\right)\\
    &\leq 
     4 N^{1-(1-\delta)/2} \epsilon^{-1} 
     \exp\left( - \frac{\epsilon N^{(1-\delta)/2}}{8} \right)
    +  2  N^{1-(1-\delta)/4} \epsilon^{3/2} %\epsilon^{1/2} 
    \exp\left(- \frac{\epsilon N^{(1-\delta)/2}}{8}\right),
%\end{split}
\end{align*}
where we use \eqref{sum2} and \eqref{max1} in the first inequality, and a standard Gaussian tail bound in the last inequality.
Then 
\begin{equation}\label{lindeberg}
\begin{aligned}
    \lim_{N\rightarrow\infty}\sum_{i=1}^N\E\left[\frac{\mu_i^2\left(g_i^2-1\right)^2}{\|\bm \mu\|^2}\bm 1_{(\epsilon,\infty)}\left(\frac{|\mu_i(g_i^2-1)|}{\|\bm \mu\|}\right)\right]=0.
\end{aligned}
\end{equation}
Together with \eqref{lindeberg}, $\E\left[\mu_i(g_i^2-1)\right]=0$,  and $$\operatorname{Var}\left(\frac{\sum_{i=1}^N\mu_i(g_i^2-1)}{(\sqrt 2\|\bm \mu\|)}\right)=1,$$
Lindeberg's central limit theorem \cite[Theorem~8.13]{ccinlar2011probability} implies that
\begin{equation*}
\begin{aligned}
    \frac{\sum_{i=1}^N\mu_i(g_i^2-1)}{\sqrt 2\|\bm \mu\|}\rightarrow\mathcal N(0,1)
\end{aligned}
\end{equation*}
in distribution. Combining this and the almost sure convergence 
\begin{equation*}
\frac{N}{\|\bm g\|^2}\rightarrow 1
\end{equation*}
guaranteed by the law of large numbers finishes the proof of \eqref{goal7}.
\end{proof}
\begin{remark}\label{r:momcvg}
The conclusion of the theorem can be strengthened to convergence in moments. See \cite[Theorems~2.3 and 2.4]{o2016eigenvectors} for the case where $A$ is a projection; the general case can be proved by straightforward, but tedious, moment computations. Using this improved result, the conclusion of Theorem~\ref{thm:CLT} can also be strengthened to convergence in moments.
\end{remark}

\section{Proof of Lemma \ref{lem:three-resolvent-local-law}}\label{sec:local-law}
The proof is based on the following cumulant expansion lemma, which can be found in \cite[Lemma 3.2]{lee2018local}.
%We will need the following cumulant expansion lemma in the proof, which is taken from \cite[Lemma~3.2]{lee2018local}.
\begin{lemma}[Cumulant expansion]\label{lem:cumulant-expansion}
Fix $T \in \mathbb{N}$ and let $F:\R\rightarrow\mathbb C^+$ be $T+1$ times continuously differentiable. Let $Y$ be a mean zero random variable with finite moments to order $T+2$. Then there exists a function $\Omega_T: \mathbb{C} \rightarrow \mathbb{C}$ such that
\begin{equation}\label{eqn:cumulant-theorem}
\mathbb{E}\big[Y F(Y)\big]=\sum_{r=1}^{T} \frac{\kappa^{(r+1)}(Y)}{r !} \mathbb{E}\left[F^{(r)}(Y)\right]+\mathbb{E}\Big[\Omega_{T}\big(Y F(Y)\big)\Big],
\end{equation}
where $\mathbb{E}$ denotes the expectation with respect to $Y, \kappa^{(r+1)}(Y)$ denotes the $(r+1)$-th cumulant of $Y$, and $F^{(r)}$ denotes the $r$-th derivative of the function $F$. Further, there exists a constant $C >0$ (not depending on $T$, $F$, or $Y$) such that for every $Q>0$,
%The error term $\Omega_{\ell}(Y F(Y))$ in (3.20) satisfies
\begin{equation*}
\begin{aligned}
\left|\mathbb{E}\Big[\Omega_{T}\big(Y F(Y)\big)\Big]\right| \leqslant \frac{(C T)^{T}}{T !} \left(   \mathbb{E}\left[|Y|^{T+2} \right] \cdot  \sup_{|t| \le Q }\left|F^{(T+1)}(t)\right| 
 + \mathbb{E}\left[|Y|^{T+2}\one\{|Y| > Q \} \right]\cdot \sup_{t \in \R }\left|F^{(T+1)}(t)\right| \right).
\end{aligned}
\end{equation*}
%where $C_{T}$ satisfies $C_{T} \leqslant \frac{(C T)^{T}}{T !}$ for some numerical constant $C$.
\end{lemma}

% \begin{lemma}Let $H$ be a Wigner matrix, and $G(z)=(H-z)^{-1}$ be its Stieltjes transform. The spectral parameter $z=E+\iu\eta$ satisfies
% \begin{align}
% \begin{split}
%     E\in I_\ell,\quad \eta=\eta_\ell.
% \end{split}
% \end{align}
% Define
% \begin{align}
% \begin{split}
%     \Psi\defeq\frac{1}{N\eta_\ell}=N^{-1/3}\ell^{1/3}N^{\delta_1}\leq N^{-\tau/3+\delta_1}.
% \end{split}
% \end{align}
% Then for any deterministic traceless matrix $A$, with $\|A\|\leq 1$, and index $c,d\in\llbracket1,N\rrbracket$, we have
% \begin{align}
% \begin{split}
%     |(GA\bar GG)_{cd}|\prec N^{3/2}\Psi^{9/4}.
% \end{split}
% \end{align}

% \end{lemma}
\begin{proof}[Proof of Lemma \ref{lem:three-resolvent-local-law}]
First, we claim that it suffices to prove the local laws in \eqref{eqn:2-resolvents} for the resolvent $G$, because the local laws for $R$ are straightforward consequences of the ones for $G$. We illustrate the procedure of deducing a local law for $R$ from the corresponding local law for $G$ for the first inequality in \eqref{eqn:2-resolvents}. The other deductions are similar. 

Pick any $z=E+\iu\eta\in\bm S$. Note that the first claim in \eqref{eqn:2-resolvents} when $S=G$ is just Theorem~\ref{thm:iso-single}. By \eqref{eqn:resolvent-expansion-original}, we have 
\begin{equation}\label{e:error-resolvent}
    R_{\bm x\bm y}=G_{\bm x\bm y}-(GUG)_{\bm x\bm y}+(GUGUG)_{\bm x\bm y}-((GU)^3G)_{\bm x\bm y}+((GU)^4R)_{\bm x\bm y}.
\end{equation}
Combining the estimates $h_{ab}\prec N^{-1/2}$, $\|R\|\leq \eta^{-1}$, $\Psi\geq C\tau^{1/4}N^{-1/2}$ from Lemma~\ref{lem:prelim-bound} and the isotropic local law \eqref{isotropic} for $G$, the first claim in \eqref{eqn:2-resolvents} also holds for $R$. 

In light of the previous discussion, we only prove the other two bounds in \eqref{eqn:2-resolvents} for $G$.
By Theorem \ref{thm:iso-single}, we have
\begin{equation}\label{a3}
\begin{aligned}
    (G\bar G)_{\x\y}=\sum_k G_{\x k}\bar G_{k\y}&\prec \left|\sum_k\langle\x,\e_k\rangle\langle\e_k,\y\rangle\right|+\left|\sum_k\langle\x,\e_k\rangle\right|\Psi+\left|\sum_k\langle\e_k,\y\rangle\right|\Psi+N\Psi^2\\
    &\leq |\langle\x,\y\rangle|+2N^{1/2}\Psi+N\Psi^2\prec N\Psi^2,
\end{aligned}
\end{equation}
where we use Cauchy--Schwarz inequality in the second-to-last inequality and $\Psi\geq C\tau^{1/4}N^{-1/2}$ on $\bm S$ in the last step (from  \eqref{eqn:psi-bound}). Similarly we have
\begin{equation}\label{a4}
    (G G )_{\x \y}\prec N\Psi^2.
\end{equation}
Then  \eqref{eqn:2-resolvents} follows from \eqref{isotropic}, \eqref{a3}, and \eqref{a4}.

Next, the expression \eqref{eqn:2-resolvents-traceless} is a direct consequence of \eqref{eqn:center-iso} and \eqref{eqn:iso} with two resolvents and one traceless matrix (and \eqref{eqn:psi-bound}). It remains to prove \eqref{eqn:3-resolvents-traceless}.

We proceed by computing the moments of the quantity $\left(GA\bar GG\right)_{cd}$. For the rest of the proof, we assume the spectral parameter $z=E+\iu\eta\in\bm S$ satisfies $E\in I_\ell$ and $\eta=\eta_\ell$ as in the statement of Lemma \ref{lem:three-resolvent-local-law}.

Let $D>1$ be a parameter. We have
\begin{equation}\label{eqn:multi-resolvent-moment}
\begin{aligned}
\mathbb{E} \left[z\left|(G A \bar{G} G)_{cd}\right|^{2 D}\right]= %& \mathbb{E} \left[z\left|\sum_{i, j,k} G_{c i}A_{ij} \bar{G}_{jk} G_{k d}\right|^{2 D}\right] \\
%= & \mathbb{E} \left[z \sum_{i, j,k}G_{c  i}A_{ij} \bar{G}_{jk} G_{k d }(G A \bar{G} G)_{c  d }^{D-1}(\bar{G} A G \bar{G})_{c  d }^D \right]\\
& \mathbb{E} \left[z (G A \bar{G} G)_{c  d }(G A \bar{G} G)_{c  d }^{D-1}(\bar{G} A G \bar{G})_{c  d }^D \right]\\
= & \mathbb{E} \left[\sum_{k} h_{c  k} (GA\bar GG)_{kd}(G A \bar{G} G)_{c  d }^{D-1}(\bar{G} A G \bar{G})_{c  d }^D\right] \\
& -\mathbb{E} \left[(A\bar{G} G)_{c  d }(G A \bar{G} G)_{c  d }^{D-1}(\bar{G} A G \bar{G})_{c  d }^D\right],
\end{aligned}
\end{equation}
where we used the identity $z G = HG -I$ in the last equality.

Let $\bm w_1,\ldots,\bm w_N$ be an orthonormal basis of $\R^N$ such that $\e_c^\trans A\bm w_1\leq 1$ and $\e_c^\trans A\bm w_i=0$ for $i=2,3,\cdots,N$.
We have the following high probability bound:
\begin{align}\label{b5}
\begin{split}
    \left|(A\bar GG)_{cd}\right|=\left| \sum_{i}\bm e_c^\trans A\bm w_i\bm w_i^\trans\bar GG\bm e_d\right|\leq(G\bar G)_{\bm w_1d}\prec N\Psi^2,
\end{split}
\end{align}
where we used \eqref{eqn:2-resolvents} and \eqref{eqn:psi-bound} in the last step.

Using Young's inequality with powers $p=2D$ and $q=(2D)/(2D-1)$, and \eqref{b5}, the last line in \eqref{eqn:multi-resolvent-moment} can be bounded by
\begin{equation}\label{eqn:4th}
\left|\mathbb{E}\left[(A\bar{G} G)_{c d}(G A \bar{G} G)_{c d}^{D-1}(\bar{G} A G \bar{G})_{c d}^D\right]\right| \leqslant N^{2 D \varepsilon} (N\Psi^2)^{2D}+N^{-(2 D \varepsilon)/(2 D-1)} \mathbb{E}\left[\left|(G A \bar{G} G)_{c d}\right|^{2 D}\right]
\end{equation}
for any small $\epsilon>0$, for $N\ge N_0(\epsilon)$. 

Applying Lemma \ref{lem:cumulant-expansion} to the second line in \eqref{eqn:multi-resolvent-moment}, with $T=12D$, we have 
\begin{equation}\label{eqn:cumulant-expansion}
\begin{aligned}
& \mathbb{E} \left[\sum_k h_{c k}(G A \bar{G} G)_{k d}(G A \bar{G} G)_{c d}^{D-1}(\bar{G} A G \bar{G})_{c d}^D\right] \\
& =  \sum_{r=1}^{ 20D } \frac{1}{r ! N^{(r+1) / 2}} \sum_k \kappa_{r+1}^{c,k}\mathbb{E}\left[\partial_{c k}^r(G A \bar{G} G)_{k d}(G A \bar{G} G)_{c d}^{D-1}(\bar{G} A G \bar{G})_{c d}^D\right]+\Omega,
\end{aligned}
\end{equation}
where $\kappa_{r}^{c,k}$ is the $r$-th cumulant of $\sqrt Nh_{ck}$ and $\Omega$ denotes the error term in \eqref{eqn:cumulant-theorem}.

We split the sum in \eqref{eqn:cumulant-expansion} into two cases and handle the error term $\Omega$ at the end.
\begin{enumerate}
\item When $r=1$, the summand is
\begin{align}\label{eqn:r=1}
\begin{split}
    \sum_{k}\frac{1+\delta_{ck}}{N}\mathbb{E}\left[\partial_{c k}(G A \bar{G} G)_{k d}(G A \bar{G} G)_{c d}^{D-1}(\bar{G} A G \bar{G})_{c d}^D\right].
\end{split}
\end{align}
Note that 
\begin{align*}
\begin{split}
    (1+\delta_{ck})\partial_{ck}G_{ij}=-G_{ic}G_{kj}-G_{ik}G_{cj}.
\end{split}
\end{align*}
We have
\begin{align}\label{eqn:integration-by-parts}
\begin{split}
    (1&+\delta_{ck})\partial_{ck}\left[(GA\bar GG)_{kd}\left(GA\bar GG\right)_{cd}^{D-1}\left(\bar GAG\bar G\right)_{cd}^{D}\right]\\
    =&-[G_{kk}(GA\bar GG)_{cd}+G_{kc}(GA\bar GG)_{kd}+(GA\bar G)_{kc}(\bar GG)_{kd}+(GA\bar G)_{kk}(\bar GG)_{cd}\\
    &\qquad+(GA\bar GG)_{kc}G_{kd}+(GA\bar GG)_{kk}G_{cd}]\cdot\left(GA\bar GG\right)_{cd}^{D-1}\left(\bar GAG\bar G\right)_{cd}^{D}\\
    &-(D-1)\cdot[G_{cc}(GA\bar GG)_{kd}+G_{ck}(GA\bar GG)_{cd}+\left(GA\bar G\right)_{cc}\left(\bar GG\right)_{kd}+\left(GA\bar G\right)_{ck}\left(\bar GG\right)_{cd}\\
    &\qquad+\left(GA\bar GG\right)_{cc}G_{kd}+\left(GA\bar GG\right)_{ck}G_{cd}]\cdot\left(GA\bar GG\right)_{kd}\left(GA\bar GG\right)_{cd}^{D-2}\left(\bar GAG\bar G\right)_{cd}^{D}\\
    &-D\cdot[\bar G_{cc}\left(\bar GAG\bar G\right)_{kd}+\bar G_{ck}\left(\bar GAG\bar G\right)_{cd}+\left(\bar GAG\right)_{cc}\left(G\bar G\right)_{kd}+\left(\bar GAG\right)_{ck}\left(G\bar G\right)_{cd}\\
    &\qquad+\left(\bar GAG\bar G\right)_{cc}\bar G_{kd}+\left(\bar GAG\bar G\right)_{ck}\bar G_{cd}]\cdot(GA\bar GG)_{kd}\left(GA\bar GG\right)_{cd}^{D-1}\left(\bar GAG\bar G\right)_{cd}^{D-1}.
\end{split}
\end{align}
Inserting \eqref{eqn:integration-by-parts} into \eqref{eqn:r=1}, we have 
\begin{equation}\label{eqn:r=1organized}
\begin{aligned}
\sum_{k}&\frac{1+\delta_{ck}}{N}\partial_{ck} (GA\bar GG)_{kd}(G A \bar{G} G)_{c d}^{D-1}(\bar{G} A G \bar{G})_{c d}^D \\
= & -m\left|(G A \bar{G} G)_{c d}\right|^{2 D} -\left(\sum_{i=1}^5 \alpha_i\right)(G A \bar{G} G)_{c d}^{D-1}(\bar{G} A G \bar{G})_{c d}^D \\
& -(D-1)\left(\sum_{i=1}^6 \beta_i\right)(G A \bar{G} G)_{c d}^{D-2}(\bar{G} A G \bar{G})_{c d}^D  -D\left(\sum_{i=1}^6 \hat{\beta}_i\right)(G A \bar{G} G)_{c d}^{D-1}(\bar{G} A G \bar{G})_{c d}^{D-1},
\end{aligned}
\end{equation}
where $m = N^{-1} \Tr G$ and 
\begin{equation*}
\begin{gathered}
\alpha_1=\frac{1}{N}(G G A \bar{G} G)_{c d}, \quad\alpha_2=\frac{1}{N}(\bar{G} A G \bar{G} G)_{c d}, \quad\alpha_3=\frac{1}{N}(G \bar{G} A G G)_{c d}, \\
\alpha_4=\frac{1}{N} \operatorname{Tr}(G A \bar{G})(G \bar{G})_{c d}, \quad\alpha_5=\frac{1}{N} G_{c d} \operatorname{Tr}(G A \bar{G} G),
\end{gathered}
\end{equation*}
\begin{equation*}
\begin{gathered}
\beta_1=\frac{1}{N} G_{c c}(G \bar{G} A G G A \bar{G} G)_{d d}, \quad\beta_2=\frac{1}{N}(G A \bar{G} G)_{c d}(G G A \bar{G} G)_{c d}, \\
\beta_3=\frac{1}{N}(G A \bar{G})_{c c}(G \bar{G} G A \bar{G} G)_{d d}, \quad\beta_4=\frac{1}{N}(\bar{G} G)_{c d}(G A \bar{G} G A \bar{G} G)_{c d}, \\
\beta_5=\frac{1}{N}(G A \bar{G} G)_{c c}(G G A \bar{G} G)_{d d}, \quad\beta_6=\frac{1}{N} G_{c d}(G A \bar{G} G G A \bar{G} G)_{c d},
\end{gathered}
\end{equation*}

\begin{equation*}
\begin{gathered}
\hat \beta_1=\frac{1}{N}\bar  G_{c c}(\bar G G A \bar G G A \bar{G} G)_{d d}, \quad \hat \beta_2=\frac{1}{N}(\bar G A G  \bar G)_{c d}(\bar G G A \bar{G} G)_{c d}, \\
\hat \beta_3=\frac{1}{N}(\bar G A G)_{c c}(\bar G G G A \bar{G} G)_{d d}, \quad \hat \beta_4=\frac{1}{N}(G \bar G)_{c d}(\bar G A G G A \bar{G} G)_{c d}, \\
\hat \beta_5=\frac{1}{N}(\bar G A G \bar G)_{c c}(\bar G G A \bar{G} G)_{d d}, \quad \hat \beta_6=\frac{1}{N} \bar G_{c d}(\bar G A G \bar G G A \bar{G} G)_{c d}.
\end{gathered}
\end{equation*}

% Recall the definition for $M$ in Section \ref{sec:multi-resolvent-local-law}.
% For our purpose, we pointed out two properties shared by $M$:
% \begin{enumerate}
% \item When $A$ is traceless,
% \begin{align}\label{eqn:traceless}
% \begin{split}
% M\left(z_1, A, z_2\right) & =A m_{\mathrm{sc}}\left(z_1\right) m_{\mathrm{sc}}\left(z_2\right);
% \end{split}
% \end{align}
% \item When $z_2\neq z_3$,
% \begin{align}\label{eqn:split}
% \begin{split}
%     M\left(z_1,B,z_2,I,z_3\right)=\frac{M\left(z_1,B,z_2\right)-M\left(z_1,B,z_3\right)}{z_2-z_3}.
% \end{split}
% \end{align}
% \end{enumerate}
% Also, it's useful to note that when the number of traceless matrices is odd, the isotropic fluctuation \eqref{eqn:iso} is larger than the size of centering \eqref{eqn:center-iso} and when the number of traceless matrices is even, the size of centering \eqref{eqn:center-iso} is larger than the isotropic fluctuation \eqref{eqn:iso}.

We used the fact that $G$ is symmetric, with $G_{ij} = G_{ji}$ for all $i,j \in \llbracket 1 , N \rrbracket$, in the above calculations. We also used that $A$ is symmetric by assumption. We now  estimate the terms $\alpha_i,\beta_i, \hat \beta_i$. 
\begin{itemize}
\item By \eqref{eqn:center-iso} and \eqref{eqn:iso}, we have 
\begin{align*}
\begin{split}
    |\alpha_1|=\left|\frac{1}{N}(G G A \overline{G} G)_{cd}\right|\prec N^{-3/2}\eta^{-3}\leq N^{3/2}\Psi^3,
\end{split}
\end{align*}
where we used $1/(N\eta)\leq \Psi$ in the last inequality.
\item %By \eqref{eqn:center-iso} and \eqref{eqn:iso}, 
The same computation used to bound $\alpha_1$ also shows that
\begin{align*}
\begin{split}
    |\alpha_2|=\left|\frac{1}{N}(\overline{G} A G \overline{G} G)_{cd}\right|\prec N^{-3/2}\eta^{-3}\leq N^{3/2}\Psi^3.
\end{split}
\end{align*}
\item By the same argument used to bound $\alpha_1$ and $\alpha_2$, we have
\begin{align*}
\begin{split}
    |\alpha_3|\prec N^{-3/2}\eta^{-3}\leq N^{3/2}\Psi^3.
\end{split}
\end{align*}
\item By the definition of $M$ in \eqref{Mdefinition}, and using $\Tr A = 0$, we have 
\begin{align*}
\begin{split}
M(z,A,\overline z) = A|m_{\mathrm{sc}}(z)|^2, \qquad 
    \Tr(M(z,A,\overline z)I)=\Tr(A|m_{\mathrm{sc}}(z)|^2)=0.
\end{split}
\end{align*}
Combining the previous equation with \eqref{eqn:avg}, we have
\begin{align*}
\begin{split}
    \left|\frac{1}{N}\Tr(GA\overline G)\right|\prec N^{-1}\eta^{-3/2}\leq N^{1/2}\Psi^{3/2}.
\end{split}
\end{align*}
Therefore
\begin{align*}
\begin{split}
    |\alpha_4|=\left|\frac{1}{N} \operatorname{Tr}(G A \overline{G})(G \overline{G})_{cd}\right|\prec N^{3/2}\Psi^{7/2},
\end{split}
\end{align*}
where we used $\left(G\overline G\right)_{cd}\prec N\Psi^2$ (from \eqref{eqn:2-resolvents}).
\item By the definition of $M$ in \eqref{Mdefinition} and \cite[Equation~(3.12)]{CipErdSch22optimal}, we have%\footnote{Alternatively, this follows from the first identity in \cite[Equation (2.9)]{CipErdSch22optimal}.}
\begin{align*}
\begin{split}
    \Tr(M(z,A,\overline z,I,z)I)&=\frac{1}{2\eta}\Tr(M(z,A,\overline z)I-M(z,A,z)I)\\
    &=\frac{1}{2\eta}\Tr(A|m_{\mathrm{sc}}(z)|^2-Am_{\mathrm{sc}}(z)^2)=0.
\end{split}
\end{align*}
Combining this equation with \eqref{eqn:avg}, we have
\begin{align*}
\begin{split}
    \left|\frac{1}{N}\Tr(GA\overline GG)\right|\prec N^{-1}\eta^{-5/2}\leq N^{3/2}\Psi^{5/2}.
\end{split}
\end{align*}
Therefore
\begin{align*}
\begin{split}
    |\alpha_5|=\left|\frac{1}{N} G_{cd} \operatorname{Tr}(G A \overline{G} G)\right|\prec N^{3/2}\Psi^{5/2}.
\end{split}
\end{align*}
\item For $\beta_1$, we cannot directly use \eqref{eqn:iso} and \eqref{eqn:center-iso}, since the estimate on operator norm of the deterministic part 
\begin{equation}\label{detpart}
M(z,I,\overline z,A,z,I,z,A,\overline z,I,z)
\end{equation}
provided by \eqref{eqn:center-iso} 
is not small compared to the estimate of the fluctuations in \eqref{eqn:iso}. Instead, we use the inequality 
\[
\left|\left(G\overline GAGGA\overline GG\right)_{dd}\right|\leq \sum_{k=1}^N\left|G_{dk}\right|\left|\left(\overline GAGGA\overline GG\right)_{kd}\right|
\]
and bound the entries of $(\overline GAGGA\overline GG)_{kd}$ by computing $M(\bar z,A,z,I,z,A,\bar z,I,z)$ explicitly.

By \cite[Equation~(3.12)]{CipErdSch22optimal}, we have
\begin{align*}
\begin{split}
M(\bar z,A,z,I,z,A,\bar z,I,z)=\frac{1}{2\eta}\big(M(\bar z,A,z,I,z,A,\bar z)-M(\bar z,A,z,I,z,A, z)\big).
\end{split}
\end{align*}
By definition, we have
\begin{align}\label{M1}
\begin{split}
    &\quad M(\bar z,A,z,I,z,A,\bar z)\\
    &=\frac{1}{N}\Tr{(A^2)}I \left(m_\circ[\bar z,\bar z]m_\circ[z,z]+m_\circ[\bar z,\bar z]m_\circ[z]^2\right)+A^2\left(m_\circ[z]^2m_\circ[\bar z]^2+m_\circ[z,z]m_\circ[\bar z]^2\right)\\
    &=
    \frac{1}{N}\Tr(A^2)I \left(\m[\bar z,\bar z]\m[z,z]-\m(\bar z)^2\m[ z, z]\right)+ A^2\m(\bar z)^2\m[z,z],
\end{split}
\end{align}
and 
\begin{align}\label{M2}
\begin{split}
    &\quad M(\bar z,A,z,I,z,A, z)\\
    &=\frac{1}{N}\Tr(A^2)I\left(m_\circ[\bar z,z]m_\circ[z,z]+m_\circ[\bar z,z]m_\circ[z]m_\circ[z]\right)+A^2\left(m_\circ[\bar z]m_\circ[z]^3+m_\circ[z,z]m_\circ[\bar z]m_\circ[z]\right)\\
    &=\frac{1}{N}\Tr(A^2)I\left(\m[\bar z,z]\m[z,z]-|\m(z)|^2\m[z,z]\right) +A^2 |\m(z)|^2\m[z,z].
\end{split}
\end{align}
By \cite[Equation (A.1)]{CipErdSch22optimal}, there exists a constant $C>0$ such that 
\begin{equation}\label{e:detpart}
\big|\m(z)\big| \le C, \quad \big|\m[\bar z,z]\big| \le C\eta^{-1}, \quad \big|\m[z,\bar z]\big|\le C\eta^{-1}, \quad \big| \m[z,z]\big| \le C\eta^{-1}.
\end{equation}
%which is of order $\eta^{-2}$, are $(\Tr A^2) I$, and the remaining terms are of order at most $\eta^{-1}$. 
Since $\Tr( A^2) I$ is diagonal and $\|A\|\leq 1$, the off-diagonal entries of $\overline GAGGA\overline GG$ are bounded by \[\eta^{-2} + N^{-1/2}\eta^{-7/2}\leq  2 N^{3}\Psi^{7/2}\] with high probability by \eqref{eqn:iso}, where we used $\eta = \eta_\ell \le N^{-1/3}$, 
 \eqref{e:detpart}, and $(A^2)_{ij} \le 1$ (from $\|A^2\| \le 1$)  to neglect the deterministic contribution from \eqref{detpart}. For the diagonal entries, we must account for the $\Tr( A^2) I$ term, and we can estimate these terms by $\eta^{-3} + N^3 \Psi^{7/2}$.  Therefore, by \Cref{thm:iso-single}, \eqref{isotropic}, and \eqref{eqn:psi-bound-specified}, we have
\begin{align*}
\begin{split}
    \left|\left(G\overline GAGGA\overline GG\right)_{dd}\right|\leq \sum_{k=1}^N\left|G_{dk}\right|\left|\left(\overline GAGGA\overline GG\right)_{kd}\right|\prec N^{4}\Psi^{9/2}.
\end{split}
\end{align*}
We conclude that
\begin{align*}
\begin{split}
    |\beta_1|=\left|\frac{1}{N} G_{cc}(G \bar{G} A G G A \bar{G} G)_{dd}\right|\prec N^{3}\Psi^{9/2}.
\end{split}
\end{align*}
\item By \eqref{eqn:center-iso} and \eqref{eqn:iso}, we have
\begin{align*}
\begin{split}
    |\beta_2|=\left|\frac{1}{N}(G A \bar{G} G)_{cd}(G G A \bar{G} G)_{cd}\right|\prec \frac{1}{N}(N^{-1/2}\eta^{-2})(N^{-1/2}\eta^{-3})\leq N^{3}\Psi^{5}.
\end{split}
\end{align*}
\item By \eqref{eqn:center-iso} and \eqref{eqn:iso}, we have
\begin{align*}
\begin{split}
    |\beta_3|=\left|\frac{1}{N}(G A \bar{G})_{cc}(G \bar{G} G A \bar{G} G)_{dd}\right|\prec \frac{1}{N}(N^{-1/2}\eta^{-1})(N^{-1/2}\eta^{-4})\leq N^{3}\Psi^{5}.
\end{split}
\end{align*}
\item By \Cref{thm:iso-single}, \eqref{eqn:center-iso}, and \eqref{eqn:iso}, we have
% and \eqref{eqn:psi-bound-specified},
\begin{align}
\begin{split}
    |\beta_4|=\left|\frac{1}{N}(\bar{G} G)_{cd}(G A \bar{G} GA  \bar{G}  G)_{cd}\right|\prec\frac{1}{N}(N\Psi^2)(\eta^{-3})\leq N^{3}\Psi^{5}.
\end{split}
\end{align}
\item By the same argument  used to bound $\beta_2$, we have
\begin{align*}
\begin{split}
    |\beta_5|=\left|\frac{1}{N}(G A \bar{G} G)_{cc}(G G A \bar{G} G)_{dd}\right|\prec N^{3}\Psi^{5}.
\end{split}
\end{align*}
\item  For $\beta_6$, we follow our approach for $\beta_1$. We have
\begin{align}\label{e:beta6bound}
\left|\left(GA \overline G G GA\overline GG\right)_{dd}\right|\leq \sum_{k=1}^N\left|\left( GA \overline G GGA\overline G\right)_{dk}\right| \left|G_{kd}\right|
\end{align}
We aim to understand the deterministic equivalent $M( z,A,\bar z,I,z, I, z, A,\bar z)$ for $GA \overline G GGA\overline G$. By \cite[Equation~(3.12)]{CipErdSch22optimal},
\[
M( z,A,\bar z,I,z, I, z, A,\bar z)
= \frac{1}{2\eta}
\left(
M( z,A, z, I, z, A,\bar z) - M( z,A,\bar z, I, z, A,\bar z)
\right)
\]
The term $M( z,A, z, I, z, A,\bar z)$ is the transpose of $M( \bar z,A, z, I, z, A, z)$, which was already bounded in \eqref{M2}. Using \cite[Equation~(3.12)]{CipErdSch22optimal} again, we can write 
\[
M( z,A,\bar z, I, z, A,\bar z)
=
\frac{1}{2\eta}\left(
M( z,A, z, A,\bar z)
- M( z,A,\bar z, A,\bar z)
\right).
\]

For any $z_1, z_2,z_3 \in \mathbb{C} \setminus \R$, the identity 
\cite[(2.9)]{CipErdSch22optimal} gives
\begin{equation}\label{e:2traceless}
M(z_1, A, z_2, A , z_3) = \frac{1}{N}\Tr{(A^2)}\big( \m[z_1, z_3] - \m(z_1) m(z_3) \big) \m(z_2) + A^2 \m(z_1) \m(z_2) \m(z_3).
\end{equation}
Using \eqref{e:detpart}, we can find the same bounds for the diagonal entries and off-diagonal entries of $M( z,A,\bar z,I,z, I, z, A,\bar z)$ that we found for the analogous deterministic equivalent in the analysis of $\beta_1$. 
%By the same argument used to bound $\beta_1$, 
Then using \eqref{e:beta6bound}, we conclude that
\begin{align*}
\begin{split}
    |\beta_6|=\left|\frac{1}{N} G_{cd}(G A \bar{G} G G A \bar{G} G)_{cd}\right|\prec N^{3}\Psi^{9/2}.
\end{split}
\end{align*}
\item  We now the consider the $\hat \beta_i$ terms. For $i$ such that $2\le i \le 5$, $\hat \beta_i$ may be bounded analogously to $\beta_i$ by directly applying \eqref{eqn:center-iso} and \eqref{eqn:iso}. The terms $\hat\beta_1$ and $\hat\beta_6$ are bounded similarly to $\beta_1$ and $\beta_6$. The only substantive changes required come in the analysis of the deterministic equivalents. 

For $\hat \beta_1$, 
\[
\left|\left(\overline G GA \overline GGA\overline GG\right)_{dd}\right|\leq \sum_{k=1}^N\left|\overline G_{dk}\right|\left|\left( GA\overline GGA\overline GG\right)_{kd}\right|,
\]
and the deterministic equivalent of
$ GA\overline GGA\overline GG$ is 
\[
M(z, A, \bar z, I, z, A, \bar z, I ,z) 
= \frac{1}{2\eta} 
\big(
M(z, A , \bar z , I ,z , A , z) 
- M( z ,A,  \bar z, I ,z , A ,\bar z)
\big)
\]
We have
\begin{equation}\label{e:boundedb4}
M(z, A , \bar z , I ,z , A , z) 
= 
\frac{1}{2\eta} 
\big(
M(z, A  ,z , A , z) 
-M(z, A , \bar z  , A , z) 
\big) ,
\end{equation}
and similarly for $M( z ,A,  \bar z, I ,z , A ,\bar z)$. The resulting deterministic equivalents can then be bounded using \eqref{e:2traceless}.

For $\hat \beta_6$, we use
\begin{align}
\left|\left(\overline GA  G \overline G GA\overline GG\right)_{dd}\right|\leq \sum_{k=1}^N\left|\left( \overline GA  G \overline GGA\overline G\right)_{dk}\right| \left|G_{kd}\right|.
\end{align}
The deterministic equivalent of $\overline GA  G \overline GGA\overline G$ is 
\[
M(\bar z , A , z, I, \bar z, I, z, A ,\bar z)
= \frac{1}{2\eta} 
\big(
M(\bar z , A , z, I,  z, A ,\bar z)
-M(\bar z , A , z, I, \bar z, A ,\bar z)
\big).
\]
The first term on the right-hand side was bounded in \eqref{M1}, and the second is the conjugate of $M(z, A , \bar z , I ,z , A , z) $, which was bounded in \eqref{e:boundedb4}. 
\end{itemize}
Let $\epsilon>0$ be a parameter. In \eqref{eqn:r=1organized}, we use Young's inequality twice in the second inequality, with powers $p=2D$ and $q=(2D)/(2D-1)$ for terms containing $\alpha_i$'s and powers $p=D$ and $q=D/(D-1)$ for terms containing $\beta_i$'s, to show that 
\begin{equation}\label{eqn:r=1final}
\begin{aligned}
&\left|\E\left[m |(GA\bar GG)_{cd}|^{2D}+
\sum_k \frac{1+\delta_{c k}}{N} \partial_{c k}(G A \bar{G} G)_{k d}(G A \bar{G} G)_{c d}^{D-1}(\bar{G} A G \bar{G})_{c d}^D\right]\right|\\
&\leqslant  \sum_{i=1}^5 \mathbb{E}\left[\left|\alpha_i\right|\left|(G A \bar{G} G)_{c d}\right|^{2 D-1}\right]+ \sum_{i=1}^6 \mathbb{E}\left[\left|\beta_i\right|\left|(G A \bar{G} G)_{c d}\right|^{2 D-2}\right] 
+  \sum_{i=1}^6 \mathbb{E}\left[\left|\hat \beta_i\right|\left|(G A \bar{G} G)_{c d}\right|^{2 D-2}\right]  \\
&\leqslant  \sum_{i=1}^5 N^{2 D \varepsilon}\E\left[\left|\alpha_i\right|^{2 D}\right]+ \sum_{i=1}^6 N^{D \varepsilon}\E\left[\left|\beta_i\right|^D\right]
+  \sum_{i=1}^6 N^{D \varepsilon}\E\left[|\hat \beta_i|^D\right]  \\
& \quad +\left(5 N^{-(2 D \varepsilon)/ (2 D-1)}+12 N^{-(D \varepsilon)/(D-1)}\right) \mathbb{E}\left[\left|(G A \bar{G} G)_{c d}\right|^{2 D}\right]\\
&\prec \, 20N^{2D\epsilon}N^{3D}\Psi^{9D/2}+20N^{-(2D\epsilon)/(2D-1)}\mathbb{E}\left[\left|(G A \bar{G} G)_{c d}\right|^{2 D}\right].
\end{aligned}
\end{equation}
\item Now we fix $r\geq 2$ in the cumulant expansion. In this case, we use the following relaxation of \eqref{eqn:center-iso} and \eqref{eqn:iso}. Since it is a direct consequence of these inequalities, we omit the proof.
\begin{corollary}[``Coarser'' isotropic local law]\label{thm:coarse}
Fix $k \in \mathbb{N}$,  and $z_1, \ldots, z_{k+1} \in \bm S$. Let $A_1, \ldots, A_k$ be deterministic matrices such that $\left\|A_j\right\| \le 1$, and $m$ of them satisfy  $\Tr A_j =0$ for some $0 \leqslant m \leqslant k$. %Let $\eta:=\min _j\left|\Im z_j\right|$ and $d:=\min _j \operatorname{dist}\left(z_j,[-2,2]\right)$
Suppose that 
$\min _j \operatorname{dist}\left(z_j,[-2,2]\right) \leqslant 1$.
Then for arbitrary deterministic vectors $\boldsymbol{x}, \boldsymbol{y}$ such that  $\|\boldsymbol{x}\|+\|\boldsymbol{y}\| \le 2$, we have %the optimal 
%isotropic local law
\begin{equation}\label{eqn:coarse}
\left|\left\langle\boldsymbol{x},\left(G_1 A_1 \cdots G_k A_k G_{k+1}\right) \boldsymbol{y}\right\rangle\right| \prec N^{k-m/2}
\end{equation}
with $G_j:=G\left(z_j\right)$.
\end{corollary}
We call a term of the form $\left(G_1 B_1 \cdots G_{s+1}\right)_{i j}$ a block. The effect of one differentiation operator $\partial_{a k}$ on a product of blocks is to increase the number of blocks and the number of $G$ 's by exactly one each, while keeping the number of $A$ factors unchanged. And from \eqref{eqn:coarse}, we know the effect of each traceless matrix $A$ is a contribution of a factor of $N^{-1 / 2}$.

\begin{itemize}
\item Suppose $r<2D-1$. Note that when using the product rule, $\partial_{ck}^r$ can at most operate on $r$ different blocks, there are at least $2D-r-1$ blocks of $(GA\bar GG)_{cd}$ or $(\bar GAG\bar G)_{cd}$ which are unaffected. In view of this, we have
\begin{align}\label{eqn:partial-bound-1}
\begin{split}
    &\left|\partial_{ck}^r\left[\left(GA\bar GG\right)_{kd} (G A \bar{G} G)_{c d}^{D-1}(\bar{G} A G \bar{G})_{c d}^D\right]\right|\\
    \leq&\,\sum_{\substack{r_1+r_2=r\\r_1\leq D-1\\r_2\leq D}}\binom{D}{r_1}\binom{D-1}{r_2}\left|\partial_{ck}^r\left[\left(GA\bar GG\right)_{kd} (G A \bar{G} G)_{c d}^{r_1}(\bar{G} A G \bar{G})_{c d}^{r_2}\right]\right|\left|\left(GA\bar GG\right)_{cd}\right|^{2D-1-r}\\
    \leq &\,D^r\sum_{\substack{r_1+r_2=r\\r_1\leq D-1\\r_2\leq D}}\left|\partial_{ck}^r\left[\left(GA\bar GG\right)_{kd} (G A \bar{G} G)_{c d}^{r_1}(\bar{G} A G \bar{G})_{c d}^{r_2}\right]\right|\cdot\left|\left(GA\bar GG\right)_{cd}\right|^{2D-1-r}.
\end{split}
\end{align}
There are $1+r$ blocks, $3(r+1)$ $G$'s, and $1+r$ $A$'s in
\begin{align*}
 \begin{split}
     \left(GA\bar GG\right)_{kd} (G A \bar{G} G)_{c d}^{r_1}(\bar{G} A G \bar{G})_{c d}^{r_2}.
 \end{split}
 \end{align*}
and after the operation of $\partial_{ck}^r$, there are $2r+1$ blocks, $4r+3$ $G$'s and $1+r$ $A$'s. By Corollary \ref{thm:coarse}, we know
\begin{align}\label{eqn:partial-bound-2}
\begin{split}
    \left|\partial_{ck}^r\left(GA\bar GG\right)_{kd} (G A \bar{G} G)_{c d}^{r_1}(\bar{G} A G \bar{G})_{c d}^{r_2}\right|\prec N^{(4r+2)-(2r)-(1+r)/2 }=N^{3(r+1)/2}.
\end{split}
\end{align}
Combining \eqref{eqn:partial-bound-1} and \eqref{eqn:partial-bound-2}, we deduce that a summand in \eqref{eqn:cumulant-expansion} with $2\leq r<2D-1$ can be bounded by
\begin{align}\label{eqn:r>1final1}
\begin{split}
    &\left|\frac{\kappa_{r+1}}{r ! N^{(r+1) / 2}} \sum_k \mathbb{E}\left[\partial_{c k}^r(G A \bar{G} G)_{k d}(G A \bar{G} G)_{c d}^{D-1}(\bar{G} A G \bar{G})_{c d}^D\right]\right|\\
    &\prec \,D^{r+1}N^{r+2}\E\left[\left|\left(GA\bar GG\right)_{cd}\right|^{2D-1-r}\right]\\
    &\leq \, D^{r+1}N^{(2D\epsilon)/(r+1)+2D(r+2)/(r+1)}+D^{r+1}N^{-(2D\epsilon)/(2D-r-1)}\E\left[\left|\left(GA\bar GG\right)_{cd}\right|^{2D}\right],
\end{split}
\end{align}
for every $\epsilon>0$, where we used Young's inequality in the last step with $p = 2D/( 2D - 1 - r)$ and $q = 2D/(1+r)$. 
\item Suppose $r\geq 2D-1$. There are initially $2D$ blocks, $6D$ $G$'s and $2D$ $A$'s in
\begin{align*}
\begin{split}
    (G A \bar{G} G)_{k d}(G A \bar{G} G)_{c d}^{D-1}(\bar{G} A G \bar{G})_{c d}^D,
\end{split}
\end{align*}
and after the operation of $\partial_{ck}^r$,  there are many terms, each with $2D+r$ blocks, $6D+r$ $G$'s and $2D$ $A$'s. The number of terms depends on $D$. Then by Corollary \ref{thm:coarse}, we have
\begin{align}\label{eqn:partial-bound-3}
\begin{split}
    \left|\partial_{ck}^r(G A \bar{G} G)_{k d}(G A \bar{G} G)_{c d}^{D-1}(\bar{G} A G \bar{G})_{c d}^D\right|\prec C(D) N^{3D}.
\end{split}
\end{align}
Therefore, a summand in \eqref{eqn:cumulant-expansion} with $r\geq 2D-1$ can be bounded by
\begin{align}\label{eqn:r>1final2}
\begin{split}
    \left|\frac{\kappa_{r+1}}{r ! N^{(r+1) / 2}} \sum_k \mathbb{E}\left[\partial_{c k}^r(G A \bar{G} G)_{k d}(G A \bar{G} G)_{c d}^{D-1}(\bar{G} A G \bar{G})_{c d}^D\right]\right| &\prec C(D)N^{3D-(r+1)/2+1}\\&\leq C(D)N^{D+1}.
\end{split}
\end{align}
\end{itemize}
\item Finally, we consider the error term $\Omega$ in \eqref{eqn:cumulant-expansion}. 
Define $H_t^{(k)}$ by 
\begin{align*}
    \left(H_t^{(k)}\right)_{ij}=\begin{cases}
        t,&\text{ if }i=c,j=k,\\
        h_{ij},&\text{ otherwise.}
    \end{cases}
\end{align*}
Let $G_t^{(k)}$ denote the resolvent of $H_t^{(k)}$. 
Fix a parameter $\zeta >0$ and set $Q = N^{-1/2 + \zeta}$. By choosing $\zeta$ small enough, in a way that depends only on $\tau$, a resolvent expansion similar to \eqref{e:error-resolvent} shows that the first claim in \eqref{eqn:2-resolvents} also holds for $G_t^{(k)}$,
for every $i,j \in \llbracket 1, N \rrbracket$, uniformly in the choice of $k \in \llbracket 1, N \rrbracket$ and $|t| \le Q$. In particular, we have
\begin{equation}\label{e:perturbedll}
\sup_{i,j\le N}  \sup_{|t| \le Q} \big|(G^{(k)}_t)_{ij}\big| \prec 1.
\end{equation}

By Lemma \ref{lem:cumulant-expansion}, 
\begin{equation}\label{eqn:error-term}
    \Omega\leq C_{D}\sum_k  K(k),
\end{equation}
where
\[
K(k) = 
\E\left[|h_{ck}|^{20D+2} \right]\sup_{|t| \le Q } 
J(t) 
 + \E\left[|h_{ck}|^{20D+2}  \one\{|h_{ck}| > Q \}\right]\sup_{t \in \R } J(t) 
\]
and 
\[
J(t) 
= \E\left[\left|\partial_{ck}^{20D+1}(G_t^{(k)}A\bar G_t^{(k)}G_t^{(k)})_{kd}(G_t^{(k)}A\bar G_t^{(k)}G_t^{(k)})_{cd}^{D-1}(\bar G_t^{(k)}AG_t^{(k)}\bar G_t^{(k)})_{cd}^D \right|\right].
\]
By the trivial bound $\|G_t^{(k)}\|\leq \eta^{-1}\leq N$ from \eqref{eqn:absolute-bound-G}, $|A_{ij}| \le 1$, and counting the number of terms generated by taking derivatives in the definition of $J(t)$, we have (as a crude bound) 
\begin{equation}\label{e:edit02}
\sup_{t \in \R } J(t) \le C(D) N^{80(D+1)}. 
\end{equation}
Further, by \eqref{finite-moment} and Markov's inequality, we have for every $M >0$ that 
\begin{equation}\label{e:edit1}
\P\big( |N^{1/2} h_{ck}| > N^\zeta  \big) \le \mu_M N^{- M \zeta}.
\end{equation}\label{e:edit12}
By taking $M$ sufficiently large, in a way that depends on $\zeta$, we can enforce that $N^{-M\zeta} \le N^{-160(D+1)}$. Then together with the Cauchy--Schwartz inequality, \eqref{e:edit02} and \eqref{e:edit12} imply that the second term in the definition of $K(k)$ is negligible. 

Next, for the first term in $K(k)$, we note that for any indices $a,b$, we have 
\begin{equation}\label{e:summation}
(G_t^{(k)}A\bar G_t^{(k)}G_t^{(k)})_{ab}
= \sum_{i,j,k =1}^N 
(G_t^{(k)})_{ai} A_{ij} \bar (G_t^{(k)})_{jl} (G_t^{(k)})_{lb}.
\end{equation}
Then by \eqref{e:perturbedll}, the recalling that $|A_{ij}| \le 1$, 
\[
\big|
(G_t^{(k)}A\bar G_t^{(k)}G_t^{(k)})_{ab}
\big|  \le C^3.
\]
When $\partial_{ck}$ acts on some $(G_t^{(k)}A\bar G_t^{(k)}G_t^{(k)})_{ab}$, it increases the number of entries of $G_t^{(k)}$ in the summation, but it does not add a new summation index. We conclude that 
\[
\sup_{|t| \le Q } 
J(t)  \le C(D) N^{6D}.
\]
Using $|h_{ck}|^{20D+2} \prec N^{-10D -1}$, we find that $K(k) \le C(D) N^{-2D-1}$, and hence
\begin{equation}\label{eqn:error-bound}
   \Omega \le C(D) N^{-2D}.
\end{equation}
\end{enumerate}

Combining \eqref{eqn:multi-resolvent-moment}, \eqref{eqn:4th}, \eqref{eqn:cumulant-expansion}, \eqref{eqn:r=1final}, \eqref{eqn:r>1final1}, \eqref{eqn:r>1final2} and \eqref{eqn:error-bound}, we have
\begin{align*}
&\left|\E\left[(z+m)|(GA\bar GG)_{cd}|^{2D}\right]\right|\\
&\prec \, C(D)N^{2D\epsilon}N^{3D}\Psi^{9D/2}+C(D)N^{-(2D\epsilon)/(2D-1)}\mathbb{E}\left[\left|(G A \bar{G} G)_{c d}\right|^{2 D}\right].
\end{align*}
Using the local law $|m-\m|\prec\Psi$ \cite[Theorem 2.6]{benaych2016lectures}, we have
\begin{equation*}
\begin{aligned}
    &(z+\m) \left|\E\left[|(GA\bar GG)_{cd}|^{2D}\right]\right|\\
&\prec \, C(D)N^{2D\epsilon}N^{3D}\Psi^{9D/2}+C(D)N^{-(2D\epsilon)/(2D-1)}\mathbb{E}\left[\left|(G A \bar{G} G)_{c d}\right|^{2 D}\right]+\Psi\cdot \E\left[|(GA\bar GG)_{cd}|^{2D}\right]
\end{aligned}
\end{equation*}
Since $z+\m=1/\m$ (recall \eqref{definingequation}), which is bounded away from $0$ (by \cite[Equation (3.2)]{benaych2016lectures}), and $\epsilon>0$ is arbitrary, we have
\begin{align*}
\begin{split}
\mathbb{E}\left[\left|(G A \bar{G} G)_{c d}\right|^{2 D}\right]\prec N^{3D}\Psi^{9D/2}.
\end{split}
\end{align*}
Since $D>1$ is arbitrary, we have by Markov's inequality that 
\begin{align*}
\begin{split}
|(GA\bar GG)_{cd}|\prec  N^{3/2}\Psi^{9/4}.
\end{split}
\end{align*}
This completes the proof of \eqref{eqn:3-resolvents-traceless}.
\end{proof}

\section{Joint Distribution of Eigenvectors} \label{a:joint}
In this section, we briefly explain how to generalize our univariate result to the following multivariate one.

\begin{theorem}\label{theo:multivariable}
    Let $H$ be a Wigner matrix and fix $\tau\in(0,1)$ and $k\in\mathbb{N}$. Then there exists $\delta=\delta(\tau)\in(0,1)$ such that the following holds. Let $A_1,\dots, A_k\in\R^{N\times N}$ be deterministic sequences of traceless matrices such that $A_i=A_i^*$, $\Vert A_i\Vert\leqslant 1$, and $\mathrm{Tr}(A_i^2)\geqslant N^{1-\delta}$. Let $\ell_1,\dots,\ell_k\in\llbracket 1,N^{1-\tau}\rrbracket \cup \llbracket N-N^{1-\tau},N\rrbracket$ be deterministic sequences of indices and let $\bm{u}_{\ell_1},\dots, \bm{u}_{\ell_k}$ be the corresponding sequences of $\ell^2$-normalized eigenvectors of $H$. Then 
    \[
    \left(\sqrt{\frac{\beta N^2}{2\mathrm{Tr}(A_1^2)}}\langle \bm{u}_{\ell_1},A_1 \bm{u}_{\ell_1}\rangle,
    \dots,
    \sqrt{\frac{\beta N^2}{2\mathrm{Tr}(A_k^2)}}\langle \bm{u}_{\ell_1},A_k \bm{u}_{\ell_k}\rangle
    \right)
    \rightarrow (\mathcal{N}_1,\dots,\mathcal{N}_k)
    \]
    where $\mathcal{N}_i$ is a family of i.i.d standard Gaussian random variables and the convergence is in distribution. We take $\beta=1$ if $H$ is real symmetric and $\beta =2 $ if it is complex Hermitian.
\end{theorem}
The overall proof is the same, as we regularize each observable in the same way but generalize each lemma to a multivariate version. The difference between this generalization and the initial proof is merely notational. Using the same notation as in \eqref{eq:notpl}, we have the following analogue of Lemma~\ref{lem:regularized-observable}, which compares our multidimensional observable to a regularized version.

\begin{lemma}
    Let $H$ be a Wigner matrix, and let the parameters $\epsilon_1>0$ and $\delta_1,\ldots,\delta_5$ be chosen as in Definition \ref{para}.
    Let $g:\R^k\rightarrow\R$ be a compactly supported smooth function.
    Then there exists a constant $c(\tau,g,k)>0$ such that 
    \begin{align*}
    \begin{split}
        \left| \E \Big[g\big(\widehat{p}_{\ell_1}(A_1),\dots,\widehat{p}_{\ell_k}(A_k)\big)\Big]-\E \big[g\big(v_{\ell_1}(\bm \delta, A_1),\dots,v_{\ell_k}(\bm \delta, A_k)\big) \Big] \right|\le c^{-1} N^{-c}.
    \end{split}
    \end{align*}
    \end{lemma}
It is important to note that our choice of parameters $\bm{\delta}$ is independent of the matrices $A_i$ and the indices $\ell_i$. Indeed, they only depend on $\varepsilon_0$ from Proposition \ref{prop:rigidity}, which is uniform in the indices of the eigenvalues, and $\tau$. 

The proof of Theorem \ref{theo:multivariable} then finishes by the same resolvent comparison argument from Subsection \ref{subsec:resolventcomparison} by considering  functions $g$ of $k$ variables.

\bibliographystyle{amsplain}
\bibliography{project_QUE.bib}

\end{document}